\documentclass[11pt, oneside]{amsart}
\usepackage[margin=1.5in]{geometry}
\usepackage{cite}
\usepackage[utf8]{inputenc}
\usepackage[T1]{fontenc}
\usepackage[english]{babel}
\usepackage{amsmath,amsfonts,amsthm, mathrsfs, amssymb}
\usepackage[linktocpage=true]{hyperref}
\usepackage{tikz-cd}

\theoremstyle{plain}
\newtheorem{theorem}{Theorem}[section]
\newtheorem{corollary}[theorem]{Corollary}
\newtheorem{lemma}[theorem]{Lemma}
\newtheorem{proposition}[theorem]{Proposition}

\theoremstyle{definition}
\newtheorem{definition}[theorem]{Definition}
\newtheorem{assumption}[theorem]{Assumption}

\newtheorem{remark}[theorem]{Remark}
\numberwithin{equation}{section}
\numberwithin{figure}{section}
\numberwithin{table}{section}

\usetikzlibrary{shapes.multipart}
\usetikzlibrary{patterns}
\usetikzlibrary{shapes.multipart}
\usetikzlibrary{arrows}
\usetikzlibrary{decorations.markings}
\usepgflibrary{decorations.shapes}
\usetikzlibrary{decorations.shapes}
\usepgflibrary{shapes.symbols}
\usetikzlibrary{shapes.symbols}
\usetikzlibrary{decorations.pathreplacing}

\tikzstyle{fleche}=[>=stealth', postaction={decorate}, thick]
\tikzstyle{axis}=[->, >=stealth', thick, gray]
\tikzstyle{path}=[->, >=stealth', thick]
\tikzstyle{grille}=[dotted, gray]
\colorlet{lgray}{white!85!black}
\colorlet{lred}{white!65!red}

\DeclareMathOperator{\E}{\mathbb E}
\DeclareMathOperator{\PP}{\mathbb P}

\DeclareMathOperator{\fix}{fix}
\DeclareMathOperator{\wt}{wt}
\DeclareMathOperator{\Pf}{pf}

\DeclareMathOperator{\Thetad}{Theta}
\DeclareMathOperator{\sgn}{sgn}

\newcommand{\R}[0]{\mathbb R}

\newcommand{\Z}[0]{\mathbb Z}
\newcommand{\C}[0]{\mathbb C}
\newcommand{\N}[0]{\mathbb N}

\newcommand{\deq}[0]{\overset{(d)}{=}}
\title[Half space current fluctuations]{Boundary current fluctuations for the half space ASEP and six vertex model}

\author{Jimmy He}
\address{Department of Mathematics, MIT, Cambridge, MA  02139}
\email{jimmyhe@mit.edu}
\date{\today}

\begin{document}
\maketitle
\begin{abstract}
We study fluctuations of the current at the boundary for the half space asymmetric simple exclusion process (ASEP) and the height function of the half space six vertex model at the boundary at large times. We establish a phase transition depending on the effective density of particles at the boundary, with GSE and GOE limits as well as the Baik--Rains crossover distribution near the critical point. This was previously known for half space last passage percolation, and recently established for the half space log-gamma polymer and KPZ equation in the groundbreaking work of Imamura, Mucciconi, and Sasamoto \cite{IMS22}. 

The proof uses the underlying algebraic structure of these models in a crucial way to obtain exact formulas. In particular, we show a relationship between the half space six vertex model and a half space Hall--Littlewood measure with two boundary parameters, which is then matched to a free boundary Schur process via a new identity of symmetric functions. Fredholm Pfaffian formulas are established for the half space ASEP and six vertex model, indicating a hidden free fermionic structure.
\end{abstract}
\setcounter{tocdepth}{1}
\tableofcontents

\section{Introduction}
Tremendous progress has been made in recent years towards our understanding of universal fluctuations for models in the Kardar--Parisi--Zhang (KPZ) universality class, including particle systems, polymers and the KPZ equation. However, the introduction of a boundary to these models introduces many difficulties, and our understanding of the asymptotic fluctuations in these models remains incomplete.

While one point fluctuations for a wide variety of models without a boundary have been established, see e.g. \cite{TW09, BCG16, A18}, and indeed there are now classes of models for which we even have multipoint and process convergence to the limiting universal objects \cite{QS23, W23, V20, DOV22, MQR21}, the analysis of systems with boundary is usually much more complicated. Models with a single boundary, which we will call half space models, are expected to exhibit universal behavior with a phase diagram depending on the boundary strength. Some progress towards this has been made in zero-temperature models \cite{BR01b, SI04,BBCS18}, hinting that similar phase transitions should hold for all half space models in the KPZ universality class, but progress in positive temperature models has been more difficult.

This problem was first considered by Kardar \cite{K85}, and has been studied in the physics literature using a replica Bethe Ansatz approach \cite{K87, GP12, DNKLDT20, KLD20}, culminating in work of De Nardis, Krajenbrink, Le Doussal, and Thiery \cite{DNKLDT20}, where a phase transition for the one point distribution at the boundary of the half space KPZ equation was found. In the physics work of Borodin, Bufetov, and Corwin \cite{BBC16}, nested contour integral formulas were found for certain moments, some of which were then proven rigorously in work of Barraquand, Borodin, and Corwin \cite{BBC20}, but these still did not lead to mathematically rigorous asymptotics. For a long time, the only mathematically rigorous result on this phase transition for positive temperature models was work of Barraquand, Borodin, Corwin, and Wheeler \cite{BBCW18} on the half space asymmetric simple exclusion process (ASEP) and KPZ equation, and this was restricted to the critical boundary strength.

Recently, the groundbreaking work of Imamura, Mucciconi, and Sasamoto \cite{IMS22} established a phase transition for the one point distribution at the boundary for the free energy of the half space log-gamma polymer and the solution to the half space KPZ equation. This was the first work which rigorously established this phase transition for a positive temperature model. They found a connection to a free fermionic model via a bijective approach, resulting in formulas amenable to asymptotics. However, this left the problem of studying a variety of other discrete models including the ASEP, which contains an additional boundary parameter.

Beyond just the phase transition studied in this paper, there has been tremendous progress in understanding half space models at positive temperature and their algebraic structure, see e.g. \cite{TW13,TW13b,OSZ14,BR22,H22, BC22, BW22}. For some additional recent works on half space models, see \cite{BFO20, CS18,P22, P19, W20, BBS23}, and also relevant work in the physics literature \cite{BD21, BKD22, BBC16}.

\subsection{Overview of results and proof}
In this paper, we study the ASEP on the non-negative integers, a system of particles which can jump to unoccupied sites, and with particles allowed to enter and exit at $0$. Initially, the system starts with no particles. From known results for zero temperature models (see e.g. \cite{BR01b,SI04,BBCS18}), it is expected that the fluctuations for the total number of particles in the system, which we call the current, has a phase transition depending on the effective density $\rho$ of particles at the boundary, which depends on the rates at which particles enter and exit. In particular, Tracy--Widom GSE and GOE fluctuations are expected when $\rho>\frac{1}{2}$ and $\rho=\frac{1}{2}$ respectively, with $\tau^{1/3}$ scaling at time $\tau$, and Gaussian fluctuations with $\tau^{1/2}$ scaling are expected when $\rho<\frac{1}{2}$. However, for the ASEP, results are currently only known at the critical point $\rho=\frac{1}{2}$, and only for a specific choice of parameters, by work of Barraquand, Borodin, Corwin, and Wheeler \cite{BBCW18}. 

Our main result is to establish these asymptotics as well as the appearance of the Baik--Rains crossover distribution introduced in \cite{BR01b} if the boundary strength is tuned appropriately, subject only to a condition that the direction of the asymmetry at the boundary matches that of the bulk. This restriction only affects the Gaussian regime, and still allows for all possible densities at the boundary. We expect that this condition can be removed, but the methods in this paper do not seem to be able to handle this. We are also able to study the half space stochastic six vertex model, and we establish the same phase transition for the fluctuations of the height function at the boundary.

In order to obtain asymptotics, we find exact Fredholm Pfaffian formulas for the distribution functions after a random shift. This suggests that there is some hidden free fermionic structure responsible for these formulas, even though the systems themselves are not free fermionic. The proof for these formulas involves relating the half space stochastic six vertex model to the half space Hall--Littlewood measure (Theorem \ref{thm: general 6vm HL distr equality}), a probability distribution on partitions, using a Yang--Baxter argument taking advantage of integrability of the model. In particular, we show that the half space six vertex model with two boundary parameters falls into the class of half space Macdonald processes, something not known before this work. This half space Hall--Littlewood measure is then related to the free boundary Schur process introduced by Betea, Bouttier, Nejjar, and Vuleti\'c \cite{BBNV18,BBNV20} (Theorem \ref{thm: HL FBS same}) via an identity of symmetric functions (Theorem \ref{thm: general identity}) which we establish using a bijection of Imamura, Mucciconi, and Sasamoto \cite{IMS21}. Exact contour integral formulas are known for the free boundary Schur process, and we are able to extract asymptotics via a steepest descent analysis. Since the ASEP can be obtained as a limit of the six vertex model, we also obtain exact formulas for it, from which we are able to extract asymptotics.

Let us remark that currently, results in the literature on asymptotics in half space models only consider models with a single boundary parameter, whereas our models have two boundary parameters. We believe that this will allow for potentially more applications, especially as the six vertex model is known to give many other stochastic models via various fusion and limiting procedures.

\subsection{ASEP}
\begin{figure}
    \centering
    \includegraphics{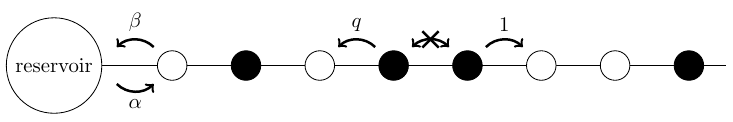}
    \caption{Configuration of the ASEP along with some possible transitions and their rates. Particles which would jump to an occupied site are blocked. Note that since the site at $0$ is empty, particles enter at rate $\alpha$. If it were occupied, the particle would instead leave at rate $\beta$.}
    \label{fig:asep}
\end{figure}
Our main result is a phase transition for the asymptotic current in the half space ASEP, which we now define. The \emph{half space asymmetric simple exclusion process (ASEP)} is a continuous time Markov chain on configurations of particles on $\N$. Initially, the system begins with no particles. At $0$, particles enter or exit at rates $\alpha$ and $\beta$ respectively (we can think of a reservoir connected to $0$ from which particles can enter or exit), and in the bulk they jump left and right at rates $q<1$ and $1$ respectively, except that if a particle would jump to a site occupied by another particle, this is blocked. See Figure \ref{fig:asep} for an example of a configuration and the possible transitions. Note there is no loss of generality in these restrictions, since rescaling the rates corresponds to rescaling time, and by reversing the roles of particles and holes, we swap $1$ and $q$.

A key quantity associated to the system is the \emph{effective density} of particles at $0$,
\begin{equation*}
    \rho=\frac{1}{1+\nu^{-1}},
\end{equation*}
where
\begin{equation*}
    \nu^{-1}=\frac{1-q+\beta-\alpha+\sqrt{(1-q+\beta-\alpha)^2+4\alpha\beta}}{2\alpha}.
\end{equation*}
This quantity arises as the density for which the product Bernoulli measure is stationary for the ASEP dynamics, and indeed $\rho$ solves $\alpha (1-\rho)+q\rho(1-\rho)-\beta \rho-\rho(1-\rho)=0$. We will often write formulas in terms of $\nu$, which simplifies certain expressions. We will let $t=\beta/\alpha$, and note that $\nu t<1$. We will be interested in studying the number of particles within the system at time $\tau$, which we will denote by $N(\tau)$. 

We will make the assumption that $q,t<1$. The assumption that $q<1$ is needed so particles will drift into the system, but the requirement that $t<1$ is a technical one that should not be needed. Since $\nu t<1$ always holds, this restriction is only relevant in the Gaussian regime when $\nu<1$. In particular, we do not assume Liggett's condition that $\alpha+\beta/q=1$ \cite{L75}. Previously asymptotics were only known in the special case that $\alpha=\frac{1}{2}$ and $\beta=\frac{q}{2}$ \cite{BBCW18}.

For the definitions of the Tracy--Widom distributions $F_{GOE}$ and $F_{GSE}$, and the Baik--Rains crossover distribution $F_{cross}(s,\xi)$, see Definition \ref{def: BR TW}.
\begin{theorem}
\label{thm: main}
Let $N(\tau)$ denote the number of particles within the system at time $\tau$ in the half space ASEP started from the empty configuration, and let $\rho$ denote the effective density of particles at $0$. Assume that $t<1$. Then depending on $\rho$, as $\tau\to\infty$,
\begin{equation*}\tag{$\rho>\frac{1}{2}$}
    \PP\left(-\frac{N\left(\frac{\tau}{1-q}\right)-\frac{\tau}{4}}{2^{-4/3} \tau^{1/3}}\leq s\right)\to F_{GSE}(s),
\end{equation*}
\begin{equation*}\tag{$\rho=\frac{1}{2}$}
    \PP\left(-\frac{N\left(\frac{\tau}{1-q}\right)-\frac{\tau}{4}}{2^{-4/3} \tau^{1/3}}\leq s\right)\to F_{GOE}(s),
\end{equation*}
\begin{equation*}\tag{$\rho<\frac{1}{2}$}
    \PP\left(-\frac{N\left(\frac{\tau}{1-q}\right)-\mu \tau}{\sigma \tau^{1/2}}\leq s\right)\to \Phi(s),
\end{equation*}
where $F_{GSE}$ and $F_{GOE}$ are the Tracy--Widom GSE and GOE distribution functions, $\Phi$ is the distribution function for a standard Gaussian, and
\begin{equation*}
    \mu=\frac{\nu}{(1+\nu)^2}, \qquad \sigma^2=\nu^{-2}\frac{1-\nu}{(1+\nu^{-1})^3}.
\end{equation*}
Furthermore, if $\rho=\frac{1}{2}+\frac{2^{-2/3}\xi}{\tau^{1/3}}$ and $t$ is fixed, then
\begin{equation*}\tag{$\rho\downarrow \frac{1}{2}$}
    \PP\left(-\frac{N\left(\frac{\tau}{1-q}\right)-\frac{\tau}{4}}{2^{-4/3} \tau^{1/3}}\leq s\right)\to F_{cross}(s,\xi),
\end{equation*}
where $F_{cross}$ is the Baik--Rains crossover distribution.
\end{theorem}

\begin{remark}
Our results require that we study the current at the boundary. It is interesting to ask about the number of particles which have passed a macroscopically far location. While certain results in this paper, in particular the connections between the six vertex model and the Hall--Littlewood measure, extend to this case, the connection to the free boundary Schur process is still not understood. It would be very interesting to develop a theory which could handle these observables. Based on results for zero-temperature models \cite{BBCS18, PS00,SI04} and universality considerations, it should be expected that suitably normalized, the number of particles past some point $x$ by time $\tau$ (both going to $\infty$) exhibits the Baik--Ben Arous--P\'ech\'e phase transition \cite{BBP05} as the boundary strength varies, with either $F_{GUE}$, $F_{GOE}^2$, or Gaussian fluctuations. It would be very interesting to show this for any positive temperature model in the half space KPZ universality class. Another possible application would be to understand the distribution of the solution to the half space KPZ equation at finite times.
\end{remark}

\begin{remark}
Although the half space ASEP is known to converge to the half space KPZ equation with Neumann boundary data \cite{CS18,P19}, formulas and asymptotics for the KPZ equation were already derived in \cite{IMS22} via the log-gamma polymer. We thus do not attempt to take a KPZ equation limit or derive asymptotics for the KPZ equation.
\end{remark}

The proof of Theorem \ref{thm: main} involves finding exact Fredholm Pfaffian formulas via an algebraic approach which gives distributional identities relating $N(\tau)$ to models where formulas are known. These formulas can then be studied via a steepest descent analysis. The formulas must be established at the level of the half space stochastic six vertex model, a more complicated model which give the ASEP as a limit, and is of independent interest.

\subsection{Half space six-vertex model}
We are also able to study the half space six vertex model with a certain family of weights. In fact, part of the proof requires working in this level of generality.

\begin{figure}
    \centering
    \includegraphics[scale=0.7]{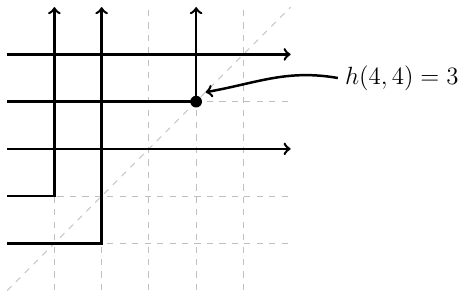}
    \caption{A configuration of the six vertex model and its height function at $(4,4)$.}
    \label{fig:6vm}
\end{figure}

The model is a probability distribution on certain configurations of arrows in the lattice $\N^2$ (see Figure \ref{fig:6vm}), which can be described with the following sampling procedure. We begin with arrows entering from the left and no arrows entering from the bottom. We then sample the outcome of vertices $(i,j)$ with $i\geq j$ in a Markovian fashion, starting with vertices whose left and bottom edges are determined, using certain integrable stochastic weights for the outcome. The outcome of vertices $(j,i)$ are determined by what occurs at $(i,j)$ in a symmetric manner. In particular, an arrow crosses an edge if and only if no arrow crosses the corresponding edge given by reflection across $x=y$. The observable which we will study is the height function $h(i,j)$, which counts the number of arrows passing at or to the left of the point $(i,j)$.

The weights depend on parameters $a_i$ associated to rows/columns, a bulk parameter $q$, and boundary parameters $t$ and $\nu$ (here $q$ and $t$ are the asymmetry parameters). For convenience we let $q_i=q$ if $i>0$ and $q_0=t$. We let
\begin{equation*}
    \mathbf{p}_{i,j}=\mathbf{p}_{|i-j|}(a_ia_j)=\begin{cases}\frac{1-a_ia_j}{1-qa_ia_j}&\text{ if }i\neq j,
    \\\frac{1-a_i^2}{(1-\nu ta_i)\left(1+\frac{1}{\nu}a_i\right)}&\text{ if }i=j.
    \end{cases}
\end{equation*}
The weights are then given in Figure \ref{fig:vtx wts}. We let $h(n,n)$ denote the \emph{height function} at $(n,n)$, defined as the number of arrows leaving a vertex upwards at or to the left of the vertex $(n,n)$. See Figure \ref{fig:6vm} for an example.

\begin{figure}
    \centering
    \begin{tabular}{l|cccccc}
    Configuration:&$\vcenter{\hbox{\includegraphics[scale=1]{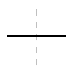}}}$ &$\vcenter{\hbox{\includegraphics[scale=1]{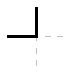}}}$&$\vcenter{\hbox{\includegraphics[scale=1]{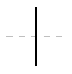}}}$&$\vcenter{\hbox{\includegraphics[scale=1]{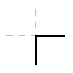}}}$&$\vcenter{\hbox{\includegraphics[scale=1]{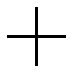}}}$& $\vcenter{\hbox{\includegraphics[scale=1]{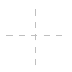}}}$\\
    Probability:&$\mathbf{p}_{i,j}$&$1-\mathbf{p}_{i,j}$&$q_{|i-j|}\mathbf{p}_{i,j}$&$1-q_{|i-j|}\mathbf{p}_{i,j}$ &1&1
    \end{tabular}
    \caption{Probabilities for sampling outgoing arrows at a vertex $(i,j)$. The black line represents an arrow and the dashed line represents no arrow.}
    \label{fig:vtx wts}
\end{figure}

\begin{theorem}
\label{thm: main 6vm}
Let $a_i=a$ for all $i$, let $h(n,n)$ denote the height function of the half space six vertex model at $(n,n)$, and let $\rho=\frac{1}{1+\nu^{-1}}$. Assume that $a\in (0,1)$ and $q,t,t\nu \in [0,1)$. Then depending on $\rho$, as $n\to\infty$,
\begin{equation*}\tag{$\rho>\frac{1}{2}$}
    \PP\left(\frac{h(n,n)-\frac{2a}{1+a}n}{\frac{(a(1-a))^{1/3}}{1+a}n^{1/3} }\leq s\right)\to F_{GSE}(s),
\end{equation*}
\begin{equation*}\tag{$\rho=\frac{1}{2}$}
    \PP\left(\frac{h(n,n)-\frac{2a}{1+a}n}{\frac{(a(1-a))^{1/3}}{1+a}n^{1/3}}\leq s\right)\to F_{GOE}(s),
\end{equation*}
\begin{equation*}\tag{$\rho<\frac{1}{2}$}
    \PP\left(\frac{h(n,n)-\mu n}{\sigma n^{1/2}}\leq s\right)\to \Phi(s),
\end{equation*}
where $F_{GSE}$ and $F_{GOE}$ are the Tracy--Widom GSE and GOE distribution functions, $\Phi$ is the distribution function for a standard Gaussian, and
\begin{equation*}
    \mu=\frac{2a^2+a(\nu+\nu^{-1})}{(1+a\nu)(1+a\nu^{-1})}, \qquad \sigma^2=\frac{a(1-a^2)(\nu^{-1}-\nu)}{(1+a\nu)^2(1+a\nu^{-1})^2}.
\end{equation*}
Furthermore, if $\rho=\frac{1}{2}+\frac{2^{-2/3}\xi}{n^{1/3}}$ and $t$ is fixed, then
\begin{equation*}\tag{$\rho\downarrow \frac{1}{2}$}
    \PP\left(\frac{h(n,n)-\frac{2a}{1+a}n}{\frac{(a(1-a))^{1/3}}{1+a}n^{1/3} }\leq s\right)\to F_{cross}(s,\xi),
\end{equation*}
where $F_{cross}$ is the Baik--Rains crossover distribution.
\end{theorem}

\begin{remark}
Our exact formulas apply to the general inhomogeneous half space stochastic six vertex model, and so in principle more general asymptotic results could be obtained, say for periodic parameters. We do not pursue the details.
\end{remark}

Since we obtain results for the half space stochastic six vertex model, it should in principle be possible to obtain formulas and thus asymptotics for other models obtained from the six vertex model via limits and a procedure known as fusion. While many of these models fall into the class of half space Macdonald processes as studied in \cite{IMS22, BBC20}, our work allows for two boundary parameters, both an asymmetry and a density, and so applies to a larger class of models.

Our method to obtain Fredholm Pfaffian formulas require the generality of the six vertex model. In the course of establishing these formulas, we find non-trivial distributional identities relating the six vertex model to two other probabilistic models, the half space Hall--Littlewood measure, and the free boundary Schur process.

\subsection{Algebraic identities}
At the heart of our approach are two unexpected distributional identities relating the half space six vertex model with two other objects, the half space Hall--Littlewood process, and then the free boundary Schur process. This allows Fredholm Pfaffian formulas derived in \cite{BBNV18,BBNV20} to be applied to the six vertex model, and ultimately the ASEP via a limiting procedure. We now state these identifications. 

Note that for these algebraic results, no assumptions on the parameters are needed if the distributional equalities are treated formally. Thus, we will interpret the following statements as equalities of formal power series, which give genuine distributional equalities if the relevant distributions are actually probability distributions, and the relevant power series converge.

We let
\begin{equation*}
    h_n(x;q)=\sum_{k=0}^n\binom{n}{k}_q x^k
\end{equation*}
denote the \emph{Rogers--Szeg\"o polynomials}, and for a partition $\lambda$, we let
\begin{equation*}
    h_\lambda(q,t,\nu)=\prod_{i\text{ even}}h_{m_i(\lambda)}(-t,q) \prod_{i\text{ odd}}(-t\nu)^{m_i(\lambda)}h_{m_i(\lambda)}\left(-\frac{1}{\nu^2 t},q\right),
\end{equation*}
where $m_i(\lambda)$ denotes the number of occurrences of $i$ in the partition $\lambda$. We define the half space Hall--Littlewood measure $\mathbb{HL}_a^{(q,t,\nu)}$ to be a (signed) probability measure on partitions proportional to $h_\lambda(q,t,\nu)P_\lambda(a;q)$, where $P_\lambda$ denotes a Hall--Littlewood polynomial. Note that this generalizes\footnote{Strictly speaking, this is non-trivially equivalent to previous definitions, see Lemma \ref{lem: match with Macdonald process}.} previous definitions of half space Hall--Littlewood measures in the literature, e.g. \cite{BBC20}. 

\begin{theorem}
\label{thm: 6vm HL ident main}
Let $h(n,n)$ denote the height function at $(n,n)$ in the half space six vertex model with parameters $(a_1,\dots, a_n)$, and let $\lambda$ be distributed according to $\mathbb{HL}_a^{(q,t,\nu)}$. Then
\begin{equation*}
    h(n,n)\deq l(\lambda).
\end{equation*}
\end{theorem}
Theorem \ref{thm: 6vm HL ident main}, which relates the half space six vertex model with the half space Hall--Littlewood measure, is a special case of the more general Theorem \ref{thm: general 6vm HL distr equality}, which applies to the joint distribution of $h(i,n)$ for all $i$. It generalizes a result in \cite{BBCW18}, which proved the $t=\nu=1$ case.  The proof uses vertex models, and in particular the Yang--Baxter equation. As a corollary, we obtain contour integral formulas for the $q$-moments of the height function for the half space six vertex model, see Corollary \ref{cor: contour int formula}.

We let $\mathbb{FBS}_{\widehat{a}}^{(q,t,\nu)}$ be a (signed) probability measure on partitions $\lambda$ proportional to
\begin{equation*}
    \sum_{\rho}\gamma_1^{o(\rho')}\gamma_2^{o(\lambda')}(q^{1/2})^{|\rho|}s_{\lambda/\rho}(\widehat{a}),
\end{equation*}
where $\gamma_1=\nu^{-1}q^{-1/2}$, $\gamma_2=-\nu t$, $o(\lambda)$ indicates the number of odd rows in $\lambda$, the sum is over all partitions $\rho\subseteq \lambda$, and $\widehat{a}$ is a particular specialization of the Schur functions defined in terms of $a_1,\dotsc,a_n$. This is a specific specialization of the free boundary Schur process as defined in \cite{BBNV18,BBNV20}. See Section \ref{sec: FBS} for further details.

We then state the following result which relates the half space Hall--Littlewood measure with the free boundary Schur process, after a random shift. We write $\chi\sim RS(q,t)$ (for what we call the \emph{Rogers--Szeg\"o distribution}, a possibly signed probability measure) if
\begin{equation*}
    \PP(\chi=k)=(q;q)_\infty(-t;q)_\infty\frac{q^kh_k(-t/q;q)}{(q;q)_k}.
\end{equation*}

\begin{theorem}
\label{thm: HL FBS same}
Let $\lambda$ and $\mu$ have distributions $\mathbb{HL}_a^{(q,t,\nu)}$ and $\mathbb{FBS}_{\widehat{a}}^{(q,t,\nu)}$. Then
\begin{equation*}
    \PP(l(\lambda)+\chi\leq s)=\PP(\mu_1\leq s),
\end{equation*}
where $\chi\sim RS(q,t)$ is independent of $\lambda$.
\end{theorem}
This is equivalent to an identity of symmetric functions (Theorem \ref{thm: general identity}), which generalizes a result in \cite{IMS21}. The proof uses a bijection found in \cite{IMS21}.

\begin{remark}
Let us remark that at certain points, the formulas we obtain involve signed probability measures. This is a technical obstacle which can be overcome, as long as these measures are exponentially decaying. We will continue to use probabilistic notation even in this case.

To ensure absolute convergence of all power series, it suffices to assume that $a_i,q,t,\nu^{-1}\in [0,1)$ and $\nu^2 t<1$. These conditions are restrictive, but can be lifted via analytic continuation to obtain Fredholm Pfaffian formulas.
\end{remark}

\subsection{Outline}
In Section \ref{sec: 6vm and HL}, we give a formal definition of the half-space six vertex model and Hall--Littlewood measure, and explain their connection. In Section \ref{sec: identity}, we establish an identity of symmetric functions connecting Schur and Hall--Littlewood functions, and in Section \ref{sec: pfaff}, we define the free boundary Schur process and use the identity from the previous section to connect it to the half space Hall--Littlewood measure. We then derive Fredholm Pfaffian formulas for the ASEP and the six-vertex model through known formulas for the free boundary Schur process. Section \ref{sec: asymp prelim} gives some useful tools and formulas for the asymptotic analysis, which is done in Sections \ref{sec: cross} and \ref{sec: Gauss} for the $\rho\geq \frac{1}{2}$ and $\rho<\frac{1}{2}$ regimes respectively.

\subsection{Notation}
We will use $(a;q)_n=\prod_{i=0}^{n-1}(1-aq^i)$ to denote the $q$-Pochhammer symbol (with $n=\infty$ allowed), $(a_1,\dotsc, a_k;q)_n=(a_1;q)_n\dotsm (a_k;q)_n$, and
\begin{equation*}
    \binom{b}{a}_q=\frac{(q;q)_b}{(q;q)_a(q;q)_{b-a}}
\end{equation*}
to denote the $q$-binomial coefficient.

We use $\deq$ to denote equality in distribution between two random variables. We use $I_S$ to denote the indicator function for a statement $S$.

We will use $c$ and $C$ to denote small and large constants respectively which may change from line to line.

\section{The half space six vertex model and half space Hall--Littlewood measure}
\label{sec: 6vm and HL}
In this section, we formally define the half space stochastic six vertex model and Hall--Littlewood process, and prove a result relating their observables.

\subsection{Half space six vertex model}
\label{sec: 6vm}
We start with parameters $a_i$ for $i\in\N$, $q$, $t$ and $\nu$. We will sometimes use the notation $q_i=q$ if $i>0$ and $q_0=t$ to unify certain expressions. We call the $a_i$ \emph{rapidities}. Initially, we will assume that they are chosen so all weights defined below are probabilities, but later on we will view the weights as rational functions in these parameters, which will then be treated as formal variables. Recall that
\begin{equation*}
    \mathbf{p}_{i,j}=\mathbf{p}_{|i-j|}(a_ia_j)=\begin{cases}\frac{1-a_ia_j}{1-qa_ia_j}&\text{ if }i\neq j,
    \\\frac{1-a_i^2}{(1-\nu ta_i)\left(1+\frac{1}{\nu}a_i\right)}&\text{ if }i=j.
    \end{cases}
\end{equation*}
We will refer to the argument of $\mathbf{p}_{|i-j|}$, the $a_ia_j$, as the \emph{spectral parameter}.

We consider certain configurations of arrows on the edges of $\N^2$. Each edge contains at most one arrow. Initially, we begin with arrows entering on the left along the $y$-axis, and no arrows entering from the bottom along the $x$-axis. See Figure \ref{fig:6vm} for an example of such a configuration. We will refer to vertices $(i,j)$ with $i\neq j$ as \emph{bulk} vertices and vertices $(i,i)$ as \emph{boundary} vertices. Vertices $(1,j)$ and $(i,1)$ will be referred to as \emph{incoming} vertices.

We now describe a Markovian sampling procedure to generate configurations. At a vertex $(i,j)$ with $i\leq j$ where both the left and bottom incoming edges already have arrows assigned, we sample outgoing arrows at the top and right according to the probabilities given in Figure \ref{fig:vtx wts}. If the vertex is not a boundary vertex (i.e. $i< j$), we then also symmetrically set the outcome in vertex $(j,i)$ by requiring that if an edge is occupied by an arrow, then its reflection about $x=y$ is not, and vice versa. Repeating this process inductively, we define a random configuration on $\N^2$. Note that as long as we choose the parameters correctly, we obtain a genuine probability distribution on configurations. For this reason, we will assume from now on that $0<a_i<1$ for all $i$, $0\leq q,t<1$, $\nu>0$, and $\nu t<1$, which guarantees that we indeed obtain a probability distribution.

The bottom half of the model contains no additional information, and so it is possible to consider the model only on vertices $(i,j)$ with $j\geq i$, which we will sometimes do. In this case, boundary vertices should be thought of as having only one incoming and one outgoing edge, with the same probabilities as described above.

We let $h(n,n)$ denote the \emph{height function} at $(n,n)$, defined as the number of arrows leaving a vertex upwards at or to the left of the vertex $(n,n)$. See Figure \ref{fig:6vm} for an example. Given a configuration, we define the \emph{path string} $S=(s_1,\dots, s_n)$ (within some fixed box of size $n$) to be the sequence of $0$'s and $1$'s indicating whether an edge is occupied with an arrow or not at the top of the box. We have $\sum_i s_i=h(n,n)$. In the example given by Figure \ref{fig:6vm}, $S=(1,1,0,1)$.

The identification of certain parameters in the ASEP and the six vertex model is not an accident. Indeed, it is known that under a certain scaling, the six vertex model converges to the ASEP. Ultimately, we need the six vertex model as parts of the proof use vertex model techniques. The formulas we obtain for the ASEP are derived as a limit of ones for the six vertex model.

\subsection{Symmetric functions}
We will use symmetric functions in an essential way. We thus collect some basic background and notation which will be useful throughout the paper. We refer the reader to \cite{M79} for further background.

A \emph{partition} is a finite sequence of non-increasing numbers non-negative integers. For a partition $\lambda=(\lambda_1,\dotsc,\lambda_n)$, we will write $l(\lambda)=n$ for the \emph{length}, $m_i(\lambda)$ for the number of occurrences of $i$ in $\lambda$, and $\lambda'$ for the \emph{conjugate partition} given by transposing its \emph{Young diagram}. We let $o(\lambda)$ denote the number of odd rows in $\lambda$. We will also sometimes allow $0$ as a part, with the obvious modifications to the above definitions when this is done.

We let $x=(x_1,\dotsc,)$ denote a formal alphabet. We will let $p_\lambda(x)$ denote the power sum symmetric functions, defined by $p_\lambda(x)=\prod p_{\lambda_i}(x)$, where $p_n(x)=\sum x_i^n$. These form a basis for the ring of symmetric functions, which we denote by $\Lambda$. It is a fact that the $p_n$ are an algebraically independent generating set for the ring of symmetric functions, i.e. $\Lambda\cong \C[p_1,p_2,\dotsc]$. We define a \emph{specialization} $a$ to be an algebra homomorphism $a:\Lambda\to \C$ (or possibly some other field). Since the $p_n$ are algebraically independent and generate $\Lambda$, a specialization $a$ is equivalent to a choice of where to send $p_n$ for all $n$. 

There is a natural family of specializations, given by plugging in $x_i=a_i$ for $a_i\in \C$ and $i\leq n$, and $x_i=0$ for $i>n$. For this reason, we will use the notation $f(a)=a(f)$ for $a$ a specialization and $f\in \Lambda$, even if $a$ is not of this form. One other family of specializations which we will use is the $q$-beta specialization. We define an automorphism of $\Lambda$ by
\begin{equation*}
    p_n(\widehat{x})=(-1)^{n-1}(1-q^{n})p_n(x),
\end{equation*}
and for any specialization $a$, we let $\widehat{a}$ denote its composition with this automorphism. We will normally consider the case when $a$ is given by plugging in complex numbers $a=(a_1,\dotsc,a_n)$. 

Finally, we will deal with three other families of symmetric functions (all indexed by partitions). The \emph{Schur functions}, which we denote $s_\lambda$, are crucial in defining the free boundary Schur process. The other two families, the \emph{Hall--Littlewood polynomials} $P_\lambda(x;q)$ and \emph{$q$-Whittaker polynomials} $\widehat{Q}_{\lambda}(x;q)$, are related by
\begin{equation*}
    \widehat{Q}_\lambda(\widehat{x};q)=P_{\lambda'}(x,q).
\end{equation*}
We will ultimately need more properties of these symmetric functions, but we will introduce these as needed.

\subsection{Half space Hall--Littlewood measures}
\label{sec: HL meas}
Recall that
\begin{equation*}
    h_n(x;q)=\sum_{k=0}^n\binom{n}{k}_q x^k
\end{equation*}
denotes the \emph{Rogers--Szeg\"o polynomials}. They satisfy the recurrence
\begin{equation*}
    h_{n+1}(x;q)=(1+x)h_n(x;q)+x(q^n-1)h_{n-1}(x;q),
\end{equation*}
with $h_0(x;q)=1$ and $h_1(x;q)=1+x$. We will also define
\begin{equation*}
    h_\lambda(q,t,\nu)=\prod_{i\text{ even}}h_{m_i(\lambda)}(-t,q) \prod_{i\text{ odd}}(-t\nu)^{m_i(\lambda)}h_{m_i(\lambda)}\left(-\frac{1}{\nu^2 t},q\right).
\end{equation*}
Due to the identity $h_n(x^{-1},q)=x^{-n}h_n(x;q)$, $h_\lambda(q,t,\nu)$ is invariant under the substitution $\nu\mapsto (-\nu t)^{-1}$, which has the effect of swapping $\nu^{-1}$ and $-\nu t$. We also note the easy bound
\begin{equation}
\label{eq: RS bound}
|h_n(x;q)|\leq \frac{n|x|^n}{(q;q)_n},
\end{equation}
valid if $0<q<1$ and $|x|\geq 1$.

\begin{remark}
We note that if $0<q,t,\nu^2 t<1$, then $h_\lambda(q,t,\nu)$ is always non-negative, because the even factors are always non-negative from the recurrence, and the odd factors are non-negative as the sign of $h_{m_i(\lambda)}\left(-\frac{1}{\nu^2 t},q\right)$ is $(-1)^{m_i(\lambda)}$, which can be proven by induction as in the recurrence, $1+x$ is negative and $x(q^n-1)$ is positive.
\end{remark}

Let $q,t\in [0,1)$ and $\nu>0$ such that $\nu^2 t<1$ be parameters, and let $a$ denote a specialization of the Hall--Littlewood polynomials of parameter $q$ such that $P_\lambda(a;q)\geq 0$ for all $\lambda$. We define the \emph{half space Hall--Littlewood measure}\footnote{Note that in general this is a signed probability measure. See Lemma \ref{lem: when signed} for conditions on the parameters which ensure it is a genuine probability measure.} to be
\begin{equation*}
    \mathbb{HL}_a^{(q,t,\nu)}(\lambda)=\frac{1}{\Pi(a;q,t,\nu)} h_\lambda(q,t,\nu)P_\lambda(a;q),
\end{equation*}
where $\Pi$ normalizes this to sum to $1$. If $a=(a_1,a_2,\dotsc)$, we have
\begin{equation}
\label{eq: Littlewood identity}
    \Pi(a;q,t,\nu)=\prod \frac{(1-a_i\nu t)(1+a_i/\nu)}{1-a_i^2}\prod_{i<j}\frac{1-qa_ia_j}{1-a_ia_j}.
\end{equation}
The fact that this sums to $1$ is equivalent to a Littlewood identity of Warnaar (Theorem 1.1 of \cite{W06}). This section will give an independent proof of this, see Remark \ref{rmk: Littlewood identity new pf}.

For a specialization $a=a_1+ a_2+\dotsm+ a_n$ (if the $a_i$ are alphabets, this notation means concatenation), we define the \emph{half space Hall--Littlewood process} to be the measure on ascending sequences of partitions of length $n$, given by
\begin{equation*}
    \mathbb{HL}_a^{(q,t,\nu)}(\Vec{\lambda})=\frac{1}{\Pi(a;q,t,\nu)} h_{\lambda^{(n)}}(q,t,\nu)\prod_i P_{\lambda^{(i)}/\lambda^{(i-1)}}(a_i;q),
\end{equation*}
where $\Vec{\lambda}=(\lambda^{(1)}\subseteq \dotsc \subseteq  \lambda^{(n)})$ is an ascending sequence of partitions, and $\lambda^{(0)}=\emptyset$. The $\lambda^{(n)}$ marginal is the half space Hall--Littlewood measure, which follows from the branching rule for Hall--Littlewood polynomials (see (5.5$'$) of \cite{M79})
\begin{equation*}
    \sum_{\rho}P_{\lambda/\rho}(x;q)P_{\rho/\mu}(y;q)=P_{\lambda/\mu}(x,y;q).
\end{equation*}
More generally, one could define measures on sequences of partitions which are both ascending and descending, as done in \cite{BBC20}, but we leave this to the interested reader.

\begin{lemma}
\label{lem: when signed}
If $q,t,a_i,\nu t,\nu^{-1}a_i\in [0,1)$, $\nu^2 t<1$, then $\mathbb{HL}_a^{(q,t,\nu)}$ is a genuine probability measure. If the condition that $\nu^2 t<1$ is dropped, then it is a signed probability measure which is absolutely summable.
\end{lemma}
\begin{proof}
First, we note that if $t\in [0,1)$ and $\nu^2t<1$, then $h_\lambda(q,t,\nu)\geq 0$. We have $h_\lambda(q,t,\nu)$ decays exponentially in $l(\lambda)$ if $t\in [0,1)$ and $\nu^{-1}<1$. The condition that $\nu^{-1}<1$ can be relaxed to $\nu^{-1}a_i<1$ for all $i$, as by homogeneity of $P_\lambda(a;q)$, we can replace $a_i$ with $\nu^{-1}a_i$.
\end{proof}

\begin{remark}
Although the half space Hall--Littlewood process is seemingly more general than those previously considered in the literature, e.g. in \cite{BBCW18,IMS21, BBC20}, which correspond to $\nu=t=1$, it is actually possible to obtain it via the usual half space Hall--Littlewood process by a plethystic substitution, which is more or less explained in \cite{W06} at the level of the Littlewood identity. However, this does not appear to have been noticed, and certainly the connection to the half space ASEP with general boundary seems to be new. This does mean that all general results on half space Macdonald processes should apply (see e.g. \cite{BBC20}), although often the results are stated for positive specializations and so do not immediately apply in our setting (but the proofs should still work). We note that this specialization is special because the factor $h_\lambda$ factorizes, something which does not happen in general.
\end{remark}
 Since the argument given in \cite{W06} is not quite enough to match to \cite{BBC20} and the substitution itself is non-trivial, we give an explanation in the following lemma.
\begin{lemma}
\label{lem: match with Macdonald process}
Following the notation of \cite{BBC20} (except that we have exchanged the roles of $q$ and $t$), we have
\begin{equation*}
    \mathcal{E}_\lambda(b)=\sum_{\mu'\text{ even}}b^{el}_\mu(q)Q_{\lambda/\mu}(b)=h_\lambda(q,t,\nu),
\end{equation*}
where $b$ is the specialization defined using plethysm as
\begin{equation*}
    b=\frac{\{1,-1\}-\{\nu t, -\nu^{-1}\}}{1-q}.
\end{equation*}
Alternatively, $b$ can be defined by
\begin{equation*}
    p_k(b)=\frac{1+(-1)^k-(\nu t)^k-(-\nu)^{-k}}{1-q^k}.
\end{equation*}
In particular, the half space Hall--Littlewood process considered in this paper can be obtained via the specialization $b$ from the one defined in \cite{BBC20}. 
\end{lemma}
\begin{proof}
Starting from the identity (see e.g. Equation 30 in \cite{BBC20}),
\begin{equation*}
    \sum_\lambda P_\lambda(x;q)\mathcal{E}_\lambda(b)= \Pi(x,b)\prod_{i<j}\frac{1-qx_ix_j}{1-x_ix_j},
\end{equation*}
it suffices to check that
\begin{equation*}
    \Pi(x;b)=\prod_i \frac{(1-\nu t x_i)(1+x_i/\nu)}{1-x_i^2},
\end{equation*}
as the $P_\lambda(x;q)$ are a basis for $\Lambda$ and we already have an expansion with coefficients $h_\lambda(q,t,\nu)$. This computation can be readily checked.
\end{proof}

Note that if $a=(a_1,\dotsc, a_n)$ and each $a_i$ is a single variable, then the lengths of the partition increase by at most $1$. When $a=(a_1,\dotsc,a_n)$, we write $[\Vec{\lambda}]=(s_1,\dotsc, s_n)$ to denote the binary sequence of length $n$ with $s_i=I_{l(\lambda^{(i)})=l(\lambda^{(i-1)})+1}$, called the \emph{support} of $\Vec{\lambda}$.

We now state the main result of this section, which gives a distributional relationship between the half space six vertex model and the half space Hall--Littlewood measure.
\begin{theorem}
\label{thm: general 6vm HL distr equality}
Let $a=(a_1,\dotsc, a_n)$ and $q,t,\nu$ denote the rapidities and other parameters in the half space six vertex model, and let $S$ denote the path string of a configuration from the six vertex model. Then
\begin{equation*}
    \PP(S=s)=\mathbb{HL}_a^{(q,t,\nu)}([\Vec{\lambda}]=s).
\end{equation*}
In particular, we have
\begin{equation*}
    (h(i,n))_{i=1}^n\deq (l(\lambda^{(i)}))_{i=1}^n.
\end{equation*}
\end{theorem}

\begin{remark}
As in Remark 4.6 of \cite{BBCW18}, we expect a more general statement to be true, relating a Hall--Littlewood process with ascending and descending partitions to $h(i,j)$ for $(i,j)$ following a jagged path. Since this is not needed in this paper, we do not pursue this further.
\end{remark}

Using Theorem \ref{thm: general 6vm HL distr equality} and known contour integral formulas for certain $q$-moments of the half space Hall--Littlewood process \cite{BBC20}, we obtain the following formulas for the $q$-moments of the height function.
\begin{corollary}
\label{cor: contour int formula}
Let $1\leq n\leq m$ and let $k\in\N$, and let $h(n,m)$ denote the height function of the half space six vertex model. Then
\begin{equation*}
\begin{split}
    \E\left(q^{-k h(n,m)}\right)=&q^{k\choose 2}\frac{1}{(2\pi i)^k}\oint_{C_1}dz_1\dotsm \oint_{C_k}dz_k\prod_{i<j}\frac{z_i-z_j}{z_i-qz_j}\frac{1-qz_iz_j}{1-z_iz_j}
    \\&\qquad \times\prod_{j=1}^k\left(\frac{1}{z_j}\frac{1-qz_j^2}{(1-\nu t z_j)(1+z_j/\nu)}\prod_{i=1}^m\frac{1-a_iz_j}{1-qa_iz_j}\prod_{i=1}^n\frac{z_j-a_i/q}{z_j-a_i}\right),
\end{split}
\end{equation*}
with the contours $C_i$ are positively oriented, contain $0$ and $a_i$, are contained in the open ball of radius $1$, and are nested such that for $i<j$, $C_i$ does not contain any part of $tC_j$.
\end{corollary}
\begin{proof}
The proof follows the same ideas as Corollary 5.8 of \cite{BBC20}. In particular, we can use Theorem \ref{thm: general 6vm HL distr equality} and Lemma \ref{lem: match with Macdonald process} to relate the expectation to a half space Hall--Littlewood process in the sense of \cite{BBC20}, with $\rho=(a_{n+1},\dotsc, a_m)+\frac{\{1,-1\}-\{\nu t, -\nu^{-1}\}}{1-q}$. The result in \cite{BBC20} is only stated when $\rho=(a_{n+1},\dotsc, a_m)$, but the same proof works in this more general case. In particular, this change only affects the integrand by a factor of $\prod_i \frac{1-z_i^2}{(1-\nu t z_i)(1+z_i/\nu)}$, and the new poles are outside the unit circle and so do not interfere with the contour manipulations done in the proof.
\end{proof}

\begin{remark}
This formula cannot be deformed to the ASEP limit, as this requires $a_i\to 1$ while the contours must contain $a_i$ but not $1$, which would cause the poles $z_i=z_j^{-1}$ to become an issue. However, in \cite{BC22}, formulas for $q$-moments are derived in terms of a residue expansion (under Liggett's condition on the jump rates). It would be interesting to understand the relationship between these formulas, and whether the formulas found in \cite{BC22} could be extended to more general rates.
\end{remark}

\begin{remark}
The restriction on the parameter $\nu$, that $\nu^2t<1$, will ultimately need to be removed. Fortunately, this is not an issue, since we will only need to study the first part of $\lambda_n$, and this marginal is a probability distribution even if $\nu^2t\geq 1$. Furthermore, Theorem \ref{thm: general 6vm HL distr equality} holds as an equality of rational functions. Thus, even if the parameters are arbitrary, the result remains true, although the two sides may be negative.
\end{remark}

\begin{remark}
Since $h_\lambda(q,t,\nu)$ is invariant under $\nu\mapsto (-\nu t)^{-1}$, and so is $\Pi(a;q,t,\nu)$, the half space Hall--Littlewood process is also invariant under this substitution. For this reason, it's completely harmless to assume $\nu\geq 0$ (except for the special case $t=0$ which will be discussed in detail in Remark \ref{rem: t=0 special}). We note that the the half space six vertex model also has this symmetry.
\end{remark}

\begin{remark}
Given that we have defined a two-parameter generalization, it is natural to wonder whether a Macdonald version of this process exists. Unfortunately, although it is possible to define such a process, see e.g. \cite{W06} where a Littlewood identity is proven, it lacks many of the nice features that the Hall--Littlewood process has, and in particular the factor replacing $h_\lambda$ no longer factorizes. However, for some one-parameter families of specializations, this property remains (see Proposition 1.3 in \cite{W06}), which suggests that these processes may also be of some interest. We do not consider these questions here.
\end{remark}

\begin{remark}
\label{rmk: Littlewood identity new pf}
Since the proof of Theorem \ref{thm: general 6vm HL distr equality} does not require a priori knowledge of the normalizing constant $\Pi(a;q,t,\nu)$, and the six vertex model does not require any normalization, one can actually derive the formula \eqref{eq: Littlewood identity} from it. This is equivalent to Theorem 1.1 of \cite{W06}, and so we have given a new vertex model proof of this result.
\end{remark}

\subsection{Deformed bosons}
A key tool to prove Theorem \ref{thm: general 6vm HL distr equality} will be a deformed model of bosons. This was used to establish a special case of Theorem \ref{thm: general 6vm HL distr equality} in \cite{BBCW18} (when $\nu=t=1$), and we extend their proof. In order to do so, we will need to introduce some results from the paper. We will require two versions of this model, which we now define. 

Both models are models for arrows on a lattice, where the horizontal edges again can have at most one arrow, but where the vertical edges can have any number of arrows. The model has parameters (called \emph{rapidities}) associated to each row and a parameter $q$. The vertex weights given by

\begin{align}
\label{eq:black-vertices}
\begin{array}{cccc}
\begin{tikzpicture}[scale=0.8,>=stealth]
\draw[lgray,ultra thick] (-1,0) -- (1,0);
\draw[lgray,line width=10pt] (0,-1) -- (0,1);
\node[below] at (0,-1) {$m$};
\draw[ultra thick,->,rounded corners] (-0.075,-1) -- (-0.075,1);
\draw[ultra thick,->,rounded corners] (0.075,-1) -- (0.075,1);
\node[above] at (0,1) {$m$};
\end{tikzpicture}
\quad\quad\quad
&
\begin{tikzpicture}[scale=0.8,>=stealth]
\draw[lgray,ultra thick] (-1,0) -- (1,0);
\draw[lgray,line width=10pt] (0,-1) -- (0,1);
\node[below] at (0,-1) {$m$};
\draw[ultra thick,->,rounded corners] (-0.075,-1) -- (-0.075,1);
\draw[ultra thick,->,rounded corners] (0.075,-1) -- (0.075,0) -- (1,0);
\node[above] at (0,1) {$m-1$};
\end{tikzpicture}
\quad\quad\quad
&
\begin{tikzpicture}[scale=0.8,>=stealth]
\draw[lgray,ultra thick] (-1,0) -- (1,0);
\draw[lgray,line width=10pt] (0,-1) -- (0,1);
\node[below] at (0,-1) {$m$};
\draw[ultra thick,->,rounded corners] (-1,0) -- (-0.15,0) -- (-0.15,1);
\draw[ultra thick,->,rounded corners] (0,-1) -- (0,1);
\draw[ultra thick,->,rounded corners] (0.15,-1) -- (0.15,1);
\node[above] at (0,1) {$m+1$};
\end{tikzpicture}
\quad\quad\quad
&
\begin{tikzpicture}[scale=0.8,>=stealth]
\draw[lgray,ultra thick] (-1,0) -- (1,0);
\draw[lgray,line width=10pt] (0,-1) -- (0,1);
\node[below] at (0,-1) {$m$};
\draw[ultra thick,->,rounded corners] (-1,0) -- (-0.15,0) -- (-0.15,1);
\draw[ultra thick,->,rounded corners] (0,-1) -- (0,1);
\draw[ultra thick,->,rounded corners] (0.15,-1) -- (0.15,0) -- (1,0);
\node[above] at (0,1) {$m$};
\end{tikzpicture}
\\
1
\quad\quad\quad
&
a
\quad\quad\quad
&
(1-q^{m+1})
\quad\quad\quad
&
a
\end{array}\end{align}
where $0\leq a<1$ is the row rapidity, and $m$ is the number of arrows entering from below.   

We will also need another version of this model with an alternative normalization. It is defined in the same way, but with alternative vertex weights

\begin{align}
\label{eq:red-vertices}
\begin{array}{cccc}
\begin{tikzpicture}[scale=0.8,>=stealth]
\draw[lred, ultra thick] (-1,0) -- (1,0);
\draw[lred,line width=10pt] (0,-1) -- (0,1);
\node[below] at (0,-1) {$m$};
\draw[ultra thick,->,rounded corners] (-0.075,-1) -- (-0.075,1);
\draw[ultra thick,->,rounded corners] (0.075,-1) -- (0.075,1);
\node[above] at (0,1) {$m$};
\end{tikzpicture}
\quad\quad\quad
&
\begin{tikzpicture}[scale=0.8,>=stealth]
\draw[lred,ultra thick] (-1,0) -- (1,0);
\draw[lred,line width=10pt] (0,-1) -- (0,1);
\node[below] at (0,-1) {$m$};
\draw[ultra thick,->,rounded corners] (-0.075,-1) -- (-0.075,1);
\draw[ultra thick,->,rounded corners] (0.075,-1) -- (0.075,0) -- (1,0);
\node[above] at (0,1) {$m-1$};
\end{tikzpicture}
\quad\quad\quad
&
\begin{tikzpicture}[scale=0.8,>=stealth]
\draw[lred,ultra thick] (-1,0) -- (1,0);
\draw[lred,line width=10pt] (0,-1) -- (0,1);
\node[below] at (0,-1) {$m$};
\draw[ultra thick,->,rounded corners] (-1,0) -- (-0.15,0) -- (-0.15,1);
\draw[ultra thick,->,rounded corners] (0,-1) -- (0,1);
\draw[ultra thick,->,rounded corners] (0.15,-1) -- (0.15,1);
\node[above] at (0,1) {$m+1$};
\end{tikzpicture}
\quad\quad\quad
&
\begin{tikzpicture}[scale=0.8,>=stealth]
\draw[lred,ultra thick] (-1,0) -- (1,0);
\draw[lred,line width=10pt] (0,-1) -- (0,1);
\node[below] at (0,-1) {$m$};
\draw[ultra thick,->,rounded corners] (-1,0) -- (-0.15,0) -- (-0.15,1);
\draw[ultra thick,->,rounded corners] (0,-1) -- (0,1);
\draw[ultra thick,->,rounded corners] (0.15,-1) -- (0.15,0) -- (1,0);
\node[above] at (0,1) {$m$};
\end{tikzpicture}
\\
b
\quad\quad\quad
&
1
\quad\quad\quad
&
b (1-q^{m+1})
\quad\quad\quad
&
1
\end{array}
\end{align}
where again, $0\leq b<1$ is the row rapidity and $m$ is the number of arrows entering from below.

In both cases, we will use a graphical notation to write partition functions. For a single row, we will write $w_a(\cdot )$ around a picture to represent the sum over all internal edges of the products of the vertex weights, with external edges fixed and $a$ denoting the row rapidity. If there is more than one row, we will instead simply give the picture with row rapidities identified.

We will also wish to consider infinite rows, formally defined as a limit of longer finite rows. For finitely supported sequences of non-negative integers $m_i$ and $n_i$, we will let
\begin{multline*}
    w_a
\left(
\begin{tikzpicture}[baseline=(current bounding box.center),>=stealth,scale=0.8]
\draw[lgray,ultra thick] (0,0) -- (7,0);
\foreach\x in {1,...,6}{
\draw[lgray,line width=10pt] (7-\x,-1) -- (7-\x,1);
}
\node[left] at (0,0) {$0$};
\node[right] at (7,0) {$j$};
\foreach\x in {1,2,3}{
\node[text centered,below] at (7-\x,-1) {\tiny $m_{\x}$};
\node[text centered,above] at (7-\x,1) {\tiny $n_{\x}$};
}
\foreach\x in {4,5}{
\node[text centered,below] at (7-\x,-1) {\tiny $\cdots$};
\node[text centered,above] at (7-\x,1) {\tiny $\cdots$};
}
\end{tikzpicture}\right)
\\=\lim_{N\to\infty}w_a
\left(
\begin{tikzpicture}[baseline=(current bounding box.center),>=stealth,scale=0.8]
\draw[lgray,ultra thick] (0,0) -- (7,0);
\foreach\x in {1,...,6}{
\draw[lgray,line width=10pt] (7-\x,-1) -- (7-\x,1);
}
\node[left] at (0,0) {$0$};
\node[right] at (7,0) {$j$};
\foreach\x in {1,2,3}{
\node[text centered,below] at (7-\x,-1) {\tiny $m_{\x}$};
\node[text centered,above] at (7-\x,1) {\tiny $n_{\x}$};
}
\foreach\x in {4,5}{
\node[text centered,below] at (7-\x,-1) {\tiny $\cdots$};
\node[text centered,above] at (7-\x,1) {\tiny $\cdots$};
}
\node[text centered,below] at (1,-1) {\tiny $m_{N}$};
\node[text centered,above] at (1,1) {\tiny $n_{N}$};
\end{tikzpicture}\right),
\end{multline*}
and similarly for the second deformed boson model, except with an arrow entering from the left. These limits are well-defined, because the $m_i$ and $n_i$ are finitely supported, so eventually the weights will all be $1$. Note that if an arrow entered the left (or no arrow entered in the second deformed boson model), the weights would eventually be $a$, and not $1$, and the limit would simply be $0$ since $a<1$.

Finally, we will introduce an extra feature which we call a \emph{boundary vertex}, which corresponds to the boundary of the half space six vertex model. We will use a dot on a horizontal edge to denote it visually, and it has the effect of probabilistically outputting an arrow depending on the input, with probabilities
\begin{align*}
\begin{array}{cccc}
\begin{tikzpicture}[scale=0.8,>=stealth]
\draw[lgray,ultra thick] (-1,0) -- (1,0);
\node at (0,0) {$\bullet$};
\end{tikzpicture}
\quad\quad
&
\begin{tikzpicture}[scale=0.8,>=stealth]
\draw[lgray,ultra thick] (-1,0) -- (1,0);
\draw[ultra thick,->] (0,0) -- (1,0);
\node at (0,0) {$\bullet$};
\end{tikzpicture}
\quad\quad
&
\begin{tikzpicture}[scale=0.8,>=stealth]
\draw[lgray,ultra thick] (-1,0) -- (1,0);
\draw[ultra thick,->] (-1,0) -- (0,0);
\node at (0,0) {$\bullet$};
\end{tikzpicture}
\quad\quad
&
\begin{tikzpicture}[scale=0.8,>=stealth]
\draw[lgray,ultra thick] (-1,0) -- (1,0);
\draw[ultra thick,->] (-1,0) -- (1,0);
\node at (0,0) {$\bullet$};
\end{tikzpicture}
\vspace{0.2cm}
\\
\frac{a\nu^{-1}(1-t\nu^2)+(1-t)}{(1-a\nu t)(1+a/\nu)}
\quad\quad
&
\frac{t(1-a^2)}{(1-a\nu t)(1+a/\nu)}
\quad\quad
&
\frac{1-a^2}{(1-a\nu t)(1+a/\nu)}
\quad\quad
&
\frac{a\nu^{-1}(1-t\nu^2)+a^2(1-t))}{(1-a\nu t)(1+a/\nu)}
\end{array}\end{align*}
where $a$ is the row rapidity. Note these are equal to the boundary vertex weights in the half space six vertex model. Here, the weights are the same in both models.

The reason this model is useful to show Theorem \ref{thm: general 6vm HL distr equality} is that it shares the same $R$-matrix as the six vertex model, but its partition functions are given by Hall--Littlewood polynomials. These two statements are given in the following lemmas. We will use a crossing rotated by $45^\circ$ to denote a \emph{Yang--Baxter vertex}, a vertex with the same weights as those of the six--vertex model with corresponding row and column parameter.

\begin{proposition}[{\hspace{1sp}\cite[Proposition 4.8]{BBCW18}}]
\label{prop: YB boson}
For any finitely supported sequences $n_i$ and $m_i$, and any $j_1,j_2=0,1$, we have
\begin{equation*}
\label{graph-exchange}
\left(
\frac{1-a b}{1-q a b}
\right)
\sum_{p_i}
\begin{tikzpicture}[baseline=(current bounding box.center),>=stealth,scale=0.7]
\draw[lgray,ultra thick] (-1,1) node[left,black] {$a$}
-- (4,1) node[right,black] {$j_2$};
\draw[lred,ultra thick] (-1,0) node[left,black] {$b$}
-- (4,0) node[right,black] {$j_1$};
\foreach\x in {0,...,3}{
\draw[lgray,line width=10pt] (3-\x,0.5) -- (3-\x,2);
\draw[lred,line width=10pt] (3-\x,-1) -- (3-\x,0.5);
}
\node[below] at (3,-1) {$m_1$};
\node at (3,0.5) {$p_1$};
\node[above] at (3,2) {$n_1$};
\node[below] at (2,-1) {$m_2$};
\node at (2,0.5) {$p_2$};
\node[above] at (2,2) {$n_2$};
\node[text centered] at (0,0.5) {$\cdots$};
\node[text centered] at (1,0.5) {$\cdots$};
\draw[ultra thick,->] (-1,0) -- (0,0);
\end{tikzpicture}
=
\sum_{p_i,k_1,k_2}
\begin{tikzpicture}[baseline=(current bounding box.center),>=stealth,scale=0.7]
\foreach\x in {0,...,3}{
\draw[lgray,line width=10pt] (3-\x,-1) -- (3-\x,0.5);
\draw[lred,line width=10pt] (3-\x,0.5) -- (3-\x,2);
}
\draw[lred,ultra thick] (-1,1) node[left,black] {$b$}
-- (4,1);
\draw[lgray,ultra thick] (-1,0) node[left,black] {$a$}
-- (4,0);
\draw[dotted,thick] (4,1) node[above] {$k_1$} -- (5,0) node[right] {$j_1$};
\draw[dotted,thick] (4,0) node[below] {$k_2$} -- (5,1) node[right] {$j_2$};
\node[below] at (3,-1) {$m_1$};
\node at (3,0.5) {$p_1$};
\node[above] at (3,2) {$n_1$};
\node[below] at (2,-1) {$m_2$};
\node at (2,0.5) {$p_2$};
\node[above] at (2,2) {$n_2$};
\node[text centered] at (0,0.5) {$\cdots$};
\node[text centered] at (1,0.5) {$\cdots$};
\draw[ultra thick,->] (-1,1) -- (0,1);
\end{tikzpicture},
\end{equation*}
where on the left, $a$ and $b$ indicate the row rapidities, and the left boundary conditions are given by the arrows as indicated.
\end{proposition}

\begin{lemma}[{\hspace{1sp}\cite[Lemma 4.11]{BBCW18}}]
\label{lem: boson HL}
For any partitions $\lambda$ and $\mu$, we have
\begingroup
\allowdisplaybreaks
\begin{align*}
    w_a
\left(
\begin{tikzpicture}[baseline=(current bounding box.center),>=stealth,scale=0.8]
\draw[lgray,ultra thick] (0,0) -- (7,0);
\foreach\x in {1,...,6}{
\draw[lgray,line width=10pt] (7-\x,-1) -- (7-\x,1);
}
\node[left] at (0,0) {$0$};
\node[right] at (7,0) {$0$};
\foreach\x in {1,2,3}{
\node[text centered,below] at (7-\x,-1) {\tiny $m_{\x}(\lambda)$};
\node[text centered,above] at (7-\x,1) {\tiny $m_{\x}(\mu)$};
}
\foreach\x in {4,5}{
\node[text centered,below] at (7-\x,-1) {\tiny $\cdots$};
\node[text centered,above] at (7-\x,1) {\tiny $\cdots$};
}
\end{tikzpicture}
\right)&=I_{l(\lambda)=l(\mu)}P_{\lambda/\mu}(a;q),
    \\w_a
\left(
\begin{tikzpicture}[baseline=(current bounding box.center),>=stealth,scale=0.8]
\draw[lgray,ultra thick] (0,0) -- (7,0);
\foreach\x in {1,...,6}{
\draw[lgray,line width=10pt] (7-\x,-1) -- (7-\x,1);
}
\node[left] at (0,0) {$0$};
\draw[ultra thick,->] (6,0) -- (7,0);
\node[right] at (7,0) {$1$};
\foreach\x in {1,2,3}{
\node[text centered,below] at (7-\x,-1) {\tiny $m_{\x}(\lambda)$};
\node[text centered,above] at (7-\x,1) {\tiny $m_{\x}(\mu)$};
}
\foreach\x in {4,5}{
\node[text centered,below] at (7-\x,-1) {\tiny $\cdots$};
\node[text centered,above] at (7-\x,1) {\tiny $\cdots$};
}
\end{tikzpicture}\right)&=I_{l(\lambda)=l(\mu)+1}P_{\lambda/\mu}(a;q),
    \\w_a
\left(
\begin{tikzpicture}[baseline=(current bounding box.center),>=stealth,scale=0.8]
\draw[lred,ultra thick] (0,0) -- (7,0);
\foreach\x in {1,...,6}{
\draw[lred,line width=10pt] (7-\x,-1) -- (7-\x,1);
}
\node[left] at (0,0) {$1$};
\draw[ultra thick,->] (0,0) -- (1,0);
\node[right] at (7,0) {$0$};
\foreach\x in {1,2,3}{
\node[text centered,below] at (7-\x,-1) {\tiny $m_{\x}(\mu)$};
\node[text centered,above] at (7-\x,1) {\tiny $m_{\x}(\lambda)$};
}
\foreach\x in {4,5}{
\node[text centered,below] at (7-\x,-1) {\tiny $\cdots$};
\node[text centered,above] at (7-\x,1) {\tiny $\cdots$};
}
\end{tikzpicture}\right)&=I_{l(\lambda)=l(\mu)+1}Q_{\lambda/\mu}(a;q),
    \\w_a
\left(
\begin{tikzpicture}[baseline=(current bounding box.center),>=stealth,scale=0.8]
\draw[lred,ultra thick] (0,0) -- (7,0);
\foreach\x in {1,...,6}{
\draw[lred,line width=10pt] (7-\x,-1) -- (7-\x,1);
}
\node[left] at (0,0) {$1$};
\draw[ultra thick,->] (0,0) -- (1,0);
\draw[ultra thick,->] (6,0) -- (7,0);
\node[right] at (7,0) {$1$};
\foreach\x in {1,2,3}{
\node[text centered,below] at (7-\x,-1) {\tiny $m_{\x}(\mu)$};
\node[text centered,above] at (7-\x,1) {\tiny $m_{\x}(\lambda)$};
}
\foreach\x in {4,5}{
\node[text centered,below] at (7-\x,-1) {\tiny $\cdots$};
\node[text centered,above] at (7-\x,1) {\tiny $\cdots$};
}
\end{tikzpicture}\right)&=I_{l(\lambda)=l(\mu)}Q_{\lambda/\mu}(a;q).
\end{align*}
\endgroup
\end{lemma}
Both these results were already used in \cite{BBCW18} to study a special case of Theorem \ref{thm: general 6vm HL distr equality}. The key new tool to extend the methods of \cite{BBCW18} to show Theorem \ref{thm: general 6vm HL distr equality} is the following lemma, which extends Proposition 4.9 of \cite{BBCW18} and shows that the more general boundary for the half space six vertex model in this paper remains compatible with this deformed boson model. Note that due to the differences in the weights for the even and odd columns, we must pass the boundary vertices through two at a time.

\begin{lemma}
\label{lem: two vertex compat}
For any choice of $i,j=0,1$ and $n_1,n_2\in\N$, we have
\begin{multline*}
    \sum_{m_1,m_2=0}^{\infty}h_{m_2}(-t,q) (-t\nu)^{m_1}h_{m_1}\left(-\frac{1}{\nu^2 t},q\right)
w_a\left(
\begin{tikzpicture}[scale=0.7,>=stealth,baseline=(current bounding box.center)]
\draw[lgray,ultra thick] (-1.5,0) -- (2,0);
\draw[lgray,line width=10pt] (0,-1) -- (0,1);
\draw[lgray,line width=10pt] (1,-1) -- (1,1);
\node[below] at (0,-1) {$m_2$};
\node[above] at (0,1) {$n_2$};
\node[below] at (1,-1) {$m_1$};
\node[above] at (1,1) {$n_1$};
\node[left] at (-1.5,0) {$i$};
\node[right] at (2,0) {$j$};
\node at (-1,0) {$\bullet$};
\end{tikzpicture}
\right)
\\=\sum_{m_1,m_2=0}^{\infty} h_{m_2}(-t,q) (-t\nu)^{m_1}h_{m_1}\left(-\frac{1}{\nu^2 t},q\right) w_a\left(
\begin{tikzpicture}[scale=0.7,>=stealth,baseline=(current bounding box.center)]
\draw[lred,ultra thick] (-1,0) -- (2.5,0);
\draw[lred,line width=10pt] (0,-1) -- (0,1);
\draw[lred,line width=10pt] (1,-1) -- (1,1);
\node[below] at (0,-1) {$m_2$};
\node[above] at (0,1) {$n_2$};
\node[below] at (1,-1) {$m_1$};
\node[above] at (1,1) {$n_1$};
\node[left] at (-1,0) {$i$};
\node[right] at (2.5,0) {$j$};
\node at (2,0) {$\bullet$};
\end{tikzpicture}
\right).
\end{multline*}
\end{lemma}
\begin{proof}
Although the sums over $m_1$ and $m_2$ are infinite, for any fixed choice of $i$, $j$, $n_1$, and $n_2$, there are only finitely many possible choices which give a non-zero weight. Thus, we verify the equality directly for each possible choice of $i$, $j$, $n_1$, and $n_2$. It turns out that the only input needed is the recurrence for the Rogers--Szeg\"o polynomials, so this is an easy but fairly tedious computation.

This is done via a symbolic computation in SageMath. Since for fixed $n_1$ and $n_2$, $m_1\in \{n_1-1,n_1,n_1+1\}$ and $m_2\in \{n_2-1,n_2,n_2+1\}$, we simply view $h_n(x,q)$ as a formal variable and use the recurrence to write $h_{n_1+1}$ in terms of $h_{n_1}$ and $h_{n_1-1}$, and similarly for $h_{n_2+1}$. A bit of care is needed when $n_1=0$ or $n_2=0$, as the recurrence is only stated for $n\geq 1$, but remains valid for $n=0$ by defining $h_{-1}=0$ and extending the deformed boson model to allow $-1$ arrows to enter from below, with weight $0$. The code used to run this computation can be found in "symbolic check.sage".
\end{proof}

By iterating Lemma \ref{lem: two vertex compat}, we obtain the following proposition.
\begin{proposition}
\label{prop: two vertex compat}
Let $n_1,n_2,\dotsc$ be a finitely supported sequence of non-negative integers. We have
\begin{multline*}
\sum_{\lambda}h_\lambda(q,t,\nu)w_a\left(
\begin{tikzpicture}[baseline=(current bounding box.center),>=stealth,scale=0.8]
\node[left] (0,0) {$1$};
\draw[lgray,ultra thick] (0,0) -- (8,0);
\draw[ultra thick,->] (0,0) -- (1,0);
\foreach\x in {0,...,5}{
\draw[lgray,line width=10pt] (7-\x,-1) -- (7-\x,1);
}
\node at (1,0) {$\bullet$};
\node[right] at (8,0) {$j$};
\foreach\x in {1,2,3}{
\node[text centered,below] at (8-\x,-1) {\tiny $m_{\x}(\lambda)$};
\node[text centered,above] at (8-\x,1) {\tiny $n_{\x}$};
}
\foreach\x in {3,4}{
\node[text centered,below] at (7-\x,-1) {\tiny $\cdots$};
\node[text centered,above] at (7-\x,1) {\tiny $\cdots$};
}
\end{tikzpicture}
\right)
\\
=
\sum_{\lambda}h_\lambda(q,t,\nu)w_a
\left(
\begin{tikzpicture}[baseline=(current bounding box.center),>=stealth,scale=0.8]
\draw[lred,ultra thick] (0,0) -- (8,0);
\foreach\x in {1,...,6}{
\draw[lred,line width=10pt] (7-\x,-1) -- (7-\x,1);
}
\node[left] at (0,0) {$1$};
\node at (7,0) {$\bullet$};
\draw[ultra thick,->] (0,0) -- (1,0);
\draw[ultra thick,->] (6,0) -- (7,0);
\node[right] at (8,0) {$j$};
\foreach\x in {1,2,3}{
\node[text centered,below] at (7-\x,-1) {\tiny $m_{\x}(\lambda)$};
\node[text centered,above] at (7-\x,1) {\tiny $n_{\x}$};
}
\foreach\x in {4,5}{
\node[text centered,below] at (7-\x,-1) {\tiny $\cdots$};
\node[text centered,above] at (7-\x,1) {\tiny $\cdots$};
}
\end{tikzpicture}
\right),
\end{multline*}
where the sums are over partitions $\lambda$.
\end{proposition}
\begin{proof}
The sum over $\lambda$ is equivalent to a sum over finitely supported sequences $m_i$ of non-negative integers. We then apply Lemma \ref{lem: two vertex compat}, where we note that after restricting to $l(\lambda)\leq 2N$, we require only finitely many applications, and taking $N\to\infty$ gives the desired equality, as the term on the left hand side with an arrow entering the row converges to $0$, the left hand side converges since any arrow entering from the bottom needs to travel to the first non-zero $n_i$, giving geometric decay, and the right hand side is a finite sum.
\end{proof}
\begin{remark}
Note that Proposition \ref{prop: two vertex compat} only holds if an arrow is entering from the left. At the finite level, it does not matter whether an arrow enters or not, but there are issues with taking a limit of the right hand side since without an arrow to block potential arrows from entering at the bottom, arrows will want to enter from the bottom at the left hand side. Heuristically, one would want to allow arrows to enter from the bottom infinitely far to the left on the right hand side in order to make sense of this case.
\end{remark}

\subsection{Proof of Theorem \ref{thm: general 6vm HL distr equality}}
With the tools developed in the previous section, the proof of Theorem \ref{thm: general 6vm HL distr equality} follows the proof in \cite{BBCW18}, using a Yang--Baxter graphical argument.

\begin{proof}[Proof of Theorem \ref{thm: general 6vm HL distr equality}]
We have by Lemma \ref{lem: boson HL} that $\mathbb{HL}_a^{(q,t,\nu)}([\vec{\lambda}]=s)$ is equal to
\begin{equation*}
\frac{1}{\Pi(a;q,t,\nu)}\sum_{\lambda}h_\lambda(q,t,\nu)\begin{tikzpicture}[scale=0.8,baseline=(current bounding box.center),>=stealth]
\foreach\x in {0,...,6}{
\draw[lgray,line width=10pt] (\x,0) -- (\x,7);
}
\foreach\y in {1,...,6}{
\draw[lgray,thick] (-1,\y) -- (7,\y);
}
\draw[ultra thick,->] (6,1) -- (7,1); \draw[ultra thick,->] (6,2) -- (7,2);
\draw[ultra thick,->] (6,4) -- (7,4); \draw[ultra thick,->] (6,6) -- (7,6);
\node[right] at (7,1) {$s_n$}; \node[right] at (7,6) {$s_1$};
\node[left] at (-1,6) {$a_1$};
\node[left] at (-1,3.5) {$\vdots$};
\node[left] at (8,3.5) {$\vdots$};
\node[left] at (-1,1) {$a_n$};
\node[below] at (6,0) {\tiny $m_1(\lambda)$};
\node[below] at (5,0) {\tiny $m_2(\lambda)$};
\node[below] at (4,0) {\tiny $m_3(\lambda)$};
\node[below] at (2,0) {$\cdots$};
\node[above] at (6,7) {\tiny $0$};
\node[above] at (5,7) {\tiny $0$};
\node[above] at (4,7) {\tiny $0$};
\node[above] at (2,7) {$\cdots$};
\end{tikzpicture}
\end{equation*}
where the boundary conditions on the left are empty (the $a_i$ indicate the row rapidities) and the boundary conditions on the right are indicated by the $s_i$. We can introduce a boundary vertex at the left on the bottom row, at the cost of a $\frac{1-a_n^2}{(1-a_n\nu t)(1+a_n/\nu)}$ factor (which will cancel with a factor in $\Pi(a;q,t,\nu)$), since if an arrow enters the row the weight will be $0$ so only one outcome is possible. We then use Lemma \ref{prop: two vertex compat} to move the boundary vertex to the right, resulting in the expression
\begin{equation*}
    \begin{array}{l}
    \frac{1-a_n^2}{(1-a_n\nu t)(1+a_n/\nu)}
    \\
    \\\qquad\times \frac{1}{\Pi(a;q,t,\nu)}\sum_{\lambda}h_\lambda(q,t,\nu)
    \end{array}
    \begin{tikzpicture}[scale=0.8,baseline=(current bounding box.center),>=stealth]
\foreach\x in {0,...,6}{
\draw[lgray,line width=10pt] (\x,1) -- (\x,7);
}
\foreach\y in {2,...,6}{
\draw[lgray,thick] (-1,\y) -- (7,\y);
}
\foreach\x in {0,...,6}{
\draw[lred,line width=10pt] (\x,-1) -- (\x,1);
}
\foreach\y in {0}{
\draw[lred,thick] (-1,\y) -- (8,\y);
}
\node at (7,0) {$\bullet$};
\draw[ultra thick,->] (-1,0) -- (0,0);
\draw[ultra thick,->] (7,0) -- (8,0); \draw[ultra thick,->] (6,2) -- (7,2);
\draw[ultra thick,->] (6,4) -- (7,4); \draw[ultra thick,->] (6,6) -- (7,6);
\node[right] at (7,2) {$s_{n-1}$};
\node[right] at (8,0) {$s_n$}; \node[right] at (7,6) {$s_1$};
\node[left] at (-1,6) {$a_1$};
\node[left] at (-1,3.5) {$\vdots$};
\node[left] at (8,3.5) {$\vdots$};
\node[left] at (-1,2) {$a_{n-1}$};
\node[left] at (-1,0) {$a_n$};
\node[below] at (6,-1) {\tiny $m_1(\lambda)$};
\node[below] at (5,-1) {\tiny $m_2(\lambda)$};
\node[below] at (4,-1) {\tiny $m_3(\lambda)$};
\node[below] at (2,-1) {$\cdots$};
\node[above] at (6,7) {\tiny $0$};
\node[above] at (5,7) {\tiny $0$};
\node[above] at (4,7) {\tiny $0$};
\node[above] at (2,7) {$\cdots$};
\end{tikzpicture}
\end{equation*}

Next, we use Proposition \ref{prop: YB boson} to swap the bottom row all the way to the top, at the cost of a $\prod _i\frac{1-a_ia_n}{1-qa_ia_n}$ factor, coming from $\Pi(a;q,t,\nu)$, resulting in the expression
\begin{equation*}
    \begin{array}{l}
    \frac{1-a_n^2}{(1-a_n\nu t)(1+a_n/\nu)}\prod_{i<n}\frac{1-a_ia_n}{1-qa_ia_n}
    \\
    \\\qquad\times \frac{1}{\Pi(a;q,t,\nu)}\sum_{\lambda}h_\lambda(q,t,\nu)
    \end{array}
    \begin{tikzpicture}[scale=0.8,baseline=(current bounding box.center),>=stealth]
\foreach\x in {0,...,6}{
\draw[lgray,line width=10pt] (\x,-1) -- (\x,5);
}
\foreach\y in {0,...,4}{
\draw[lgray,thick] (-1,\y) -- (7,\y);
}
\foreach\x in {0,...,6}{
\draw[lred,line width=10pt] (\x,5) -- (\x,7);
}
\foreach\y in {6}{
\draw[lred,thick] (-1,\y) -- (7,\y);
}
\draw[thick, dotted] (7,6) -- (10.5,2.5);
\node at (10.5,2.5) {$\bullet$}; \draw[ultra thick,->] (10.5,2.5) -- (11,3);
\draw[thick, dotted] (7,4) -- (8.5,5.5) node[above right] {$s_1$};
\draw[thick, dotted] (7,3) -- (9,5);
\draw[thick, dotted] (7,2) -- (9.5,4.5);
\draw[thick, dotted] (7,1) -- (10,4);
\draw[thick, dotted] (7,0) -- (10.5,3.5) node[above right] {$s_{n-1}$};
\draw[thick, dotted] (10.5,2.5) -- (11,3) node[above right] {$s_n$};
\draw[ultra thick,->] (8,5) -- (8.5,5.5); \draw[ultra thick,->] (9,4) -- (9.5,4.5);
\draw[ultra thick,->] (10,3) -- (10.5,3.5);
\draw[ultra thick,->] (-1,6) -- (0,6);
\node[left] at (-1,4) {$a_1$};
\node[left] at (-1,1.5) {$\vdots$};
\node[left] at (-1,0) {$a_{n-1}$};
\node[left] at (-1,6) {$a_n$};
\node[below] at (6,-1) {\tiny $m_1(\lambda)$};
\node[below] at (5,-1) {\tiny $m_2(\lambda)$};
\node[below] at (4,-1) {\tiny $m_3(\lambda)$};
\node[below] at (2,-1) {$\cdots$};
\node[above] at (6,7) {\tiny $0$};
\node[above] at (5,7) {\tiny $0$};
\node[above] at (4,7) {\tiny $0$};
\node[above] at (2,7) {$\cdots$};
\end{tikzpicture}
\end{equation*}

Repeating this process, we obtain
\begin{equation*}
    \sum_\lambda h_\lambda(q,t,\nu)\begin{tikzpicture}[scale=0.8,baseline=(current bounding box.center),>=stealth]
\foreach\x in {0,...,6}{
\draw[lred,line width=10pt] (\x,0) -- (\x,7);
}
\foreach\y in {1,...,6}{
\draw[lred,thick] (-1,\y) -- (7,\y);
\draw[ultra thick,->] (-1,\y) -- (0,\y);
}
\draw[thick, dotted] (7,6) -- (12.5,0.5); \node at (12.5,0.5) {$\bullet$};
\draw[thick, dotted] (12.5,0.5) -- (13,1) node[above right] {$s_n$}; \draw[ultra thick,->] (12.5,0.5) -- (13,1);
\draw[thick, dotted] (7,5) -- (11.5,0.5); \node at (11.5,0.5) {$\bullet$};
\draw[thick, dotted] (11.5,0.5) -- (12.5,1.5); \draw[ultra thick,->] (12,1) -- (12.5,1.5);
\draw[thick, dotted] (7,4) -- (10.5,0.5); \node at (10.5,0.5) {$\bullet$};
\draw[thick, dotted] (10.5,0.5) -- (12,2);
\draw[thick, dotted] (7,3) -- (9.5,0.5); \node at (9.5,0.5) {$\bullet$};
\draw[thick, dotted] (9.5,0.5) -- (11.5,2.5); \draw[ultra thick,->] (11,2) -- (11.5,2.5);
\draw[thick, dotted] (7,2) -- (8.5,0.5); \node at (8.5,0.5) {$\bullet$};
\draw[thick, dotted] (8.5,0.5) -- (11,3);
\draw[thick, dotted] (7,1) -- (7.5,0.5); \node at (7.5,0.5) {$\bullet$};
\draw[thick, dotted] (7.5,0.5) -- (10.5,3.5) node[above right] {$s_1$}; \draw[ultra thick,->] (10,3) -- (10.5,3.5);
\node[left] at (-1,6) {$a_n$};
\node[left] at (-1,3.5) {$\vdots$};
\node[left] at (-1,1) {$a_1$};
\node[below] at (6,0) {\tiny $m_1(\lambda)$};
\node[below] at (5,0) {\tiny $m_2(\lambda)$};
\node[below] at (4,0) {\tiny $m_3(\lambda)$};
\node[below] at (2,0) {$\cdots$};
\node[above] at (6,7) {\tiny $0$};
\node[above] at (5,7) {\tiny $0$};
\node[above] at (4,7) {\tiny $0$};
\node[above] at (2,7) {$\cdots$};
\end{tikzpicture}
\end{equation*}
where we note that the resulting constant after performing all these operations exactly cancels the $\Pi(a;q,t,\nu)$ factor.

Since the deformed boson portion is frozen as all arrows entering from the left must continue rightwards since no arrows exit from the top, the only non-zero summand is $\lambda=\emptyset$, for which the weight of the deformed boson portion is $1$. Moreover, the remaining portion is exactly the half space six vertex model where the boundary condition has arrows entering from the left, and so is equal to $\PP(S=(s_1,\dotsc,s_n))$.
\end{proof}

\section{An identity of symmetric functions}
\label{sec: identity}
The goal of this section is to prove the following symmetric function identity. This will be used in the following section to relate the half space Hall--Littlewood measure to the free boundary Schur measure. Recall that $o(\lambda)$ denotes the number of odd parts in $\lambda$.

\begin{theorem}
\label{thm: general identity}
We have as an identity in the ring of symmetric functions,
\begin{equation*}
    \sum_{l=0}^k\frac{q^lh_l(z^2/\sqrt{q};q)}{(q;q)_l}\sum_{\lambda: \lambda_1=k-l}h'_\lambda(q,z,w)\widehat{Q}_\lambda(x;q)=\sum_{\lambda,\rho:\lambda_1=k}z^{o(\lambda')+o(\rho')}w^{o(\lambda')-o(\rho')}q^{|\rho|/2}s_{\lambda/\rho}(x),
\end{equation*}
where
\begin{equation*}
    h_\lambda'(q,z,w)=\prod_{i \text{ even}}h_{\lambda_i-\lambda_{i-1}}(\sqrt{q}z^2)\prod_{i\text{ odd}}(zw)^{\lambda_i-\lambda_{i-1}}h_{\lambda_i-\lambda_{i-1}}(\sqrt{q} w^{-2}).
\end{equation*}
\end{theorem}

The proof of Theorem \ref{thm: general identity} follows that of Theorem 10.12 in \cite{IMS21}, who proved the $w=1$ case. We now introduce some of the results from \cite{IMS21} needed for the proof.

\subsection{Sagan--Stanley skew RSK correspondence}
A \emph{weighted biword} is a triple $\overline{\pi}=(\pi^{(1)},\pi^{(2)},\pi^{(3)})$ of words of the same (finite) length (with alphabets $\N$), such that the words are lexicographically ordered, with the reverse order for $\pi^{(3)}$ (i.e. $\pi^{(1)}_i\leq \pi^{(1)}_{i+1}$, and if they are equal, $\pi^{(2)}_i\leq \pi^{(2)}_{i+1}$, and if they are also equal, then $\pi^{(3)}_i\geq \pi^{(3)}_{i+1}$). We refer to $\pi^{(3)}_i$ as the \emph{weight} of that position. We let $\overline{\mathbb{A}}_n^+$ denote the set of weighted biwords where $\pi^{(1)}$ and $\pi^{(2)}$ use the alphabet $\{1,\dotsc, n\}$, and $\pi^{(3)}_i\geq 0$ for all $i$.

A \emph{biword} is simply a weighted biword where all positions have weight $0$, in which case we can forget about the weights. We will let $\pi$ denote the corresponding biword obtained by forgetting $\pi^{(3)}$. 

We define the transpose $\overline{\pi}^T$ by swapping $\pi^{(1)}$ and $\pi^{(2)}$, and sorting the entries of all three words to satisfy the above constraints. We say that $\overline{\pi}$ is \emph{symmetric} if $\overline{\pi}^T=\overline{\pi}$, and analogously for unweighted biwords. 

We let $w_i(\overline{\pi})$ denote the biword given by taking the subword of $\pi$ positions of weight $i$. A fixed point of $\overline{\pi}$ is a position $i$ such that $\pi^{(1)}_i=\pi^{(2)}_i$, and we let $\fix(\overline{\pi})$ denote the number of fixed points, and $\fix^-(\overline{\pi})=\sum_{i=0}^m (-1)^{i} \fix(w_i(\overline{\pi}))$ denote the signed count, where $m$ is the maximal weight appearing in $\overline{\pi}$. We let $\wt(\overline{\pi})=\sum_i \pi^{(3)}_i$ denote the sum of the weights in $\overline{\pi}$.

The \emph{Sagan--Stanley skew RSK correspondence}, introduced in \cite{SS90}, is a bijection between pairs $(\overline{\pi},\nu)$ of a symmetric weighted biword and a partition $\nu$, and a tableaux $P$. We refer the reader to \cite{SS90} for its definition and more details, and here we simply summarize its key properties that we will require. 

It is obtained via iterating an insertion operation on biwords, once for each weight starting from the maximal weight appearing in $\overline{\pi}$, and ending at $0$. This insertion operation takes a symmetric biword $\pi=(\pi^{(1)},\pi^{(2)})$ and a skew tableau $P$ of shape $\lambda/\rho$, and outputs a skew tableau $Q$ of shape $\alpha/\lambda$, where the numbers appearing in $Q$ are the union of the numbers appearing in $P$ and $\pi^{(1)}$ (and by symmetry also $\pi^{(2)}$).  The procedure begins with the empty tableaux on the shape $\nu/\nu$. Note that even if the biword is empty, this insertion operation results in a different output, unlike the usual RSK insertion operation. We will label the sequence of shapes for the tableau which appear by $\lambda^{(i)}/\lambda^{(i-1)}$, starting from $\lambda^{(-1)}/\lambda^{(-2)}=\nu/\nu$ and ending at $\lambda^{(m)}/\lambda^{(m-1)}$, which is the shape of the final tableaux after all insertions are completed.

\begin{lemma}
\label{lem: signed SS}
If the Stanley--Sagan correspondence sends $(\overline{\pi},\nu)$ to $T$, where $T$ has shape $\lambda/\rho$ and $\pi$ is a weighted biword with maximal weight $m$, then
\begin{equation*}
\begin{split}
    \fix(\overline{\pi})+2o(\nu')&=o(\lambda')+o(\rho'),
    \\\fix^-(\overline{\pi})&=o(\lambda')-o(\rho').
\end{split}
\end{equation*}
\end{lemma}
\begin{proof}
The first equality is given by Corollary 4.6 of \cite{SS90}, and the proof is similar to that of the second equality given below.

From the definition of the Sagan--Stanley correspondence, we have that $S$ is obtained by starting from the empty diagram $\nu$, and repeatedly applying the basic skew RSK insertion on the subwords of $\pi$ by weight. The basic skew RSK insertion is a correspondence $(\pi,S)\Longleftrightarrow T$ where $S$ has shape $\alpha/\rho$ and $T$ has shape $\lambda/\alpha$. This satisfies $\fix(\pi)+o(\rho')=o(\lambda')$ (see Corollary 3.6 of \cite{SS90}). In our situation, if $\lambda^{(i)}$ are the outer shapes under applying RSK insertion, we have that we start with $\alpha=\rho=\nu$, and have $\fix(w_{m-i}(\overline{\pi}))+o((\lambda^{(i-2)})')=o((\lambda^{(i)})')$, where $\lambda^{(-1)}=\lambda^{(-2)}=\nu$. Then summing these with a factor of $(-1)^{m-i}$ gives $\sum_i \fix(w_i(\overline{\pi}))(-1)^{i}=o((\lambda^{(m)})')-o((\lambda^{(m-1)})')$, which is exactly what we want, since the final shape is $\lambda/\rho=\lambda^{(m)}/\lambda^{(m-1)}$.
\end{proof}

\subsection{A bijection of Imamura--Mucciconi--Sasamoto}
We now introduce the iterated skew RSK correspondence introduced in \cite{IMS21} which we need for the proof. Since we will not need the details of the construction, we will treat the bijection as a black box, and simply state the needed properties, all of which are proved in \cite{IMS21}, and we refer the reader there for more details.

A \emph{vertically strict tableaux (VST)} of shape $\mu$ is a filling of $\mu$ which is strictly increasing in the columns, and has no condition on the rows. We denote the set of VSTs of shape $\mu$ by $VST(\mu)$. We will require certain statistics on VSTs called the \emph{intrinsic energy}, denoted $\mathcal{H}(V)$, and the local energies, $\mathcal{H}_i(V)$, for a VST $V$. The exact definitions are unimportant, so we refer the reader to \cite{IMS21} for details. We also require a statistic, $\mu(\overline{\pi})$ called the \emph{Greene invariant}, for which we again refer the reader to \cite{IMS21} for the definition. Before we discuss the bijection itself, we record the following fact about the Sagan--Stanley RSK correspondence. It follows immediately from Theorem 1.2 and Corollary 1.3 of \cite{IMS21}.
\begin{lemma}
\label{lem: greene skew RSK}
If $(\overline{\pi},\nu)$ corresponds to $T$ under the Stanley--Sagan correspondence, where $T$ has shape $\lambda/\rho$, then
\begin{equation*}
    \lambda_1=\nu_1+\mu(\overline{\pi})_1.
\end{equation*}
\end{lemma}

Now we turn to the bijection itself. For a partition $\mu$, we define the set
\begin{equation*}
    \mathcal{K}(\mu)=\{(\kappa_1,\dotsc, \kappa_{\mu_1})\mid \kappa_i\geq \kappa_{i+1}\text{ if }\mu'_i=\mu'_{i+1}\}.
\end{equation*}
This should be thought of as a collection of partitions $\kappa^{(i)}$ "hanging off of" $\mu$, in the sense that the length of $\kappa^{(i)}$ is $\mu_i-\mu_{i+1}$. In \cite{IMS21}, a bijection $\widetilde{\Upsilon}$ is given between symmetric weighted biwords with $\mu(\overline{\pi})=\mu$ and pairs $(V,\kappa)$ with $V\in VST(\mu)$ and $\kappa\in \mathcal{K}(\mu)$. First, use the Sagan--Stanley skew RSK correspondence to map $(\overline{\pi},\emptyset)$ to a tableau $P_1$. We then define dynamics by defining $P_{t+1}$ to be obtained from $P_t$ via insertion of an empty word into $P_t$. This means that under the skew RSK correspondence, $P_t$ corresponds to $(\overline{\pi}^{(t)},\emptyset)$, where $\overline{\pi}^{(t)}$ is simply $\overline{\pi}$ with the weights shifted up by $t-1$ (this has the effect of causing $t-1$ weights with no entries starting from $0$, which means we end with $t-1$ many empty insertions). We then obtain $V$ and $\kappa$ via studying the asymptotics of $P_t$ as $t\to\infty$. Their exact definitions are a little complicated, so we refer the reader to \cite{IMS21} and just state the needed results.

The following result is essentially given by Corollary 8.2 of \cite{IMS21}, together with the observation that since the Sagan--Stanley skew RSK correspondence restricts to a bijection between symmetric biwords and a single tableau, the map $\widetilde{\Upsilon}$, which is defined via iterating the Sagan--Stanley correspondence, does as well.

\begin{theorem}[{\hspace{1sp}\cite[Corollary 8.2]{IMS21}}]
\label{thm: IMS bij}
The map $\widetilde{\Upsilon}$ is a bijection between symmetric weighted biwords with weights in $\{0,\dotsc, n\}$ and $\mu(\overline{\pi})=\mu$, and pairs $(V,\kappa)$, where $V\in VST(\mu)$ and $\kappa\in \mathcal{K}(\mu)$, such that
\begin{equation*}
    \wt(\overline{\pi})=2\mathcal{H}(V)+|\kappa|.
\end{equation*}
\end{theorem}

The following lemma is a restatement of Equations 10.17 and 10.18 in \cite{IMS21}.
\begin{lemma}
\label{lem: skew rsk asymp}
Let $P_t$ denote the skew RSK dynamics defining $\widetilde{\Upsilon}$ with $P_0$ corresponding to $(\overline{\pi},\emptyset)$ under the skew RSK correspondence. Let $\lambda^{(t)}/\rho^{(t)}$ denote the shape of $P_t$. Then for large enough $t$, we have
\begin{equation*}
\begin{split}
    (\lambda^{(t)})'_i&=2\mathcal{H}_i(V)+\kappa_i+t\mu_i',
    \\(\rho^{(t)})'_i&=2\mathcal{H}_i(V)+\kappa_i+(t+1)\mu_i',
\end{split}
\end{equation*}
where $\widetilde{\Upsilon}(\overline{\pi})=(V,\kappa)$.
\end{lemma}

\begin{lemma}
\label{lem: signed fix}
If $\mu(\overline{\pi})=\mu$ and $\widetilde{\Upsilon}(\overline{\pi})=(V,\kappa)$, we have
\begin{equation*}
\begin{split}
    \fix(\overline{\pi})&=o(\kappa)+o(\kappa+\mu')
    \\\fix^-(\overline{\pi})&=-o(\kappa)+o(\kappa+\mu').
\end{split}
\end{equation*}
\end{lemma}
\begin{proof}
The first statement is given by Lemma 10.5 of \cite{IMS21}, and we follow their proof closely for the second statement.

Consider the skew RSK dynamics $(P_t,P_t)$ associated to the initial data $(P,P)$, where $P$ corresponds to $(\overline{\pi},\emptyset)$ under the Sagan--Stanley correspondence, and let $\lambda^{(t)}/\rho^{(t)}$ denote the shape of $P_t$. We let $\overline{\pi}^{(t)}$ correspond to $P_t$ in the same manner. Then $\overline{\pi}^{(t)}$ is given by $\overline{\pi}$ with the weights shifted by $t-1$. Thus, $\fix^-(\overline{\pi}^{(t)})=(-1)^{t-1}\fix^-(\overline{\pi})$. Now by Lemma \ref{lem: signed SS}, this implies that $\fix^-(\overline{\pi})=(-1)^{t-1}(o((\lambda^{(t)})')-o((\rho^{(t)})'))$.

On the other hand, by Lemma \ref{lem: skew rsk asymp}, we have for large enough $t$,
\begin{equation*}
\begin{split}
    (\lambda^{(t)})'_i&=2\mathcal{H}_i(V)+\kappa_i+t\mu_i'
    \\(\rho^{(t)})'_i&=2\mathcal{H}_i(V)+\kappa_i+(t+1)\mu_i'.
\end{split}
\end{equation*}
Thus, when $t$ is even, $o((\lambda^{(t)})')=o(\kappa)$ and $o((\rho^{(t)})')=o(\kappa+\mu')$, and when $t$ is odd, the opposite occurs, and in both cases our desired equality holds.
\end{proof}

\subsection{Proof of Theorem \ref{thm: general identity}}
We require the following expansions. First, recall that $s_{\lambda/\mu}(x_1,\dotsc, x_n)=\sum_{P\in SST(\lambda/\mu,n)}x^P$, where $SST(\lambda/\mu,n)$ denotes the set of \emph{semistandard Young tableau} of shape $\lambda/\mu$, filled with numbers $1,\dotsc, n$, and $x^P$ is the product of a factor of $x_i$ for each occurrence of $i$ in $P$.

We let $\widehat{P}_\mu(x;q)=\prod_{i}(q;q)_{\mu_i-\mu_{i+1}}\widehat{Q}_\mu(x;q)$.
\begin{lemma}[{\hspace{1sp}\cite[Proposition 10.1]{IMS21}}]
\label{lem: q-whit expansion}
We have
\begin{equation*}
    \widehat{P}_\mu(x;q)=\sum_{V\in VST(\mu)}q^{\mathcal{H}(V)}x^V.
\end{equation*}
\end{lemma}
\begin{lemma}[{\hspace{1sp}\cite[Lemma 10.6]{IMS21}}]
\label{lem: rogers-szego expansion}
We have
 \begin{equation*}
     \frac{q^{n}h_n(z^2/\sqrt{q};q)}{(q;q)_n}=\sum_{\nu:\nu_1=n}q^{|\nu|/2}z^{2o(\nu')}
 \end{equation*}
 and
\begin{equation*}
    \frac{h_n(\sqrt{q}z^2;q)}{(q;q)_n}=\sum_{\nu:\nu_1\leq n}q^{|\nu|/2}z^{2o(\nu')}.
\end{equation*}
\end{lemma}

\begin{proof}[Proof of Theorem \ref{thm: general identity}]
We follow the proof of Theorem 10.12 in \cite{IMS21}. Note that to prove the identity in the ring of symmetric functions, it suffices to prove the identity for an arbitrarily large finite collection of variables $x_1,\dotsc, x_n$, which we fix. Then
\begin{equation*}
    \begin{split}
        &\sum_{\lambda,\rho:\lambda_1=k}z^{o(\lambda')+o(\rho')}w^{o(\lambda')-o(\rho')}q^{|\rho|/2}s_{\lambda/\rho}(x)
        \\=&\sum_{\lambda,\rho:\lambda_1=k}\sum _{P\in SST(\lambda/\mu,n)}z^{o(\lambda')+o(\rho')}w^{o(\lambda')-o(\rho')}q^{|\rho|/2}x^P
        \\=&\sum_{\substack{\mu,\nu\\\mu_1+\nu_1=k}}\sum_{\substack{\overline{\pi}\in \overline{\mathbb{A}}_n^+\\\overline{\pi}^{-1}=\overline{\pi}, \mu(\overline{\pi})=\mu}}q^{|\nu|/2+\wt(\overline{\pi})/2}z^{\fix(\overline{\pi})+2o(\nu')}w^{\fix^-(\overline{\pi})}x^{p(\overline{\pi})}
        \\=&\sum_{l=0}^k\left(\sum_{\nu: \nu_1=l}q^{|\nu|/2}z^{2o(\nu')}\right)
        \\&\qquad \times\sum_{\mu: \mu_1=k-l}\left(\sum_{\kappa\in \mathcal{K}(\mu)}q^{|\kappa|/2}z^{o(\kappa)+o(\kappa+\mu')}w^{-o(\kappa)+o(\kappa+\mu')}\right)\left(\sum_{V\in VST(\mu)}q^{\mathcal{H}(V)}x^V\right),
    \end{split}
\end{equation*}
where the first equality is the usual expansion of skew Schur functions in terms of semistandard Young tableau, the second equality is given by the Sagan--Stanley skew RSK correspondence along with Lemmas \ref{lem: greene skew RSK} and \ref{lem: signed SS}, and the third equality is given by the bijection $\widetilde{\Upsilon}$ of Theorem \ref{thm: IMS bij} along with Lemma \ref{lem: signed fix}.

We then decompose $\kappa$ into partitions $\kappa^{(i)}$ and rewrite the sum over $\kappa$ as a product over $i$. Note that if $i$ is odd, then $o(\kappa^{(i)}+(i^{\mu_i-\mu_{i+1}}))=e(\kappa^{(i)})$ and if $i$ is even, $o(\kappa^{(i)}+(i^{\mu_i-\mu_{i+1}}))=o(\kappa^{(i)})$ (with the caveat that $\kappa^{(i)}$ has exactly $\mu_i-\mu_{i+1}$ rows, and so $0$'s are counted). Then the even factors are given by
\begin{equation*}
    \sum_{\kappa^{(i)}, l(\kappa^{(i)})\leq \mu_i-\mu_{i+1}}q^{|\kappa^{(i)}|/2}z^{2o(\kappa^{(i)})}
\end{equation*}
and the odd factors are given by
\begin{equation*}
    (zw)^{\mu_i-\mu_{i+1}}\sum_{\kappa^{(i)}, l(\kappa^{(i)})\leq \mu_i-\mu_{i+1}}q^{|\kappa^{(i)}|/2}w^{-2o(\kappa^{(i)})}.
\end{equation*}
Finally, we can compute all these quantities using Lemmas \ref{lem: rogers-szego expansion} and \ref{lem: q-whit expansion}, as well as the identity $\prod \frac{1}{(q;q)_{\mu_i-\mu_{i+1}}}\widehat{P}_\mu(x;q)=\widehat{Q}_\mu(x;q)$. The desired equality immediately follows.
\end{proof}

\begin{remark}
The special case of Theorem \ref{thm: general identity} when $w=1$ was given in \cite{IMS21}. Our proof follows the same ideas, the only difference being that we keep track of an extra statistic $o(\lambda')-o(\rho')$ through the variable $w$. Summing over $k$ gives the Littlewood identity of Warnaar \cite[Theorem 1.1]{W06}, and the above proof thus gives a fully bijective proof (the special case $w=1$ is Theorem 10.3 in \cite{IMS21}). This identity is also somewhat similar to certain bounded Littlewood identities proven in \cite{RW21}. It is unclear if there there is any direct relationship, and this seems to be an interesting avenue for further exploration.
\end{remark}

\section{Fredholm Pfaffian formulas}
\label{sec: pfaff}
In this section, we obtain Fredholm Pfaffian formulas for the half space six vertex model and the half space ASEP. This is done via the free boundary Schur process, which is introduced and related to these models in this section. However, while Fredholm Pfaffian formulas for the free boundary Schur process are available in \cite{BBNV18}, they require restrictions on the parameters, which require analytic continuation to lift. We will deal with this in later sections.

\subsection{Free boundary Schur process}
\label{sec: FBS}
We recall that $s_\lambda$ and $p_\lambda$ denote the Schur and power sum symmetric functions respectively. Let $a,b$ denote specializations, and let $q,t,\nu,z$ be parameters. We will let
\begin{equation*}
    H(a;b)=\exp\left(\sum_{n\geq 1}\frac{1}{n}p_n(a)p_n(b)\right),
\end{equation*}
and
\begin{equation*}
    \widetilde{H}(a)=\exp\left(\sum_{n\geq 1}\frac{p_{2n-1}(a)}{2n-1}+\frac{p_n(a)^2}{2n}\right).
\end{equation*}
For later use, we will note that if $a=(a_1,\dotsc)$, then
\begin{equation}
\label{eq: H spec}
    H(\widehat{a};z)=\prod\frac{1+a_i z}{1+qa_i z}.
\end{equation}

We let
\begin{equation*}
    \Phi(a;q,t,\nu)=\frac{1}{(q;q)_\infty(-t;q)_\infty}\prod_{n\geq 0}\widetilde{H}(q^na) H(q^na;\{\nu t,-\nu^{-1}\}/(1-q)),
\end{equation*}
where $\{\nu t,-\nu^{-1}\}/(1-q)$ is defined to be the specialization
\begin{equation*}
    \{\nu t,-\nu^{-1}\}/(1-q)=(\nu t,-\nu^{-1},q\nu t,-q\nu^{-1},q^2\nu t,-q^2\nu^{-1},\dotsc).
\end{equation*}
With the specialization $\widehat{a}$ given by $p_n(\widehat{a})=(-1)^{n-1}(1-q^n)p_n(a)$, which will ultimately correspond to the six vertex model and the ASEP, we have
\begin{equation*}
    \Phi(\widehat{a};q,t,\nu)=\frac{1}{(q;q)_\infty(-t;q)_\infty}\Pi(a;q,t,\nu),
\end{equation*}
where we recall that $\Pi$ is the partition function for the half space Hall--Littlewood measure.

For convenience, we will introduce parameters $\gamma_1=\nu^{-1}\sqrt{q}^{-1}$ and $\gamma_2=-t\nu$, and let $a$ be a specialization. The \emph{free boundary Schur measure} $\mathbb{FBS}_a^{(q,t,\nu)}$ is the probability on partitions defined by
\begin{equation*}
 \mathbb{FBS}_a^{(q,t,\nu)}(\lambda)=\frac{1}{\Phi(a;q,t,\nu)}\sum_{\rho}\gamma_1^{o(\rho')}\gamma_2^{o(\lambda')}(q^{1/2})^{|\rho|}s_{\lambda/\rho}(a).
\end{equation*}
This was introduced in a more general fashion in \cite{BBNV18}. 
The fact that $\Phi(a;q,t,\nu)$ is the correct normalization is shown in \cite{BBNV18}. Alternatively, Theorem \ref{thm: general identity} actually gives a new proof of this.

\begin{remark}
Note that due to the sign, the free boundary Schur process we consider is not really a probability measure on partitions. However, the marginal $\lambda_1$ is a genuine random variable, at least when $t\leq q$ (this can be seen from Theorems \ref{thm: 6vm HL ident main} and \ref{thm: HL FBS same}). Moreover, as long as $t<1$ and $\nu^{-1}<1$, the definition gives an absolutely summable signed probability measure, and the expression $\PP(\lambda_1\leq s)=\sum_{x\leq s}\mathbb{FBS}_a^{(q,t,\nu)}(x)$ is analytic in $t$ and $\nu$. Thus, the formulas we will derive for $\PP(\lambda_1\leq s)$, which are also analytic functions of $t$ and $\nu$, will continue to apply even when $\lambda_1$ is not a random variable.
\end{remark}

Recall that $\chi\sim RS(q,t)$ if
\begin{equation*}
    \PP(\chi=k)=(q;q)_\infty(-t;q)_\infty\frac{q^kh_k(-t/q;q)}{(q;q)_k}.
\end{equation*}
Note that this always sums to $1$ by the $q$-binomial theorem. If $t\leq q$, then this defines a genuine probability distribution, since the probabilities are always non-negative. If $1>t>q$, then these probabilities may be negative, but still decay as $k\to\infty$. For a random variable $X$, we will continue to write $\PP(X+\chi\leq s)=\sum_{s_1+s_2\leq s}\PP(\chi=s_1)\PP(X\leq s_2)$ even if $1>t>q$. We note that there will be no convergence issues since $\PP(\chi=s)$ is absolutely summable if $t<1$. When $t<1$, one can recover the distribution of $X$ from $\PP(X+\chi\leq s)$, but when $t\geq 1$, this distribution no longer decays, and it is not clear whether one could extract asymptotic information after convolution against it. See Section \ref{sec: conv inv} for details.

We now give a proof of Theorem \ref{thm: HL FBS same}, which follows easily from Theorem \ref{thm: general identity}.
\begin{proof}[Proof of Theorem \ref{thm: HL FBS same}]
This follows immediately from Theorem \ref{thm: general identity}, taking $zw=-t\nu$ and $z/w=(\nu\sqrt{q})^{-1}$. In particular, since the identity holds at the level of symmetric functions, we can simply apply the involution $p_n(a)\mapsto p_n(\widehat{a})$ which takes $\widehat{Q}_{\lambda}(a;q)$ to $P_{\lambda'}(a;q)$. In particular, this swaps the roles of $\lambda_1$ and $l(\lambda)$.
\end{proof}

\subsection{Fredholm Pfaffian formulas}
We now present Fredholm Pfaffian formulas for the free boundary Schur process, and thus the half space six vertex model and ASEP. For more details on Fredholm Pfaffians, see Appendix \ref{sec: app}. Let $\zeta\in (0,1)$ be an auxiliary parameter. We write $S\sim \Thetad(\zeta^2,q^2)$ to indicate that $S$ is distributed according to
\begin{equation*}
    \PP(S=k)=\frac{q^{k^2}\zeta^{2k}}{\theta_3(\zeta^2;q^2)}.
\end{equation*}
Here, $\theta_3(\zeta;q)=(q;q)_\infty(-\sqrt{q}\zeta;q)_\infty(-\sqrt{q}/\zeta;q)_\infty$ is the Jacobi theta function, and the fact that this sums to $1$ is the Jacobi triple product identity.

We will let $D_x^{(\gamma_2)}$ denote the difference operator defined by
\begin{equation}
\label{eq: diff op}
    D_x^{(\gamma_2)}f(x)=\frac{(1-\gamma_2)^2}{2}\sum_{i=1}^\infty \gamma_2^{i-1}(f(x+i)-f(x-i))
\end{equation}
in the (discrete) variable $x$. Note that this definition converges absolutely for any function with a bound of the form $|f(x)|\leq C \gamma^{-|x|}$ for some $\gamma>|\gamma_2|$ (and in particular can be defined even if the function does not decay).

\begin{theorem}
\label{thm: pfaff formula}
Assume that $|\gamma_1\sqrt{q}|<1$, $|\gamma_2|<1$, and that $p_n(a)=O(c^n)$ for some $c<1$. Let $\lambda\sim \mathbb{FBS}_{a}^{(q,t,\nu)}$, and let $S\sim \Thetad(\zeta^2,q^2)$. Then if $1<r'<r$ are close enough to $1$, 
\begin{equation*}
    \PP(\lambda_1+2S\leq s)=\Pf(J-K)_{l^2(\Z_{>s})},
\end{equation*}
where the correlation kernel $K$ is given by
\begin{equation*}
    K=\begin{pmatrix}
    k(x,y)&-2D^{(\gamma_2)}_y k(x,y)
    \\-2D^{(\gamma_2)}_x k(x,y)& 4D^{(\gamma_2)}_x D^{(\gamma_2)}_y k(x,y)+(D^{(\gamma_2)}_x-D^{(\gamma_2)}_y)\Delta(x,y),
    \end{pmatrix}
\end{equation*}
where $\Delta(x,y)$ is the indicator that $x=y$, $D_x^{(\gamma_2)}$ is defined by \eqref{eq: diff op},
\begin{equation*}
    k(x,y)=\frac{(1-\gamma_2)^{-2}}{(2\pi i)^2}\oint_{|z|=r}\oint_ {|w|=r'} z^{-x-3/2}w^{-y-5/2}F(z)F(w)K(z,w)dzdw
\end{equation*}
with
\begin{equation*}
    F(z)=\prod_{n\geq 1}\frac{H(q^na;z)}{H(q^na;z^{-1})}
\end{equation*}
and
\begin{equation*}
    K(z,w)=\frac{(q,q,w/z,qz/w;q)_\infty}{(1/z^2,1/w^2,1/zw,qwz;q)_\infty}\frac{\theta_3(\zeta^2z^2w^2;q^2)}{\theta_3(\zeta^2;q^2)}g(z)g(w),
\end{equation*}
with
\begin{equation*}
    g(z)=\frac{(\gamma_1\sqrt{q}/z,\gamma_2/z;q)_\infty}{(\gamma_1\sqrt{q}z,\gamma_2qz;q)_\infty}.
\end{equation*}
\end{theorem}
\begin{proof}
This is essentially given in \cite{BBNV18, BBNV20} and the manipulations follow those in Theorem 4.5 of \cite{IMS22}. We start with \cite[Theorem 2.5]{BBNV18} (see Theorem 2.10 along with Section 3.3.2 in the preprint version for the most general statement needed in our setting), although see also \cite{BBNV20}, with the specializations $u=\sqrt{q}$, $v=1$, $\gamma_1=\nu^{-1}/\sqrt{q}$, $\gamma_2=-t\nu$, $\rho^+=a$, and $\rho^-=\emptyset$, which gives after some elementary manipulations (including multiplying first/second row and column by $(1-\gamma_2)^{-1}$ and $1-\gamma_2$ respectively) the formulas
\begin{equation*}
\begin{split}
    K_{11}(x,y)&=\frac{(1-\gamma_2)^{-2}}{(2\pi i)^2}\oint_{|z|=r}\oint_ {|w|=r'}z^{-x-3/2}w^{-y-5/2}F(z)F(w)K(z,w)dzdw
    \\K_{12}(x,y)&=\frac{1}{(2\pi i)^2}\oint_{|z|=r}\oint_ {|w|=r'}z^{-x-3/2}w^{y-1/2}\frac{F(z)}{F(w)}K(z,1/w)\frac{(1-w^2)}{(1-\gamma_2w)(1-\gamma_2/w)}dzdw
    \\K_{22}(x,y)&=\frac{(1-\gamma_2)^2}{(2\pi i)^2}\oint_{|z|=r}\oint_ {|w|=r'}z^{x-3/2}w^{y-1/2}\frac{1}{F(z)F(w)}K(1/z,1/w)
    \\&\qquad\qquad\qquad\qquad\qquad\qquad \times\frac{(1-w^2)(1-z^2)}{(1-\gamma_2w)(1-\gamma_2/w)(1-\gamma_2z)(1-\gamma_2/z)}dzdw,
\end{split}
\end{equation*}
and $K_{21}(x,y)=-K_{12}(y,x)$.

Then the $K_{11}$ entry is equal to $k(x,y)$. The $K_{12}$ entry is obtained by deforming the contour to $|w|=1/r'$ and substituting $w\mapsto w^{-1}$. Note that the differences in the integral to the one defining the $K_{11}$ entry amount to a factor of
\begin{equation}
\label{eq: part frac}
    -(1-\gamma_2)^2\frac{w-w^{-1}}{(1-\gamma_2w)(1-\gamma_2w^{-1})}=\frac{(1-\gamma_2)^2}{\gamma_2}\left(\frac{1}{1-\gamma_2w}-\frac{1}{1-\gamma_2w^{-1}}\right),
\end{equation}
which matches the factor given by
\begin{equation*}
    2D_y^{(\gamma_2)}w^{-y}=\frac{(1-\gamma_2)^2}{\gamma_2}\left(\frac{1}{1-\gamma_2w^{-1}}-\frac{1}{1-\gamma_2w}\right)w^{-y}.
\end{equation*}
It is easy to see that the claimed formula for $K_{21}(x,y)$ equals $-K_{12}(y,x)$.

To check the formula for $K_{22}$, we perform contour deformations and substitutions for both $z$ and $w$. However, when we deform the contours, we encounter a pole at $z=w^{-1}$ due to the $1-zw$ factor in the denominator of $K(1/z,1/w)$, after which the same computations as above show the differences in the integrals amount to a factor matching the ones obtained from the difference operators. Since $F(z^{-1})=F(z)^{-1}$ and $K(w,1/w)=(1-\gamma_2w)(1-\gamma_2/w)/(1-w^2)$, the residue at $z=w^{-1}$ is
\begin{equation*}
    \frac{1}{2\pi i}\oint_{|w|=r'} w^{y-x-1}\frac{w^{-1}-w}{(1-\gamma_2w)(1-\gamma_2/w)}dw.
\end{equation*}
By assumption, $|\gamma_2|<1/r'$ and so the partial fraction decomposition \eqref{eq: part frac} can be expanded into two power series, giving a residue of $a_i^{|x-y|-1}$ if $|x-y|\geq 1$ and $0$ if $x=y$, which matches $(D^{(\gamma_2)}_x-D^{(\gamma_2)}_y)\Delta(x,y)$. In particular, $|\gamma_2|<1/r'$ means the exponential growth rate of $k(x,y)$ is strictly smaller than $\gamma_2^{-1}$ so $D_x^{(\gamma_2)}$ is well-defined in the all the claimed expressions.
\end{proof}
\begin{remark}
Let us note that one needs to take
\begin{equation*}
    1<r'<r<\min(\nu,|\gamma_2|^{-1},\sqrt{q}^{-1})
\end{equation*}
That is, all poles within the unit ball should be contained in the contours, and all poles outside the unit ball should be outside these contours (and no poles should lie on the unit sphere). From now on, we will always implicitly assume this when using Theorem \ref{thm: pfaff formula}, and will simply say that $r,r'$ are sufficiently close to $1$. This is quite restrictive, and we will want to relax these assumptions to obtain formulas in the $\rho\leq \frac{1}{2}$ regimes. This will require analytic continuation.
\end{remark}

Let us record for later use the two specific cases we are interested in. 
\begin{corollary}
\label{cor: 6vm formula}
Let $h(n,n)$ denote the height function in the six vertex model at $(n,n)$. Then
\begin{equation*}
    \PP(h(n,n)+\chi+2S\leq s)=\Pf(J-K)_{l^2(\Z_{>s})},
\end{equation*}
with
\begin{equation}
\label{eq: 6vm F}
    F(z)=\prod \frac{1+a_iz}{1+a_i/z}.
\end{equation}
\end{corollary}
\begin{proof}
By Theorems \ref{thm: general 6vm HL distr equality} and \ref{thm: HL FBS same}, we have that $h(n,n)+\chi$ is distributed as $\lambda_1$ in a free boundary Schur process $\mathbb{FBS}_{\widehat{a}}^{(q,t,\nu)}$. The relevant specialization of $H(a;z)$ is given by \eqref{eq: H spec}, giving
\begin{equation*}
    H(\widehat{a};z)=\prod _{i}\frac{1+a_iz}{1+qa_iz},
\end{equation*}
which gives the stated formula for $F(z)$. Furthermore, the restriction that $\nu^2 t<1$ can be removed, as both sides are analytic in $\nu$ and $t$, and make sense even without this restriction.
\end{proof}

\begin{corollary}
\label{cor: ASEP formula}
Let $N(\tau)$ denote the current at $0$ in the half space ASEP. Then
\begin{equation*}
    \PP(-N(\tau)+\chi+2S\leq s)=\Pf(J-K)_{l^2(\Z_{>s})},
\end{equation*}
where
\begin{equation}
\label{eq: asep F}
    F(z)=e^{(1-q)\tau\frac{1}{2}\frac{1-z}{1+z}}.
\end{equation}
\end{corollary}
\begin{proof}
The ASEP can be obtained from the six vertex model. Specifically, setting $a_i=1-(1-q)\varepsilon/2$, if $n=\tau/\varepsilon$, we have as $\varepsilon\to 0$,
\begin{equation*}
    n-h(n,n)\to N(\tau),
\end{equation*}
see for example Proposition 4.1 of \cite{H22}. Starting from \eqref{eq: 6vm F}, we can compute that $z^{-n}F(z)$ converges to \eqref{eq: asep F}. The shift in $z$ is due to the corresponding shift in $x,y$ to cancel the shift in $s$. As the integrands are bounded in $\varepsilon$ (and in particular no poles cross the contour as $\varepsilon\to 0$), taking a limit of Corollary \ref{cor: 6vm formula} and applying dominated convergence gives the desired result. In particular, note that the parameters $\nu$ and $t$ in the six vertex model determine $\alpha$ and $\beta$, the boundary rates in the ASEP, in such a way that the formula for $\nu^{-1}$ for the ASEP is satisfied.
\end{proof}
\begin{remark}
\label{rem: t=0 special}
There is some subtlety in Corollary \ref{cor: ASEP formula} when $t=0$. Normally, the correspondence between the parameters $(\alpha,\beta)$ and $(t,\nu)$ in the ASEP is bijective, but this fails when $t=0$. Indeed, for any $\alpha\geq 1-q$, we have the density is $1$ at the boundary. Thus, while the formula given works for $\alpha< 1-q$, it doesn't make sense beyond this. Instead, one must allow $\nu\in (-\infty,-1)$ to recover all values of $\alpha$. Normally, this would correspond to the symmetry exchanging $-\nu t$ and $\nu^{-1}$ in the six vertex model, but in this degenerate situation, it instead allows access to the full range of parameters. Note that there are no issues with this in the formula (or in the six vertex model) since $t=0$, and in particular, the $\nu\to\infty$ and $\nu\to -\infty$ limits match, both giving the ASEP at $\alpha=1-q$. This ultimately has no effect on the asymptotic analysis, and so we will ignore this subtlety from now on.
\end{remark}

\section{Asymptotics: Preliminaries}
Having established the algebraic identities needed to connect the half space ASEP and six vertex model with free boundary Schur processes and obtaining exact formulas, we now turn to the problem of computing asymptotics with these formulas. In fact, our formulas currently only hold when $\nu>1$, and so much of the difficulty will be to analytically continue this to all $\nu>0$. We first begin by giving Fredholm Pfaffian formulas for our limiting distributions, and then establish some preliminary tools. In the following two sections, we establish Theorems \ref{thm: main} and \ref{thm: main 6vm}, first in the $\rho\geq \frac{1}{2}$ regimes, and then the $\rho<\frac{1}{2}$ regime.

\label{sec: asymp prelim}
\subsection{Fredholm Pfaffian formulas for limiting distributions}
The Tracy--Widom GOE and GSE distributions arise as the distribution of the largest eigenvalue of certain matrix ensembles. Introduced in \cite{TW96}, they appear ubiquitously in the study of random matrices and stochastic models in the KPZ universality class as a universal scaling limit. The Baik--Rains crossover distribution was introduced in \cite{BR01a} as a family of distributions interpolating between the GOE and GSE distributions.

We first give the limiting Fredholm Pfaffian kernels that we wish to show as limits. These formulas were introduced in \cite{IMS22}, although many other formulas are known. For the purposes of this paper, these formulas will serve as definitions for the Tracy--Widom GOE, GSE, and Baik--Rains crossover distributions.

For $\delta\in \R$, we let $C_\delta$ denote the contour given by two lines starting at $\delta$ and making a $\pm\pi/3$ angle with the positive real line, oriented to start below and end above the real axis. 

\begin{definition}
\label{def: BR TW}
The \emph{Baik--Rains crossover distribution} is a family of distributions $F_{cross}(s;\xi)$ indexed by $\xi\in (0,\infty)$, with
\begin{equation*}
    F_{cross}(s;\xi)=\Pf(J-K^{(\xi)}_{cross})_{L^2(s,\infty)},
\end{equation*}
with
\begin{equation*}
    K^{(\xi)}_{cross}=\begin{pmatrix}
    k^{(\xi)}_{cross}(u,v)&-\partial_v k^{(\xi)}_{cross}(u,v)
    \\-\partial_u k^{(\xi)}_{cross}(u,v)&\partial_u\partial_v k^{(\xi)}_{cross}(u,v)
    \end{pmatrix},
\end{equation*}
where
\begin{equation*}
    k_{cross}^{(\xi)}(u,v)=\frac{1}{4}\frac{1}{(2\pi i)^2}\oint_{C_\delta} \oint_{C_\delta}\frac{(\alpha-\beta)(\xi+\alpha)(\xi+\beta)}{\alpha\beta(\alpha+\beta)(\xi-\alpha)(\xi-\beta)}e^{\frac{\alpha^3}{3}-\alpha u+\frac{\beta^3}{3}-\beta v}d\alpha d\beta,
\end{equation*}
and $\xi>\delta>0$. We define $F_{GSE}(s)=F_{cross}(s;\infty)$, given by
\begin{equation*}
    F_{GSE}(s)=\Pf(J-K_{GSE})_{L^2(s,\infty)},
\end{equation*}
with
\begin{equation*}
    K_{GSE}=\begin{pmatrix}
    k_{GSE}(u,v)&-\partial_v k_{GSE}(u,v)
    \\-\partial_u k_{GSE}(u,v)&\partial_u\partial_v k_{GSE}(u,v)
    \end{pmatrix},
\end{equation*}
where
\begin{equation*}
    k_{GSE}(u,v)=\frac{1}{4}\frac{1}{(2\pi i)^2}\oint_{C_1} \oint_{C_1}\frac{\alpha-\beta}{\alpha\beta(\alpha+\beta)}e^{\frac{\alpha^3}{3}-\alpha u+\frac{\beta^3}{3}-\beta v}d\alpha d\beta.
\end{equation*}
Finally, we define $F_{GOE}(s)$ as $F_{cross}(s;0)$, given by the formula
\begin{equation*}
    F_{GOE}(s)=\Pf(J-K_{GOE})_{L^2(s,\infty)},
\end{equation*}
with
\begin{equation*}
    K_{GOE}=\begin{pmatrix}
    k_{GOE}(u,v)&-\partial_v k_{GOE}(u,v)
    \\-\partial_u k_{GOE}(u,v)&\partial_u\partial_v k_{GOE}(u,v)
    \end{pmatrix},
\end{equation*}
where
\begin{equation*}
    k_{GOE}(u,v)=\frac{1}{4}\frac{1}{(2\pi i)^2}\oint_{C_1} \oint_{C_1}\frac{\alpha-\beta}{\alpha\beta(\alpha+\beta)}e^{\frac{\alpha^3}{3}-\alpha u+\frac{\beta^3}{3}-\beta v}d\alpha d\beta+b(u)-b(v)
\end{equation*}
and
\begin{equation*}
    b(u)=\frac{1}{2\pi i}\oint_{C_1}\frac{e^{\frac{\alpha^3}{3}-\alpha u}}{2\alpha}d\alpha.
\end{equation*}
\end{definition}
\begin{remark}
The extra summands in the definition of $F_{GOE}$ come from poles which cross the contour as $\xi\to 0$, see \cite{IMS22} for details. We will see the same phenomenon for the pre-limiting kernels when we study asymptotics in the $\rho=\frac{1}{2}$ regime.
\end{remark}

\subsection{Summation by parts}
We first establish some summation by parts formulas for $D_x^{(\gamma)}$, which we recall was defined by \eqref{eq: diff op}. The following result is easily seen by comparing the two sides of the equality.
\begin{lemma}
Let $f,g:\Z\to\R$ be of sufficient decay. We have the summation by parts formula
\begin{equation*}
    \sum_{x> s}(D_x^{(\gamma)}f)(x)g(x)=-\sum_{x>s}f(x)(D_x^{(\gamma)}g)(x)+B^{(\gamma)}(f,g)(s)+B^{(\gamma)}(g,f)(s),
\end{equation*}
where
\begin{equation*}
    B^{(\gamma)}(f,g)(s)=\frac{(1-\gamma)^2}{2}\sum_{i\geq 1}\sum_{1\leq j\leq i}\gamma^{i-1}f(s+j)g(s+j-i).
\end{equation*}
\end{lemma}
We will only need to use the following special case.
\begin{lemma}
\label{lem: sum by parts}
Let $f:\Z^2\to\R$ be of sufficient decay, and let $\Delta(x,y)$ be the indicator that $x=y$. Then
\begin{equation*}
    \sum_{x,y>s}((D_x^{(\gamma)}-D_y^{(\gamma)})\Delta(x,y))f(x,y)=-\sum_{x>s}((D_y^{(\gamma)}-D_x^{(\gamma)})f(x,y))|_{y=x}+(B^{(\gamma)}_x-B^{(\gamma)}_y)f(s,s),
\end{equation*}
where $|_{y=x}$ means we substitute $y=x$ into the expression, and
\begin{equation*}
    (B^{(\gamma)}_x-B^{(\gamma)}_y)f(s,s)=\frac{(1-\gamma)^2}{2}\sum_{i\geq 1}\sum_{1\leq j\leq i}\gamma^{i-1}(f(s+j-i,s+j)-f(s+j,s+j-i)).
\end{equation*}
\end{lemma}
\begin{proof}
We simply apply the previous lemma. Note that $B^{(\gamma)}_x(f(x,y),\Delta(x,y))=0$ since each term contains a factor $\Delta(s+j-i,y)$ with $j\leq i$ and $y>s$, and similarly for $B_y^{(\gamma)}(f(x,y),\Delta(x,y))$. The expression for $(B^{(\gamma)}_x-B^{(\gamma)}_y)f(s)$ comes from evaluating the remaining boundary terms.
\end{proof}

More generally, for variables $x_1,\dotsc, x_n$, we define analogously $D_M^{(\gamma)}$ and $B_M^{(\gamma)}$ for any matching $M$ (i.e. a set of disjoint pairs from $[n]$), with the convention that the smaller index always appears as the first term.

We will need to repeatedly use a multivariate version of this integration by parts. First, we introduce some notation. We let $x=(x_1,\dotsc, x_n)$ and $y=(y_1,\dotsc, y_n)$ denote a set of $2n$ variables. For a subset $I\subseteq [n]=\{1,\dotsc, n\}$, we let $D_{I}^{(\gamma)}=\prod _{i\in I}(D_{x_i}^{(\gamma)}-D{y_i}^{(\gamma)})$ and $B_I^{(\gamma)}=\prod_{i\in I}(B_{x_i}^{(\gamma)}-B_{y_i}^{(\gamma)})$. Note that since the different operators act on disjoint sets of variables, they all commute.
\begin{lemma}
\label{lem: int by parts multivariate}
We let $f:\Z^{2n}\to \R$ be of sufficient decay. Then
\begin{equation*}
    \sum_{\substack{x_1,\dotsc, x_n>s\\y_1,\dotsc, y_n>s}}(D_{[n]}^{(\gamma)}\prod \Delta(x_i,y_i))f(x,y)=\sum_{I\subseteq [n]}(-1)^{|I|}\sum_{x_i>s, i\in I}D_I^{(\gamma)}B_{I^c}^{(\gamma)}f(x,y)\Bigg|_{\substack{y_i=x_i, i\in I\\ x_i=y_i=s, i\in I^c}}.
\end{equation*}
\end{lemma}
\begin{proof}
This follows immediately from applying the previous lemma to each pair of variables and expanding.
\end{proof}

Finally, we will eventually want to show that asymptotically, these boundary terms vanish on functions which are skew-symmetric. The following simple estimate, which follows from splitting the sum at $i=N+1$, will suffice.

\begin{lemma}
\label{lem: boundary bound}
Let $f:\Z^2\to\R$ such that $|f(x)|\leq C((1-c)\gamma^{-1})^{|x|}$ for some $c$. Then for all $N\geq 0$,
\begin{equation*}
    |B^{(\gamma)}f(s,s)|\leq (1-\gamma)^{-2}(C(1-c)^N+\|f|_{[s-N,s+N]^2}\|_\infty).
\end{equation*}
\end{lemma}

\subsection{Extracting asymptotics from signed measures}
\label{sec: conv inv}
We needed to add random shifts to our observables of interest to obtain exact formulas. Asymptotically, these random shifts do not matter since they are of constant order and we scale non-trivially. However, when $t>q$, $\chi\sim RS(q,t)$ is not actually a random variable, but rather a signed measure. We thus show that asymptotics can still be extracted in this case.

We first show that $RS(q,t)$ is itself the convolution of two simpler signed measures, and that the signed portion has a convolutional inverse.
\begin{lemma}
We have
\begin{equation*}
    \frac{q^nh_n(-t/q;q)}{(q;q)_n}=\sum_{k+l=n}\frac{q^k}{(q;q)_k}\frac{(-t)^l}{(q;q)_l}.
\end{equation*}
\end{lemma}
\begin{proof}
We compute
\begin{equation*}
    \frac{q^nh_n(-t/q;q)}{(q;q)_n}=\sum_{k}\frac{q^{n-k}(-t)^k}{(q;q)_k(q;q)_{n-k}}
\end{equation*}
and the result follows.
\end{proof}

\begin{lemma}
We have
\begin{equation*}
    \sum_{k+l=n}\frac{(-t)^k}{(q;q)_k}\frac{t^lq^{l\choose 2}}{(q;q)_l}=\begin{cases} 1&\text{ if }n=0
    \\0 &\text{ if }n>0
    \end{cases}
\end{equation*}
\end{lemma}
\begin{proof}
We compute the generating function
\begin{equation*}
    \sum_{k,l}\frac{(-t)^k}{(q;q)_k}\frac{t^lq^{l\choose 2}}{(q;q)_l}z^{k+l}=\sum_{k}\frac{(-tz)^k}{(q;q)_k}\sum_l\frac{(tz)^lq^{l\choose 2}}{(q;q)_l}=1
\end{equation*}
using the $q$-binomial theorem on each factor. Taking the coefficient of $z^n$ gives the desired result.
\end{proof}

Next, we show that convolution against a fixed signed measure of sufficient decay has no effect on the asymptotics.
\begin{lemma}
\label{lem: signed measure asymp}
Let $f:\N\to \R$ be a distribution function on $\N$, and let $F_n$ denote a sequence of functions on $\Z$ such that $\sum_{x} |F_n(x)|$ is uniformly bounded, $\sum F_n(x)=1$, and $\sup_{x\geq a(n)+b(n)(s-1)}|F_n(x)|\to 0$ as $n\to\infty$, and
\begin{equation*}
    \sum_{x\leq  a(n)+b(n)s}F_n(x)\to \mu(-\infty, s)
\end{equation*}
for functions $a,b:\N\to\N$ with $b(n)\to\infty$ and $\mu$ a probability measure on $\R$. Then
\begin{equation*}
    \sum_{x+y\leq  a(n)+b(n)s}f(x)F_n(y)\to \mu(-\infty, s)
\end{equation*}
\end{lemma}
\begin{proof}
It suffices to show
\begin{equation*}
    \sum_{x\leq \omega(n)}f(x)\sum_{y\leq a(n)+b(n)s-x}F_n(y)\to \mu(-\infty, s)
\end{equation*}
for $\omega(n)$ a function with $\omega(n)\to\infty$, since
\begin{equation*}
    \sum_{x>\omega(n)}f(x)\left|\sum_{y\leq a(n)+b(n)s-x}F(y)\right|\to 0.
\end{equation*}
Take $\omega(n)^2$ to grow slower than $(\sup_{x\geq a(n)+b(n)(s-1)}|F_n(x)|)^{-1}$, and slower than $b(n)$. Then we may replace the sum over $y$ with a sum up to $a(n)+b(n)s$ for all $x$, with an error bounded by
\begin{equation*}
    \omega(n)^{2}\sup_{a(n)+b(n)s-\omega(n)\leq y}|F_n(y)|,
\end{equation*}
which goes to $0$ by the choice of $\omega(n)$. This then implies the desired result.
\end{proof}

\section{Asymptotics: Crossover, GSE, and GOE asymptotics}
\label{sec: cross}
In this section, we establish GSE, GOE, and crossover asymptotics for $\nu^{-1}\geq 1$. The crossover and GSE regimes are very similar, while the GOE regime requires a little extra work due to the crossing of poles through the contours. We leave the Gaussian asymptotics for Section \ref{sec: Gauss} as they are much more involved.

\subsection{Crossover and GSE}
Unfortunately, the kernel from Theorem \ref{thm: pfaff formula} is not of the same form as the kernels defining the limiting distributions. We must first use summation by parts to obtain a kernel of the same form, and show that the boundary terms appearing are negligible.

Let us now state the assumptions that we will require in this section. These assumptions will need to be suitably modified in the other regimes, but we will explain how to adapt the arguments here in each other case separately.
\begin{assumption}
\label{asmp: kernels}
Let $k(x,y):\Z^2\to \R$ be a family of skew symmetric functions indexed by $\tau$ (or $n$ with the obvious adjustments below). Fix $s\in \R$. We introduce new variables $u,v$ via $x=\mu\tau+\sigma\tau^{1/3}u$ and $y=\mu\tau+\sigma\tau^{1/3}v$ for some constants $\mu,\sigma$. We will assume that there exists $c,C>0$ such that
\begin{equation}
\label{eq: asmp exp bound}
    \left|\left(\sigma\tau^{1/3}D_x^{(\gamma_2)}\right)^i\left(\sigma\tau^{1/3}D_y^{(\gamma_2)}\right)^jk(u,v)\right|\leq Ce^{-c(u+v)}
\end{equation}
for $i,j\leq 2$ and $u,v>s$, and
\begin{equation}
\label{eq: asmp perturb bound}
    \left|\left(\sigma\tau^{1/3}D_x^{(\gamma_2)}\right)^i\left(\sigma\tau^{1/3}D_y^{(\gamma_2)}\right)^j(k(u+du,v+dv)-k(u,v))\right|\leq C\tau^{-1/6}
\end{equation}
for all $i,j\geq 0$, $u,v>s$, and $du, dv\in [-\tau^{-1/6}, \tau^{-1/6}]$.

We study the ASEP and the six vertex model in this paper. For the half space ASEP, $\mu=-\frac{1-q}{4}$ and $\sigma=2^{-4/3}(1-q)^{1/3}\tau^{1/3}$. For the homogeneous six vertex model with $a_i=a$ for all $i$, we have $\mu=\frac{2a}{1+a}$ and $\sigma=\frac{(a(1-a))^{1/3}}{1+a}$.
\end{assumption}

\begin{lemma}
\label{lem: boundary bound complicated}
Let $k(x,y):\Z^2\to \R$ be a family of skew-symmetric kernels satisfying Assumption \ref{asmp: kernels}. Then for any $M$ a perfect matching of some even subset $J$ of $\{1,\dotsc, n\}$ and $M'\subseteq M$, we have
\begin{equation*}
\begin{split}
    &\sum_{\substack{u_i>s\\i\not\in J\text{ or }(i,j)\in M'}}\left|B^{(\gamma_2)}_{(M')^c}D_{M'}^{(\gamma_2)}\Pf(k(u_i,u_j))_{(2J)^c}\bigg|_{\substack{u_i=u_j=s, (i,j)\in (M')^c\\ u_i=u_j, (i,j)\in M'}}\right|
    \\\leq &C(2(n-|J|))^{(n-|J|)/2}\tau^{-|(M')^c|/12}C^n,
\end{split}
\end{equation*}
where the convention is that the Pfaffian is of the submatrix obtained by removing rows and columns $2j$ for $j\in J$.
\end{lemma}
\begin{proof}
We can move the derivatives $D_{M'}^{(\gamma_2)}$ inside the Pfaffian since $D_{M'}^{(\gamma_2)}$ affects only variables for which one row has been removed from $k$, which means that the variables appear in exactly one row/column. Thus, we obtain a new skew-symmetric kernel of a very similar form, which in particular still satisfies the decay assumptions of Assumption \ref{asmp: kernels}. We now apply Lemma \ref{lem: boundary bound}, with $N=\tau^{1/6}$. Since the Pfaffian has an exponential growth rate less than $\gamma_2^{-1}$ due to the radius of the contour defining $k$, the first term can essentially be ignored, so we focus on bounding the second term. For this, we need a uniform estimate on 
\begin{equation*}
    |D_{M'}^{(\gamma_2)}\Pf(k(x_i,x_j))_{J^c}|_{u_i=u_j, (i,j)\in M'}
\end{equation*}
for $u_i,u_j\in [s-\tau^{-1/6},s+\tau^{-1/6}]$ for $(i,j)\in (M')^c$. 

We subtract column $i$ from column $j$ if $(i,j)\in (M')^c/2$. Since $|u_i-u_j|\leq 2\tau^{-1/6}$, we know the corresponding matrix entries in column $j$ are bounded by $C\tau^{-1/6}$ for some constant $C$. We also know that the other matrix entries are bounded by $Ce^{-cu_i}$ where $u_i$ is the variable corresponding to the column (we're absorbing the row dependence into $C$). Then applying Lemma \ref{lem: Hadamard}, and taking a geometric average with the bound obtained by subtracting column $j$ from column $i$ gives a bound of
\begin{equation*}
    (2(n-|J|))^{(n-|J|)/2}\tau^{-|(M')^c|/12}e^{-c\sum u_i},
\end{equation*}
where the sum over $i$ includes all variables in $J^c$ and $i$ for $(i,j)\in M'$, and summing over the remaining variables gives the desired bound.
\end{proof}

We now prove an approximate discrete analogue of an identity of Fredholm Pfaffians (see Proposition 5.7 of \cite{BBCS18}). This lets us avoid working with distribution-valued kernels.
\begin{proposition}
\label{prop: approx distr pf identity}
Let $k(x,y):\Z^2\to \R$ be a family of skew-symmetric kernels satisfying Assumption \ref{asmp: kernels}. Let
\begin{align*}
    K&=\begin{pmatrix}
    k(x,y)&-D_y^{(\gamma)}k(x,y)\\
    -D_x^{(\gamma)}k(x,y)&D_x^{(\gamma)}D_y^{(\gamma)} k(x,y)
    \end{pmatrix},
    \\K'&=\begin{pmatrix}
    k(x,y)&-2D_y^{(\gamma)}k(x,y)\\
    -2D_x^{(\gamma)}k(x,y)&4D_x^{(\gamma)}D_y^{(\gamma)} k(x,y)+(D_x^{(\gamma)}-D_y^{(\gamma)})\Delta
    \end{pmatrix}.
\end{align*}
Assume that the series expansion for $\Pf(J-K)_{l^2(\Z_{>s})}$ converges absolutely. Then for all $s$, the series expansion for $\Pf(J-K')_{l^2(\Z_{>s})}$ converges absolutely, and
\begin{equation*}
    \left|\Pf(J-K)_{l^2(\Z_{>s})}-\Pf(J-K')_{l^2(\Z_{>s})}\right|\leq C\tau^{-1/12}.
\end{equation*}
\end{proposition}
\begin{proof}
The idea is that if summation by parts held, the two kernels have the same Fredholm Pfaffian. We thus show that the boundary terms appearing in Lemma \ref{lem: int by parts multivariate} are small by Lemma \ref{lem: boundary bound complicated}. The details are essentially the same as in Proposition 5.7 of \cite{BBCS18}, and so we will only give the details when showing that the boundary terms are small.

Write $K'=A+E$, with $E$ being a matrix containing only the terms $(D_x^{(\gamma)}-D_y^{(\gamma)})\Delta$. Using Lemma \ref{lem: pf expansion} and the expansion of $\Pf(E_J)$ in terms of matchings, we have that
\begin{equation*}
\begin{split}
    \sum_{x_1,\dotsc, x_n>s}\Pf(K'(x_i,x_j))_{i,j=1}^n&=\sum_{x_1,\dotsc, x_n>s}\sum_{J}\sum_M \Pf(A_{J^c})D_J^{(\gamma_2)}\prod _{(i,j)\in M}\Delta(x_i,x_j),
\end{split}
\end{equation*}
where the sum is over even subsets $J$ and matchings $M$ of $J$. Note in particular that $J$ must be a subset of the even integers, since the odd rows/columns of $E$ are $0$.

Now we can use the summation by parts given by Lemma \ref{lem: int by parts multivariate}, giving a main term where no boundary operators $B$ are present, as well as the error term bounded by
\begin{equation*}
    \sum_{J}\sum_{M}\sum_{M'\subsetneq M}\left|D_{M'}^{(\gamma)}B_{(M')^c}^{(\gamma)}\Pf(A_{J^c})\Bigg|_{\substack{x_i=x_j, (i,j)\in M'\\ x_i=y_j=s, (i,j)\in (M')^c}}\right|,
\end{equation*}
where $M'\neq M$.

First, let us briefly explain why the main terms sum to $\Pf(J-K)_{l^2(\Z_{>s})}$. Since this is identical to the proof of Proposition 5.7 in \cite{BBCS18}, we will not give all the details. The idea is that each pair of difference operator acts on a pair of variables which each appear in exactly one row/column. For this reason, these operators may be moved inside the Pfaffian. But notice that after setting the two variables equal, we end up with two rows/columns for that variable of the same type as in $K$ (up to some factors of $2$). The proof proceeds by then collecting and summing the different coefficients for the terms corresponding to the same term in $\Pf(J-K)_{l^2(\Z_{>s})}$, and showing that this is $1$.

We now turn to the error terms. Note that all error terms are of the form appearing in Lemma \ref{lem: boundary bound complicated}, which applies since we assume Assumption \ref{asmp: kernels}. The exponential bound allows the integrals to be uniformly bounded in each variable. Then the total error in the $n$th term from the boundary terms after integrating is bounded by
\begin{equation*}
\begin{split}
    &\sum_{J}\sum_{M}\sum_{M'\subsetneq M}C_1\tau^{-1/12}(2|J^c|)^{|J^c|/2}C_2^{n}
    \\\leq &C_1\tau^{-1/12}\sum_{k=0}^{n/2} {n\choose 2k}\frac{(2k)!}{2^kk!}2^k(2n-2k)^{(n-k)/2}C_2^{n}
    \\\leq &C\tau^{-1/12}n! \sum_{k=0}^{n/2}\frac{e^{2k}}{k!}\left(\frac{\sqrt{2n-2k}}{n-2k}\right)^{n-k}C_2^n
    \\\leq &C\tau^{-1/12}n!c^n,
\end{split}
\end{equation*}
for some small constant $c>0$, using the fact that there are $n\choose 2k$ ways to pick the even subset $J$ and $(2k)!/2^kk!$ many matchings $M$ of $J$. Summing this over $n$ (with the $1/n!$ factor), we have the error is bounded by
\begin{equation*}
    C\tau^{-1/12}\sum_{n\geq 0}c^n=O(\tau^{-1/12}),
\end{equation*}
as desired.

Moreover, it's clear from the above proof that the error terms are absolutely summable, and since the series expansion for $\Pf(J-K')_{l^2(\Z_{>s})}$ converges absolutely, this implies the same for $\Pf(J-K)_{l^2(\Z_{>s})}$.
\end{proof}

\begin{remark}
Note that the proof of Proposition \ref{prop: approx distr pf identity} only relies on the fact that the difference operator $D^{(\gamma)}$ satisfies an approximate summation by parts formula with the uniform error estimates given by Lemma \ref{lem: boundary bound}. Thus, it applies to any operator satisfying these bounds.
\end{remark}

We now show that the kernels in the Fredholm Pfaffian formulas for the ASEP and the six vertex model converge to $K^{(\xi)}_{cross}$ and $K_{GSE}$ in their respective regimes. We then use a dominated convergence argument to show that this implies convergence of the Fredholm Pfaffians themselves. To show convergence of the kernels, we do a steepest descent analysis. This argument is fairly standard, see e.g. \cite{BBCS18}. Note that we also do the $\nu=1$ limit here, because the argument is essentially identical. However, obtaining asymptotics for the distribution function requires additional work, and so is delayed until Section \ref{sec: GOE}.
\begin{lemma}
\label{lem: GSE cross ptwise conv}
Let $x=-\frac{1-q}{4}\tau+2^{-4/3}(1-q)^{1/3}\tau^{1/3}u$ and $y=-\frac{1-q}{4}\tau+2^{-4/3}(1-q)^{1/3}\tau^{1/3}v$. For $i,j\in \N$, we have
\begin{equation*}
    \left(2^{-4/3}\tau^{1/3}D_x^{(\gamma_2)}\right)^i\left(2^{-4/3}\tau^{1/3}D_y^{(\gamma_2)}\right)^jk(x,y)\to \partial_u^i\partial_v^j k_{GSE}
\end{equation*}
if $\nu\geq 1$ is fixed, and
\begin{equation*}
    \left(2^{-4/3}\tau^{1/3}D_x^{(\gamma_2)}\right)^i\left(2^{-4/3}\tau^{1/3}D_y^{(\gamma_2)}\right)^jk(x,y)\to  \partial_u^i\partial_v^j k_{cross}^{\xi}(u,v)
\end{equation*}
if $\frac{1}{\nu}=1-\frac{2^{4/3}\xi}{\tau^{1/3}}$. Moreover, $k(u,v)$ satisfies Assumption \ref{asmp: kernels}.
\end{lemma}
\begin{proof}
We proceed by a steepest descent analysis. Since these statements are all similar, we will show $k(x,y)\to k_{cross}^{\xi}(u,v)$ in detail, and note the differences for the other statements. It should be noted that $D_x^{(\gamma_2)}$ in some sense converges to $\partial_u$, but this does not immediately imply the claimed statements, since the functions the difference operators are applied to also vary with $n$. Instead, we directly show convergence of the functions themselves.

\begin{figure}
    \centering
    \includegraphics[scale=0.7]{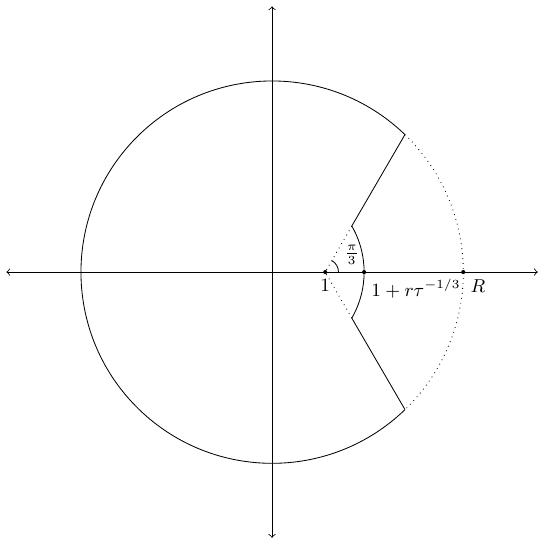}
    \caption{The contour $C_{r,R}$.}
    \label{fig: contour}
\end{figure}

Let $C_{r,R}$ be the contour consisting of the shorter arc of the circle $|z|=1+r\tau^{-1/3}$ for some $1.5>r\tau^{-1/3}>0$ (if needed just take $\tau$ large enough) between the lines $y=\pm \sqrt{3}(x-1)$, line segments along the lines $y=\pm\sqrt{3}(x-1)$ between the endpoints of the arc and the circle $|z|=R$, and the longer arc of the circle $|z|=R$ connecting the endpoints of the line segments (see Figure \ref{fig: contour}). 

For simplicity, let us rescale $\tau$ by $1/(1-q)$, which has no effect on the $\tau\to\infty$ limit. We begin by deforming the contours $|z|=r$ and $|w|=r'$ to contours $C_{r,R}$. Here, we take $r<2^{4/3}\xi$ and $R>1$ but small enough to avoid any poles.

Let
\begin{equation*}
    G(z)=\frac{1}{2}\frac{1-z}{1+z}+\frac{1}{4}\log(z),
\end{equation*}
where the branch cut is the negative real axis. We note that $G(1)=G'(1)=G''(1)=0$ and $G'''(1)=\frac{1}{8}$, and that (recalling that $\tau$ has been rescaled)
\begin{equation}
\label{eq: int first formula}
\begin{split}
    k(x,y)=&\frac{(1-\gamma_2)^{-2}}{(2\pi i)^2}\oint_{C_{r,R}} \oint_{C_{r,R}}z^{-3/2}w^{-5/2}K(z,w)\\&\qquad\times\exp\left(\tau(G(z)+G(w))-2^{-4/3}\tau^{1/3}(u\log z+v\log w)\right)dzdw.
\end{split}
\end{equation}
The idea will be to show that the integral localizes near the critical point $z=1$ and $w=1$, and near this point it converges to the desired limit as $\tau\to\infty$.

If $z=re^{i\theta}$, we have
\begin{equation*}
    \Re(G(z))=\frac{1}{2}\frac{1-r^2}{(1-r)^2+2r(\cos(\theta)+1)}+\frac{1}{4}\log r.
\end{equation*}
We can compute
\begin{equation*}
    \frac{d}{d\theta}\Re(G(z))=-\frac{r(r^2-1)\sin(\theta)}{(r^2+2r\cos(\theta)+1)^2},
\end{equation*}
which implies that $\Re(G(z))$ is decreasing as you move away from the positive real axis along the circular portions of $C_{r,R}$. Similarly, if $z=x+iy$, the real part of $G$ on the line $y=\sqrt{3}(x-1)$ is given by
\begin{equation*}
    -\frac{1}{2}\frac{2x^2-3x+1}{2x^2-2x+2}+\frac{1}{8}\log(4x^2-6x+3),
\end{equation*}
and so on this line,
\begin{equation*}
    \frac{d}{dx}\Re(G(z))=\frac{(x-1)^2(4x^3-7x^2-2x+3))}{4(x^2-x+1)^2(4x^2-6x+3)}<0
\end{equation*}
for $1<x<3/2$, and so $\Re(G(z))$ is also decreasing along the line segments of $C_r$. Thus, we can conclude that $\Re(G(z))$ attains a single maximum on the positive real line, and is strictly decreasing along the upper and lower portions of the contour, i.e. that $C_r$ is steepest descent for $\Re(G(z))$.

Plugging in any fixed $x$, we see that $\Re(G(z))$ is uniformly bounded below $0$, and so the integrand decays exponentially in $\tau$ outside of any fixed $\varepsilon$ neighbourhood of $1$ and has exponential decay in $u,v$ (the rest of the integrand has at most polynomial growth in $\tau$). Thus, we can neglect the contribution of that portion of the contour. We let $C_{r,\varepsilon}$ denote the portion of $C_{r,R}$ excluding this negligible portion.

We now rescale by $z=1+\frac{2^{4/3}\alpha}{\tau^{1/3}}$ and $w=1+\frac{2^{4/3}\beta}{\tau^{1/3}}$, and we have $G(z)=\tau^{-1}\alpha^3/3+O(|\alpha|^4\tau^{-4/3})$ and $\log(z)=2^{4/3}\tau^{-1/3}\alpha+O(|\alpha|^2\tau^{-2/3})$, where the constants depends only on $\varepsilon$. The error term in the exponent is thus bounded by $c\tau^{-1/3}(|\alpha|^4+|\beta|^4+|\alpha|^2+|\beta|^2)$ for some small $c>0$. Note also that if $\varepsilon$ is chosen small enough, we can ensure $\tau^{-1/3}|\alpha|$ is small.

Using the bound $|e^x-1|\leq |x|e^{|x|}$, this approximation has an error bounded by
\begin{equation}
\label{eq: int error}
\begin{split}
    &C\oint_{\widetilde{C}_{r,\varepsilon}} \oint_{\widetilde{C}_{r,\varepsilon}}c\tau^{-1/3}\left(|\alpha|^4+|\beta|^4+|\alpha|^2+|\beta|^2\right)\exp\left(c(|\alpha|^3+|\beta|^3+|\alpha|+|\beta|)\right)
    \\\qquad \qquad&\times\exp\left(\frac{\alpha^3}{3}-\alpha u+\frac{\beta^3}{3}-\beta v\right)|z^{-3/2}w^{-5/2}\tau^{-2/3}K(z,w)|d\alpha d\beta,
\end{split}
\end{equation}
where we have changed variables and $\widetilde{C}_{r,\varepsilon}$ denotes the rescaled contour. It's easy to see that
\begin{equation*}
    g(z)\longrightarrow -\frac{(1-\gamma_2)(\alpha+\xi)}{(\alpha-\xi)},
\end{equation*}
and so
\begin{equation*}
    2^{8/3}\tau^{-2/3}K(z,w)\longrightarrow \frac{(\alpha-\beta)}{4\alpha\beta(\alpha+\beta)}\frac{(1-\gamma_2)(\alpha+\xi)}{(\alpha-\xi)}\frac{(1-\gamma_2)(\beta+\xi)}{(\beta-\xi)},
\end{equation*}
where the extra factor comes from the change of variables. Thus, as $\tau\to \infty$, the rest of the integrand is bounded. Thus, the integral converges uniformly in $\tau$, so the error goes to $0$. Moreover, the tails of the integral decay exponentially, and so can be continued from $\widetilde{C}_{r,\varepsilon}$ to $C_\delta$ with a uniformly decaying error. We can then take a limit as $\tau\to \infty$ in
\begin{equation}
\label{eq: int after taylor series}
\begin{split}
    &\frac{(1-\gamma_2)^{-2}}{(2\pi i)^2}\oint_{\widetilde{C}_{r,\varepsilon}} \oint_{\widetilde{C}_{r,\varepsilon}}\exp\left(\frac{\alpha^3}{3}-\alpha u+\frac{\beta^3}{3}-\beta v\right)\frac{2^{8/3}K(z,w)}{z^{3/2}w^{5/2}\tau^{2/3}}d\alpha d\beta,
\end{split}
\end{equation}
again using exponential decay of the tails to apply dominated convergence to obtain the desired limit. Note that the limiting contour can be deformed to the one defining $K_{GSE}$, again due to exponential decay in the tails. 

Note that the difference operators act by introducing a factor of
\begin{equation*}
    \left(-\frac{\tau^{1/3}(1-\gamma_2)^2}{2^{7/3}}\frac{z-z^{-1}}{(1-\gamma_2z)(1-\gamma_2z^{-1})}\right)^i \left(-\frac{\tau^{1/3}(1-\gamma_2)^2}{2^{7/3}}\frac{w-w^{-1}}{(1-\gamma_2w)(1-\gamma_2w^{-1})}\right)^j,
\end{equation*}
which converges to the corresponding factor $(-\alpha)^i(-\beta)^j$ coming from differentiating $k_{cross}^{(\xi)}$ by $\partial_u^i\partial_v^j$, and the rest of the asymptotic analysis is identical. If $\nu>1$ is fixed, the same argument gives convergence to $k_{GSE}$, with the only difference being that
\begin{equation*}
    g(z)\to 1-\gamma_2,
\end{equation*}
which what is expected when setting $\xi=\infty$. If $\nu=1$, again, the only difference is
\begin{equation*}
    g(z)\to -(1-\gamma_2),
\end{equation*}
which corresponds to $\xi=0$.

The exponential decay of Assumption \ref{asmp: kernels} comes from the fact that in \eqref{eq: int first formula}, we have a uniform bound on the real parts of $G(z)$ and $\log(z)$ outside of a $\varepsilon$ neighbourhood of $1$, and within this neighbourhood, \eqref{eq: int error} and \eqref{eq: int after taylor series} both clearly decay exponentially in $u,v$ as $\alpha,\beta$ are bounded away from $0$.

Finally, if we perturb $u$ (or $v$) by a term $du$ of order at most $\tau^{-1/6}$, the same asymptotic analysis applies, except after reaching \eqref{eq: int error}, there is an extra factor of the form $z^{-d}=\exp(-d\log z)$, which is an additive $\tau^{-1/6}$ error term. Thus,
\begin{equation*}
    |k(u+du,v)-k(u,v)|\leq C\tau^{-1/6},
\end{equation*}
and similarly even if difference operators are applied, or $v$ is perturbed instead of $u$. Again, this can be made uniform in $u,v>s$ due to the exponential decay in $u,v$.
\end{proof}

\begin{lemma}
\label{lem: GSE cross ptwise conv 6vm}
Let $\sigma=\frac{(a(1-a))^{1/3}}{(1+a)}$, $x=\frac{2a}{1+a}n+\sigma n^{1/3} u$ and $y=\frac{2a}{1+a}n+\sigma n^{1/3}v$. For $i,j\in \N$, we have
\begin{equation*}
    \left(\sigma n^{1/3} D_x^{(\gamma_2)}\right)^i\left(\sigma n^{1/3} D_y^{(\gamma_2)}\right)^jk(x,y)\to \partial_u^i\partial_v^j k_{GSE}
\end{equation*}
if $\nu\geq 1$ is fixed, and
\begin{equation*}
    \left(\sigma n^{1/3} D_x^{(\gamma_2)}\right)^i\left(\sigma n^{1/3} D_y^{(\gamma_2)}\right)^jk(x,y)\to  \partial_u^i\partial_v^j k_{cross}^{\xi}(u,v)
\end{equation*}
if $\frac{1}{\nu}=1-\frac{\xi}{\sigma n^{1/3}}$. Moreover, $k(u,v)$ satisfies Assumption \ref{asmp: kernels}.
\end{lemma}
\begin{proof}
Since the analysis is extremely similar to that of Theorem \ref{lem: GSE cross ptwise conv}, we will simply explain the essential differences. Note that only the $F(z)$ and $F(w)$ factors in the integrals change. They are now of the form $F(z)=\left(\frac{1+az}{1+a/z}\right)^n$, and so we pick
\begin{equation*}
    G(z)=\log(1+az)-\log(1+a/z)-\frac{2a}{1+a}\log z.
\end{equation*}
Again, we can compute $G(1)=G'(1)=G''(1)=0$, and $G'''(1)=\frac{2a(1-a)}{(1+a)^3}$.

If $z=re^{i\theta}$, we have
\begin{equation*}
    \Re(G(z))=\frac{1}{2}\left(\log(1+2ar\cos\theta+a^2r^2)-\log(1+2a/r\cos\theta+a^2/r^2)-\frac{2a}{1+a}\log(r^2)\right),
\end{equation*}
and so
\begin{equation*}
\begin{split}
    \frac{d}{d\theta}\Re(G(z))&=-\frac{a(1-a^2)r(r^2-1)\sin(\theta)}{(a^2+2ar\cos(\theta)+r^2)(a^2r^2+2ar\cos(\theta)+1)}
    \\&\leq -\frac{a(1-a^2)r(r^2-1)\sin(\theta)}{(a-r)^2(1-ar)^2},
\end{split}
\end{equation*}
which implies that $\Re(G(z))$ is decreasing along the circular portions of $C_{r,R}$. If $z=x+iy$, then on the line $y=\pm\sqrt{3}x$,
\begin{equation*}
    \Re(G(z))=\frac{1}{2}\left(\log(1+2ax+4a^2x^2)-\log(a^2+2ax+4x^2)-\frac{1-a}{1+a}\log(4x^2)\right).
\end{equation*}
We have that $\frac{d^2}{dx^2}\Re(G(z))$ at $x=1$ is equal to $0$, and $\frac{d^3}{dx^3}\Re(G(z))$ at $x=1$ is equal to $-32(1-a)a/(1+a)^3<0$. Thus, we have that $\Re(G(z))$ is decreasing along these lines in a neighbourhood of $z=1$ (depending only on $a$, which is fixed). Thus, for $R$ close enough to $0$, we have that $C_{r,R}$ is steepest descent for $\Re(G(z))$. The rest of the argument is entirely analogous, and involves Taylor expanding $G(z)$ and $\log(z)$. Note in particular that everything in the integrand except for the $F(z)$ and $F(w)$ are identical.
\end{proof}

We now show that the convergence of the kernels implies convergence of the Fredholm Pfaffians.
\begin{proposition}
\label{prop: GSE cross pf conv}
Under the scalings of Assumption \ref{asmp: kernels}, we have
\begin{equation*}
    \Pf(J-K)_{l^2(\Z_{>s})}\to \Pf(J-K_{GSE})_{L^2(s,\infty)}
\end{equation*}
if $\nu>1$ is fixed, and
\begin{equation*}
    \Pf(J-K)_{l^2(\Z_{>s})}\to \Pf(J-K_{cross}^{(\xi)})_{L^2(s,\infty)}
\end{equation*}
if $\frac{1}{\nu}=1-\frac{2^{4/3}\xi}{\tau^{1/3}}$.
\end{proposition}
\begin{proof}
By Proposition \ref{prop: approx distr pf identity} and conjugating the kernel by $2^{-4/3}\tau^{1/3}$, it suffices to show convergence of the modified kernels
\begin{equation*}
    K'=\begin{pmatrix}
    2^{4/3}\tau^{-1/3}k(x,y)&-D_y^{(\gamma)}(x,y)\\
    -D_x^{(\gamma)}k(x,y)&2^{-4/3}\tau^{1/3}D_x^{(\gamma)}D_y^{(\gamma)} k(x,y)
    \end{pmatrix},
\end{equation*}
instead of $K$. We can replace the discrete functions by piecewise constant functions and sums by integrals, accounting for a global factor of $2^{4/3}\tau^{-1/3}$ for each variable. By Lemma \ref{lem: GSE cross ptwise conv}, we have pointwise convergence of the kernel. Moreover, the exponential bounds of Lemma \ref{lem: GSE cross ptwise conv} imply via Lemma \ref{lem: Hadamard} that
\begin{equation*}
    |\Pf(K'(x_i,x_j))_{i,j=1}^n|\leq C^nn^{n/2}e^{-c\sum u_i}.
\end{equation*}
This allows a dominated convergence argument to show that each term, and then the whole sum of the series expansion for $\Pf(J-K')_{l^2(\Z_{>s})}$ converges to that of $\Pf(J-K_{GSE})_{L^2(s,\infty)}$ or $\Pf(J-K_{cross}^{(\xi)})_{L^2(s,\infty)}$, depending on how $\nu$ is scaled.
\end{proof}

\begin{proposition}
\label{prop: GSE asymp}
For the ASEP, if $\rho=\frac{1}{1+\nu^{-1}}>\frac{1}{2}$, we have
\begin{equation*}
    \PP\left(-\frac{N\left(\frac{\tau}{1-q}\right)-\frac{\tau}{4}}{2^{-4/3} \tau^{1/3}}\leq s\right)\to F_{GSE}(s),
\end{equation*}
and if $\rho=\frac{1}{2}+\frac{2^{-2/3}\xi}{\tau^{1/3}}$, then
\begin{equation*}\tag{$\rho\downarrow \frac{1}{2}$}
    \PP\left(-\frac{N\left(\frac{\tau}{1-q}\right)-\frac{\tau}{4}}{2^{-4/3} \tau^{1/3}}\leq s\right)\to F_{cross}(s,\xi).
\end{equation*}
For the six vertex model, if $\rho=\frac{1}{1+\nu^{-1}}>\frac{1}{2}$, we have
\begin{equation*}
    \PP\left(\frac{h(n,n)-\frac{2a}{1+a}n}{\frac{(a(1-a))^{1/3}}{1+a}n^{1/3} }\leq s\right)\to F_{GSE}(s),
\end{equation*}
and if $\rho=\frac{1}{2}+\frac{2^{-2/3}\xi}{\tau^{1/3}}$, then
\begin{equation*}\tag{$\rho\downarrow \frac{1}{2}$}
    \PP\left(\frac{h(n,n)-\frac{2a}{1+a}n}{\frac{(a(1-a))^{1/3}}{1+a}n^{1/3} }\leq s\right)\to F_{cross}(s,\xi).
\end{equation*}
\end{proposition}
\begin{proof}
The proofs in all cases are essentially the same, so we give the details for one case. Note that by Proposition \ref{prop: GSE cross pf conv} and Corollary \ref{cor: ASEP formula}, we have the desired convergence for $-N(\tau)+\chi+2S$. In general, $\chi$ is a signed measure, but it has a convolutional inverse which is a genuine probability distribution, so by Lemma \ref{lem: signed measure asymp}, since $\chi$ has exponentially decaying tails, $-N(\tau)+\chi+2S$ has uniformly bounded tails, so it suffices to show that $\sup_{k}\PP(-N(\tau)+\chi+2S=k)\to 0$.

By Proposition \ref{prop: approx distr pf identity} (which is uniform over $k\geq -\frac{1-q}{4}\tau+2^{-4/3}(1-q)^{1/3}\tau^{1/3} s$ for fixed $s$), it suffices to instead consider
\begin{equation*}
    \Pf(J-K')_{l^2(\Z_{>k})}-\Pf(J-K')_{l^2(\Z_{>k-1})}
\end{equation*}
for the modified kernel $K'$. This converges to $0$
uniformly for $k\geq -\frac{1-q}{4}\tau+2^{-4/3}(1-q)^{1/3}\tau^{1/3} s$, which follows from Assumption \ref{asmp: kernels} which was established in Lemma \ref{lem: GSE cross ptwise conv}. In particular, the bounds established show via similar arguments to the convergence of the Fredholm Pfaffian that this is $O(\tau^{-1/3})$, since the sums over $n$ dimensional sets $x_i>s$ are replaced by a sum over the boundary.
\end{proof}

\subsection{GOE asymptotics}
\label{sec: GOE}
We now proceed to the GOE regime. Note that this was already established for the ASEP in \cite{BBCW18} in a special case. However, given the machinery that we have developed, the proof is not too difficult, and so we include it along with the proof for the six vertex model.

Unfortunately, the Fredholm Pfaffian formulas for the ASEP and six vertex model do not directly apply when $\nu=1$, due to the poles at $z=\nu^{-1}$ and $w=\nu^{-1}$ crossing the contour. However, this can be dealt with by subtracting this contribution using an analytic continuation argument. We first formally derive a Fredholm Pfaffian expansion. We will later show that this converges, and compute asymptotics.

\begin{theorem}
\label{thm: analytic cont fredholm pfaff}
Let $|\gamma_2|<1$ (we are taking $\gamma_1\sqrt{q}=1$), let $\lambda\sim \mathbb{FBS}_a^{(q,t,1)}$ with $p_n(a)=O(c^n)$ for some $c<1$, and let $S\sim \Thetad(\zeta^2,q^2)$. Then if $1<r'<r$ are close to $1$, we have
\begin{equation*}
    \PP(\lambda_1+2S\leq s)=\Pf(J-\widetilde{K})_{l^2(\Z_{>s})},
\end{equation*}
where $\widetilde{K}$ is defined in the same manner as $K$, but with
\begin{equation*}
    \widetilde{k}(x,y)=k(x,y)-B(y)+B(x),
\end{equation*}
with
\begin{equation*}
    B(y)=\frac{(1-\gamma_2)^{-2}}{(2\pi i)^2}\oint_{|w|=r'}-w^{-y-3/2}F(w)g(w)\frac{(q;q)_\infty}{(1/w^2;q)_\infty}\frac{\theta_3(\zeta^2w^2;q^2)}{\theta_3(\zeta^2;q^2)}dw,
\end{equation*}
assuming that the Fredholm Pfaffian on the right hand side has an absolutely convergent expansion.
\end{theorem}
\begin{proof}
First, note that $\mathbb{FBS}_a^{(q,t,\nu)}(\lambda_1=x)$ is an analytic function in $\nu$. Maybe the simplest argument is to note that the distribution function for $\mathbb{HL}_a^{(q,t,\nu)}$ is analytic (in fact rational) in $\nu$, and $\mathbb{FBS}_a^{(q,t,\nu)}$ is a convolution with a function supported on $\N$ by Theorems \ref{thm: general 6vm HL distr equality} and \ref{thm: HL FBS same}. Further convolving with $2S$, we still have an analytic distribution function.

Next, we note that as the integrand in $k(x,y)$ is analytic in $\nu$, so is the sum of its residues at the poles within the contour. The issue is that Theorem \ref{thm: pfaff formula} only applies when the poles $z=\nu$ and $w=\nu$ lie outside the contour, but when $\nu=1$, this is impossible. It's easy to see that the residues at $z=\nu$ and $w=\nu$ are given by $A(x)B(y)$ and $-A(y)B(x)$ respectively (the latter can be seen via skew symmetry of $k(x,y)$), where
\begin{equation*}
    A(x)=\nu^{x+3/2}F(\nu)\frac{(-t;q)_\infty}{(-qt\nu^2;q)_\infty}.
\end{equation*}
and
\begin{equation*}
    B(y)=\frac{(1-\gamma_2)^{-2}}{2\pi i}\oint_{|w|=r'}w^{-y-5/2}\frac{F(w)g(w)(q;q)_\infty(\nu^{-1}w;q)_\infty(q\nu/w;q)_\infty}{(1/w^2;q)_\infty(\nu^{-1}/w;q)_\infty(q\nu w;q)_\infty}\frac{\theta_3(\zeta^2\nu^2w^2;q^2)}{\theta_3(\zeta^2;q^2)}dw.
\end{equation*}

Thus, we subtract off the contributions from these poles, and obtain that
\begin{equation*}
    \widetilde{k}(x,y)=k(x,y)-A(x)B(y)+A(y)B(x)
\end{equation*}
extends $k(x,y)$ to an analytic function in $\nu$ near $\nu^{-1}=1$. At $\nu=1$, we can compute $A(x)=1-\gamma_2$ and
\begin{equation*}
    B(x)=\frac{(1-\gamma_2)^{-2}}{2\pi i}\oint_{|w|=r'}-w^{-x-3/2}F(w)g(w)\frac{(q;q)_\infty}{(1/w^2;q)_\infty}\frac{\theta_3(\zeta^2w^2;q^2)}{\theta_3(\zeta^2;q^2)}dw.
\end{equation*}
\end{proof}

Recall that Lemma \ref{lem: GSE cross ptwise conv} established $k(x,y)\to k_{GSE}(u,v)$ even if $\nu=1$. Here, we establish the analogous results for $B$.
\begin{lemma}
\label{lem: GOE conv}
Let $\nu=1$. Then with the same scalings as in Lemma \ref{lem: GSE cross ptwise conv}, we have for all $i$,
\begin{equation*}
    \left(2^{-4/3}\tau^{1/3}D_x^{(\gamma_2)}\right)^iB(u)\to \partial_u^i b(u),
\end{equation*}
and the bounds
\begin{equation}
\label{eq: B bound}
    \left|\left(2^{-4/3}\tau^{1/3}D_x^{(\gamma_2)}\right)^iB(u)\right|\leq Ce^{-cu},
\end{equation}
and
\begin{equation}
\label{eq: B diff bound}
    \left|\left(2^{-4/3}\tau^{1/3}D_x^{(\gamma_2)}\right)^i\left(2^{-4/3}\tau^{1/3}D_y^{(\gamma_2)}\right)^j(B(u+du)-B(u))\right|\leq C\tau^{-1/12}
\end{equation}
if $du\in [-\tau^{-1/6},\tau^{1/6}]$, uniformly in $u>s$.
\end{lemma}
\begin{proof}
Again, we proceed by a steepest descent analysis, which is extremely similar to that in the proof of Lemma \ref{lem: GSE cross ptwise conv}. Note that we may use the same contour $C_{r,R}$, since $F(w)$ has the same form. Making the same change of variables, we can check that the $w^{-x}F(w)\to e^{\frac{\alpha^3}{3}-\alpha u}$, $g(w)\to -(1-\gamma_2)$, and $\frac{2^{4/3}\tau^{-1/3}}{1-w^{-2}}\to \frac{1}{2\alpha}$, with the other factors going to $1$. 

The discrete derivatives amount to a change in the integrand by the same factors as before, and so can be dealt with in the same way. The bounds \eqref{eq: B bound} and \eqref{eq: B diff bound} can similarly be obtained from the steepest descent asymptotics.
\end{proof}
With this, we may proceed as in the previous cases. 
\begin{proposition}
\label{prop: GOE pf conv}
Under the scalings of Assumption \ref{asmp: kernels}, if $\nu=1$, we have that the series expansion for $\Pf(J-\widetilde{K})_{l^2(\Z_{>s})}$ converges absolutely, and
\begin{equation*}
    \Pf(J-\widetilde{K})_{l^2(\Z_{>s})}\to \Pf(J-K_{GOE})_{L^2(s,\infty)}.
\end{equation*}
\end{proposition}
\begin{proof}
By Proposition \ref{prop: approx distr pf identity}, it suffices to study the modified kernel
\begin{equation*}
    \widetilde{K}'=\begin{pmatrix}
    \widetilde{k}(x,y)&-D_y^{(\gamma)}\widetilde{k}(x,y)\\
    -D_x^{(\gamma)}\widetilde{k}(x,y)&D_x^{(\gamma)}D_y^{(\gamma)} \widetilde{k}(x,y)
    \end{pmatrix}.
\end{equation*}
Strictly speaking, Proposition \ref{prop: approx distr pf identity} does not apply, but as we describe below it can be modified to hold up to an additional $O(n^2)$ factor.

To see that the series expansion of the Fredholm Pfaffian converges absolutely, note that by Lemma \ref{lem: pf expansion}, we can write
\begin{equation*}
    \left|\Pf(\widetilde{K}'(x_i,x_j)_{i,j=1}^n\right|\leq \sum_{I}|\Pf(K')_I\Pf(B')_{I^c}|,
\end{equation*}
where $B'$ is the Pfaffian kernel defined by the function $B(y)-B(x)$, the subscripts denote which rows and columns are included, and the sum is over even subsets of $\{1,\dotsc, n\}$. Now $\Pf(B')_{I^c}=0$ unless $|I^c|=0,2$, since $B'$ is a rank 2 matrix (see e.g. Proposition A.1 of \cite{IMS22}). No matter which subset $I$ we pick, either $\Pf(K')_I$ will include all variables $x_i$, or it will miss exactly one, say $x_i$, in which case $\Pf(B')_{I^c}=D_{x_i}^{(\gamma_2)}B(x_i)$. Thus, using Lemma \ref{lem: Hadamard}, we can then bound
\begin{equation*}
    \sum_{I}|\Pf(K')_I\Pf(B')_{I^c}|\leq (2n)^{n/2} \left(1+{2n\choose 2}\right)C^ne^{-c\sum u_i},
\end{equation*}
which implies that the series expansion for $\Pf(\widetilde{K})_{l^2(\Z_{>s})}$ converges absolutely. Note this same argument establishes a version of Proposition \ref{prop: approx distr pf identity}. In particular, we may expand each boundary term further using Lemma \ref{lem: pf expansion}, again noting that each variable appears in two columns/rows, and only one will be used up in the $\Pf(B')_{I^c}$ factor, leaving the other one to ensure exponential decay in that variable giving integrability. There are $O(n^2)$ terms introduced this way, which does not affect the rest of the proof of Proposition \ref{prop: approx distr pf identity}.

The proof of convergence for $\Pf(J-\widetilde{K}')_{l^2(\Z_{>s})}$ proceeds in a similar manner as in the previous cases. In particular, the bound above is uniform in $\tau$, so we can use dominated convergence to turn the pointwise convergence of Lemma \ref{lem: GOE conv} into convergence of the Fredholm Pfaffians.
\end{proof}

Since the analogous statements for the six vertex model have an essentially identical proof (in particular, the only part of the integrand that changes is $F(z)$, but $F(1)=1$ still holds), we omit them. This then leads to the following asymptotics by a similar argument as in Proposition \ref{prop: GSE asymp}.
\begin{proposition}
For the ASEP, if $\rho=\frac{1}{2}$, we have 
\begin{equation*}\tag{$\rho=\frac{1}{2}$}
    \PP\left(-\frac{N\left(\frac{\tau}{1-q}\right)-\frac{\tau}{4}}{2^{-4/3} \tau^{1/3}}\leq s\right)\to F_{GOE}(s),
\end{equation*}
and for the six vertex model, if $\rho=\frac{1}{2}$, we have
\begin{equation*}\tag{$\rho=\frac{1}{2}$}
    \PP\left(\frac{h(n,n)-\frac{2a}{1+a}n}{\frac{(a(1-a))^{1/3}}{1+a}n^{1/3}}\leq s\right)\to F_{GOE}(s).
\end{equation*}
\end{proposition}

\section{Asymptotics: Gaussian asymptotics}
\label{sec: Gauss}
In this section, we establish Gaussian fluctuations in the $\rho<\frac{1}{2}$ regime. However, as we have already seen in the GOE case, this causes certain poles to cross the contour and these residues must be removed. In the GOE case, there was only one pole and the asymptotic analysis was similar to the GSE case. In the Gaussian case, many poles will cross, and more serious convergence issues will arise which must be dealt with.

\subsection{Analytic continuation}
We now wish to extend the Pfaffian formulas in Theorem \ref{thm: pfaff formula}, valid for $\nu^{-1}<1$ to all $\nu>0$. This requires analytic continuation in the parameter $\nu$, which amounts to removing the extra poles coming from the $(\gamma_1\sqrt{q}z;q)_\infty$ and $(\gamma_1\sqrt{q}w;q)_\infty$ factors in $k(x,y)$. We have actually already encountered this in the GOE regime.

In the Gaussian case, this is complicated by the fact that naively doing so, the resulting expressions are exponentially large in $x$ and $y$, requiring additional work to show that we actually obtain a convergent formula. 

We let $\nu_-<\nu$ denote a number close to but smaller than $\nu$, and let $k_{\nu^{-1}}$ denote the same contour integral defining $k$ except with $\nu^{-1}<r'<r$, and $r',r$ close to $\nu^{-1}$. The reason for this is both that the critical point now lies at $\nu^{-1}$ rather than $1$, and to give sufficient decay to allow the Fredholm Pfaffian to converge. This in addition to the fact that $\nu<1$ introduces many poles which must be dealt with. We first determine the relevant residues.

\begin{lemma}
\label{lem: q-poch identity}
We have
\begin{equation*}
    (-z)^k q^{k\choose 2}(zq^{k};q)_\infty(q^{1-k}/z;q)_\infty=(z;q)_\infty(q/z;q)_\infty.
\end{equation*}
\end{lemma}
\begin{proof}
We proceed by induction. The base case $k=0$ is trivial. We have the identity
\begin{equation*}
    -z(zq;q)_\infty(1/z;q)_\infty=(z;q)_\infty(q/z;q)_\infty,
\end{equation*}
and plugging in $z=zq^{k-1}$ gives
\begin{equation*}
\begin{split}
    (-z)^k q^{k\choose 2}(zq^{k};q)_\infty(q^{1-k}/z;q)_\infty&=(-z)^{k-1}q^{k-1\choose 2}(zq^{k-1};q)_\infty(q^{1-(k-1)}/z;q)_\infty
    \\&=(z;q)_\infty(q/z;q)_\infty
\end{split}
\end{equation*}
by the inductive hypothesis.
\end{proof}

\begin{lemma}
\label{lem: residues}
Assume that $\nu^{-1}>1$. The residue of $k_{\nu^{-1}}(x,y)$ at $z=q^{-k}\nu$ is $A_k(x)B(y)$ and the residue of $k_{\nu^{-1}}(x,y)$ at $w=q^{-k}\nu$ is $-A_k(y)B(x)$, where
\begin{equation*}
    A_k(x)=-\left(q^{-k}\nu\right)^{-x-3/2}\frac{F(q^{-k}\nu)q^{-k^2-k}\zeta^{2k}\nu^{2k+1}(q;q)_\infty(q^k\nu^{-2};q)_\infty(-tq^k;q)_\infty}{(q^{-k};q)_{k-1}(-tq^{1-k}\nu^2;q)_\infty(q^{2k}\nu^{-2};q)_\infty\theta_3(\zeta^2;q^2)},
\end{equation*}
and
\begin{equation*}
    B(y)=\frac{(1-\gamma_2)^2}{2\pi i}\oint_{|w|=r'}w^{-y-5/2}F(w)\frac{(\nu^{-1} w;q)_\infty(q^{1}\nu/w;q)_\infty\theta_3(\zeta^2\nu^2w^2;q^2)}{(1/w^2;q)_\infty(\nu^{-1}/w;q)_\infty(q^{1}\nu w;q)_\infty}g(w)dw,
\end{equation*}
where $r'>\nu^{-1}$ is close to $\nu^{-1}$.

The residue of $k_{\nu^{-1}}(x,y)$ at $z=q^{-k}w^{-1}$ is given by
\begin{equation*}
\begin{split}
    S_k(x,y)=&-\frac{(1-\gamma_2)^{-2}}{2\pi i}\oint _{|w|=r'}(q^kw)^{x+1/2} w^{-y-5/2}F(w)F(q^{-k}w^{-1})
    \\&\qquad \times w^{-2k}\frac{(w^2;q)_\infty(q/w^2;q)_\infty}{(q^{2k}w^2;q)_\infty(1/w^2;q)_\infty}\frac{\theta_3(\zeta^2q^{-2k};q^2)}{\theta(\zeta^2;q^2)}g(w)g(q^{-k}w^{-1})dw,
\end{split}
\end{equation*}
where $1<r'<q^{-1}\nu^{-1}$, and satisfies $S_k(x,y)=-S_k(y,x)$.
\end{lemma}
\begin{proof}
Since $k_{\nu^{-1}}(x,y)$ is skew-symmetric, the first two residue computations are completely equivalent, so we give one. Note that pole occurs in the $K(z,w)$ factor, so we compute the residue there first, obtaining
\begin{equation}
\label{eq: residue long}
\begin{split}
    &-\nu q^{-k}g(w)\frac{(q^k\nu^{-1} w;q)_\infty(q^{1-k}\nu/w;q)_\infty\theta_3(\zeta^2q^{-2k}\nu^2 w^2;q^2)}{(1/w^2;q)_\infty(q^k\nu^{-1}/w;q)_\infty(q^{1-k}\nu w;q)_\infty}
    \\&\qquad\qquad\times\frac{(q;q)_\infty(q^k\nu^{-2};q)_\infty(-tq^k;q)_\infty}{(q^{-k};q)_{k-1}(-tq^{1-k}\nu^2;q)_\infty(q^{2k}\nu^{-2};q)_\infty\theta_3(\zeta^2;q^2)}.
\end{split}
\end{equation}
We then use Lemma \ref{lem: q-poch identity}, which in particular implies that
\begin{equation*}
    (q\zeta^{-2}\nu^{-2}w^{-2})^kq^{k(k-1)}\theta_3(\zeta^2 q^{-2k}\nu^2 w^2;q^2)=\theta_3(\zeta^2 \nu^2 w^2;q^2),
\end{equation*}
to simplify \eqref{eq: residue long} to
\begin{equation*}
\begin{split}
    &-\nu q^{-k}g(w)\frac{(\nu^{-1} w;q)_\infty(q^{1}\nu/w;q)_\infty\theta_3(\zeta^2\nu^2w^2;q^2)}{(1/w^2;q)_\infty(\nu^{-1}/w;q)_\infty(q^{1}\nu w;q)_\infty}q^{-k^2}\zeta^{2k}\nu^{2k}
    \\&\qquad\qquad\times\frac{(q;q)_\infty(q^k\nu^{-2};q)_\infty(-tq^k;q)_\infty}{(q^{-k};q)_{k-1}(-tq^{1-k}\nu^2;q)_\infty(q^{2k}\nu^{-2};q)_\infty\theta_3(\zeta^2;q^2)}.
\end{split}
\end{equation*}
Combining this with the rest of the integrand in $k(x,y)$ gives the desired result. Note that the radius $r'$ can be expanded up until $\nu^{-1}$, since there are no poles obstructing this.

For the final residue computation, again the pole occurs in $K(x,y)$, and there using Lemma \ref{lem: q-poch identity} we have a residue of
\begin{equation*}
    -q^{-k}w^{-1}w^{-2k}\frac{(w^2;q)_\infty(q/w^2;q)_\infty}{(q^{2k}w^2;q)_\infty(1/w^2;q)_\infty}\frac{\theta_3(\zeta^2q^{-2k};q^2)}{\theta(\zeta^2;q^2)}g(w)g(q^{-k}w^{-1}).
\end{equation*}
Combining this with the other factors in the integrand gives the desired result. In particular, the radius $r'$ can be extended due to the lack of poles obstructing this. The skew-symmetry of $S_k(x,y)$ follows from that of the integrand defining $k(x,y)$. 
\end{proof}

We now give an analytic continuation of the Fredholm Pfaffian formulas in the Gaussian case. Note that we will later establish convergence of the series expansion, so this is purely formal. 
\begin{theorem}
\label{thm: pfaff analytic cont}
Let $\gamma_1\sqrt{q}\in [q^{-m+\frac{2-\varepsilon}{2}},q^{-m+\frac{1-\varepsilon}{2}})$ for some $m\in \N$ and $\varepsilon=0,1$, and assume that $|\gamma_2|<1$ and $t<1$. Let
\begin{equation*}
    S(x,y)=\sum_{1\leq k\leq 2m+\varepsilon-2}S_k(x,y),\qquad A(x)=\sum_{j=0}^{2m+\varepsilon-2}A_j(x).
\end{equation*}

Then we have
\begin{equation*}
    \PP(\lambda_1+2S\leq s)=\Pf(J-\widehat{K})_{l^2(\Z_{>s})},
\end{equation*}
where
\begin{equation*}
\begin{split}
    \widehat{K}&=\begin{pmatrix}
    \widehat{k}(x,y)&-2D_y^{(\gamma)}\widehat{k}(x,y)\\
    -2D_x^{(\gamma)}\widehat{k}(x,y)&4D_x^{(\gamma)}D_y^{(\gamma)} \widehat{k}(x,y)+(D_x^{(\gamma)}-D_y^{(\gamma)})\Delta
    \end{pmatrix}
    \\\widehat{k}&=k_{\nu^{-1}}(x,y)-S(x,y)-A(x)B(y)+A(y)B(x),
\end{split}
\end{equation*}
whenever the series expansion on the right converges absolutely.
\end{theorem}
\begin{proof}
Since the left hand side is analytic in $\gamma_2$, it suffices to show that the right hand side agrees with an analytic extension of the formula in Theorem \ref{thm: pfaff formula}. Since we assume the series converges absolutely, it suffices to show the same for $\widehat{K}$.

We note that the residues at all poles are analytic in $\gamma_2$, and so it suffices to show that $\widehat{K}$ is obtained from $K$ by removing any additional poles that are introduced when $\nu^{-1}>1$ and by moving the contour defining $k(x,y)$. It's easy to see that the poles introduced are at $w,z=q^{-k}\nu$, for $0\leq k\leq 2m+\varepsilon -2$, and at $z=q^{-k}w^{-1}$, for $1\leq k\leq 2m+\varepsilon-2$. The perturbations $S$ and $A(x)B(y)-A(y)B(x)$ are exactly defined to cancel these poles, noting that the potential double count of poles of the form $z=q^{-k}\nu$ and $w=q^{-l}\nu$ is not an issue, since the factor $(w/z;q)_\infty(qz/w)_\infty$ in the numerator cancels these potential poles. The assumption that $t<1$ prevents the poles coming from the $(-qt\nu w;q)_\infty$ and $(-qt\nu z;q)_\infty$ factors from interfering.
\end{proof}

\subsection{Asymptotics}
We now study the asymptotics of the analytically continued Pfaffian kernel.
\begin{lemma}
\label{lem: asep gaus asymp}
Let 
\begin{equation*}
    \mu=\frac{\nu}{(1+\nu)^2}, \qquad \sigma^2=\nu^{-2}\frac{1-\nu}{(1+\nu^{-1})^3},
\end{equation*}
let $G(w)=\frac{1}{2}\frac{1-w}{1+w}+\mu\log w$, and let
\begin{equation*}
    x=\mu (1-q)\tau +\sigma (1-q)^{1/2}\tau^{1/2}u,\qquad y=\mu (1-q)\tau +\sigma (1-q)^{1/2}\tau^{1/2}v.
\end{equation*}

For the ASEP, we have
\begin{equation*}
\begin{split}
    &\left|(D_x^{(\gamma_2)})^i(D_y^{(\gamma_2)})^j\left(k_{\nu^{-1}}(u,v)-S(u,v)\right)\right|
    \\\leq &Ce^{-c\tau}\min\left(\nu_-^{\sigma\tau^{1/2}u}(q\nu_-^{-1})^{\sigma\tau^{1/2}v},\nu_-^{\sigma\tau^{1/2}v}(q\nu_-^{-1})^{\sigma\tau^{1/2}u}\right),
\end{split}
\end{equation*}
uniformly in $u,v>s$, where $\nu_-<\nu$ is close to $\nu$, and
\begin{equation*}
\begin{split}
    (D_x^{(\gamma_2)})^iA(u)&=-(1+o(1))\left(-\frac{(1-\gamma_2)^2}{2}\frac{\nu-\nu^{-1}}{(1-\gamma_2\nu)(1-\gamma_2\nu^{-1})}\right)^i
    \\&\qquad \times\nu^{-\sigma\tau^{1/2}u}\exp(\tau G(\nu))\nu^{-1/2}\frac{(q;q)_\infty(-t;q)_\infty}{(-qt\nu^2;q)_\infty\theta_3(\zeta^2;q^2)},
    \\(D_x^{(\gamma_2)})^iB(u)&=-(1+o(1))\left(\frac{(1-\gamma_2)^2}{2}\frac{\nu-\nu^{-1}}{(1-\gamma_2\nu)(1-\gamma_2\nu^{-1})}\right)^{i-1}
    \\&\qquad \times \nu^{\sigma \tau^{1/2}u}e^{-\tau G(\nu)}\nu^{1/2}\frac{(-qt\nu^2;q)_\infty\theta_3(\zeta^2;q^2)}{(q;q)_\infty(-t;q)_\infty}\frac{1}{2\sqrt{2\pi}}\frac{\tau^{-1/2}}{\sigma}e^{-u^2/2}
    \\&\qquad +O(e^{\tau(G(\nu^{-1})-c)}\nu_-^{\sigma \tau^{1/2}u}).
\end{split}
\end{equation*}

\end{lemma}
\begin{proof}
Let us rescale $\tau$ by $1-q$. We can compute $G'(\nu^{-1})=0$, $G''(\nu^{-1})=\nu^{2}\sigma^2>0$, and moreover $G(\nu^{-1})<0$ since as a function of $\nu^{-1}$, it is $0$ at $1$ and decreasing for $\nu^{-1}>1$. Also, note that if $z=re^{i\theta}$, we have that $\Re(G(z))$ is maximized at $\theta=0$ (the computation is essentially the same as the GSE case, since the only change is to the $\log z$ term whose real part is independent of $\theta$).

We notice that after rescaling $x$ and $y$ to $u,v$, the integral for $k_{\nu^{-1}}(u,v)$ is of the form
\begin{equation*}
    \oint_{|z|=r}\oint_{|w|=r'}\exp(\tau(G(z)+G(w))z^{-\sigma\tau^{1/2}u}w^{-\sigma\tau^{1/2}v}F(z,w)dzdw,
\end{equation*}
where $F(z,w)$ is bounded on the contour, and $\nu_{-}^{-1}\leq r<r'$ are larger than but close to $\nu^{-1}$. Then the claimed bound for $k_{\nu^{-1}}$ follows immediately, since $\Re(G(z)+G(w))<2G(\nu^{-1})<0$, $|z|,|w|\geq (\nu_-)^{-1}>\nu^{-1}$, and the rest of the integrand is bounded.

The integral for $S_k(u,v)$ is of the form
\begin{equation*}
    \oint_{|w|=r'}(q^kw)^{\sigma \tau^{1/2}u}w^{-\sigma\tau^{1/2}v}\exp(\tau(G(w)+G(q^{-k}w^{-1})))H(w)dw
\end{equation*}
for some $r'>\nu^{-1}$ but close to $\nu^{-1}$, and $H(w)$ bounded on the contour. If $w=re^{i\theta}$, we have that $\Re(G(w))+\Re(G(q^{-k}w^{-1}))$ is negative if $q^{-k}r^{-1}>1$, and otherwise attains its maximum at $\theta=0$. There, we have that $\Re(G(1/r))=-\Re(G(r))$, and $\frac{d}{dr}\Re(G(r))<0$ on $(\nu,\nu^{-1})$. Together, this implies that $\Re(G(w)+G(q^{-k}w^{-1}))<0$ for $|w|=r'$ with $r'=\nu^{-1}$, and so by continuity also for $r'$ slightly larger. This gives the desired bound. The other bound for $S(u,v)$ follows from skew-symmetry.

Next, note that
\begin{equation*}
    A_k(u)=\left(q^{-k}\nu\right)^{-\sigma\tau^{1/2}u}\exp(\tau G(q^{-k}\nu))C_k
\end{equation*}
where $C_k$ is not a function of $u$. For $k>0$, since we have already seen that $G(\nu)=-G(\nu^{-1})>G(q^{-k}\nu)$ for $k\leq 2m+\varepsilon-2$, this is bounded by
\begin{equation*}
    C_k\left(q^{-k}\nu\right)^{-\sigma\tau^{1/2}u}e^{-c\tau}
\end{equation*}
for some $c>0$. On the other hand,
\begin{equation*}
    A_0(u)=-\nu^{-\sigma\tau^{1/2}u}\exp(\tau G(\nu))\nu^{-1/2}\frac{(q;q)_\infty(-t;q)_\infty}{(-qt\nu^2;q)_\infty\theta_3(\zeta^2;q^2)},
\end{equation*}
which immediately gives the claimed asymptotics for $A(u).$

\begin{figure}
    \centering
    \includegraphics[scale=0.7]{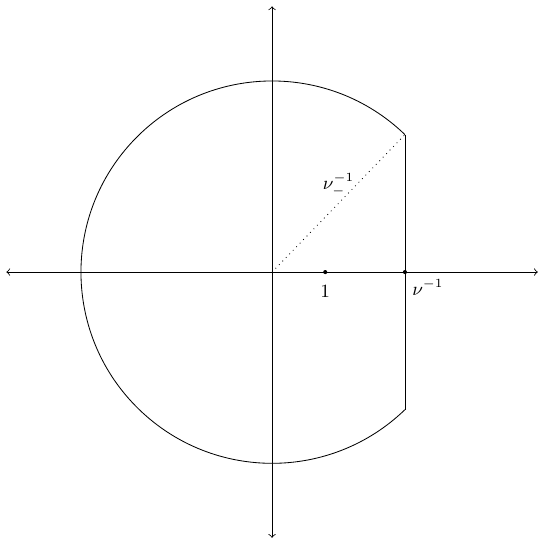}
    \caption{The contour for the Gaussian regime.}
    \label{fig:contour2}
\end{figure}

Finally, we obtain asymptotics for $B(u)$ through a steepest descent argument. We have that the contour $|w|=\nu^{-1}$ is steepest descent for the function $\Re(G(w))$, and this will remain true if we take a circle of radius $\nu_-^{-1}$ slightly larger than $\nu^{-1}$ and flatten it at $\Re(z)=\nu^{-1}$ to a vertical line (see Figure \ref{fig:contour2}). The error from the circular part of the contour is $O(e^{\tau(G(\nu^{-1})-c)}\nu_-^{\sigma \tau^{1/2}u})$ for some $c>0$ and $\nu_-<\nu$, and so we now focus on the contour near $\nu^{-1}$. Thus, we can localize the integral to a neighbourhood of $\nu^{-1}$ with the change of variables $w=\nu^{-1}+i\nu^{-1}\tau^{-1/2}\alpha/\sigma$ and expand $G(w)$ using a Taylor series, with the exponential factor becoming
\begin{equation*}
    \nu^{\sigma\tau^{1/2}u}\exp\left(\tau G(\nu^{-1})-i\alpha u-\alpha^2/2+O(\tau^{-1/2})\right).
\end{equation*}
The remaining factors in the integrand give
\begin{equation*}
    \nu^{3/2}\frac{(1-\gamma_2\nu^{-1})(1-\gamma_2\nu)}{(\nu^{-1}-\nu)}\frac{(-qt\nu^2;q)_\infty\theta_3(\zeta^2;q^2)}{(q;q)_\infty(-t;q)_\infty}(1+O(\tau^{-1/2})),
\end{equation*}
uniformly since the contour is far from any singularities due to the cancellation of the $(\nu^{-1}/w;q)_\infty$ factor. This gives
\begin{equation*}
\begin{split}
    B(u)&=\nu^{3/2}\frac{(1-\gamma_2\nu^{-1})(1-\gamma_2\nu)}{(\nu^{-1}-\nu)}\frac{(-qt\nu^2;q)_\infty\theta_3(\zeta^2;q^2)}{(q;q)_\infty(-t;q)_\infty}(1+O(\tau^{-1/2}))
    \\&\qquad \times\frac{(1-\gamma_2)^{-2}}{2\pi}\frac{\nu^{-1}\tau^{-1/2}}{\sigma}\int_{-\infty}^\infty \nu^{\sigma\tau^{1/2}u}\exp\left(\tau G(\nu^{-1})-i\alpha u-\alpha^2/2\right)d\alpha
    \\&=\nu^{1/2}\frac{(1-\gamma_2\nu^{-1})(1-\gamma_2\nu)}{(\nu^{-1}-\nu)}\frac{(-qt\nu^2;q)_\infty\theta_3(\zeta^2;q^2)}{(q;q)_\infty(-t;q)_\infty}(1+O(\tau^{-1/2}))
    \\&\qquad \times \nu^{\sigma \tau^{1/2}u}e^{-\tau G(\nu)}\frac{(1-\gamma_2)^{-2}}{\sqrt{2\pi}}\frac{\tau^{-1/2}}{\sigma}\exp(-u^2/2),
\end{split}
\end{equation*}
where we use $G(\nu^{-1})=-G(\nu)$, plus an error coming from the rest of the contour, which is $O(e^{\tau(G(\nu^{-1})-c)}\nu_-^{\sigma \tau^{1/2}u})$.

Note that in all cases, the effect of the difference operator $D_x^{(\gamma_2)}$ is multiplication by a factor of the form
\begin{equation*}
    \frac{w-w^{-1}}{(1-\gamma_2w)(1-\gamma_2w^{-1})},
\end{equation*}
which has no relevant poles since $t<1$, and so a similar analysis establishes the analogous results with $D_x^{(\gamma_2)}$ applied.

\end{proof}

\begin{lemma}
\label{lem: 6vm gaus asymp}
Let 
\begin{equation*}
    \mu=\frac{2a^2+a(\nu+\nu^{-1})}{(1+a\nu)(1+a\nu^{-1})}, \qquad \sigma^2=\frac{a(1-a^2)(\nu^{-1}-\nu)}{(1+a\nu)^2(1+a\nu^{-1})^2},
\end{equation*}
let $G(z)=\log(1+az)-\log(1+a/z)+\mu\log z$, and let
\begin{equation*}
    x=\mu n +\sigma n^{1/2}u,\qquad y=\mu n +\sigma n^{1/2}v.
\end{equation*}
For the six vertex model, the conclusions of Lemma \ref{lem: asep gaus asymp} continue to hold (with the above definitions replacing those in Lemma \ref{lem: asep gaus asymp} and $n$ replacing $\tau$).    
\end{lemma}
\begin{proof}
The proof is extremely similar, and so we emphasize the key differences. Except in the steepest descent analysis, the analysis is essentially the same, since the only change is to the definition of $F$ and the slightly different scaling. We thus focus on the properties of $G(z)$ required for the analysis.

We note that $G'(\nu^{-1})=0$, $G''(\nu^{-1})=\nu^{2}\sigma^2$, $G(z^{-1})=-G(z)$, and using the computation for the derivative in the GSE case, if $z=re^{i\theta}$, then for $r>1$, $\Re(G(z))$ is decreasing on circles away from the real axis. Furthermore, $\frac{d}{dr}G(r)<0$ on $(\nu,\nu^{-1})$, which along with $G(1/r)=-G(r)$ implies $G(\nu^{-1})<0$. These were the only properties used to establish the upper bounds for $k_{\nu^{-1}}$, $S$, and the asymptotics for $A$. The steepest descent analysis used to study $B(x)$ is also similar. In particular, the only part of the integrand that changes is $G$ (and a corresponding change in $\sigma$). The new pole $-a$ is located inside the unit circle, and so do not prevent the needed contour changes.
\end{proof}

\begin{proposition}
\label{prop: gauss pfaff conv}
For both the ASEP and the six vertex model, if $t<1$ and $\nu<1$, we have that the series expansion for $\Pf(J-K)_{l^2(\Z_{>s})}$ converges absolutely, and that
\begin{equation*}
    \Pf(J-\widehat{K})_{l^2(\Z_{>s})}\to \Phi(s),
\end{equation*}
where $\Phi$ is the distribution function for a standard Gaussian.
\end{proposition}
\begin{proof}
Since both convergence statements have essentially the same proof, we focus on the ASEP case. Let us first explain the heuristics behind the limit. Note that by Lemma \ref{lem: asep gaus asymp} and Lemma \ref{lem: 6vm gaus asymp}, in the limit we have that $k_{\nu^{-1}}(u,v)\to 0$ and $S(u,v)\to 0$ (along with all derivatives). On the other hand, if these are taken to $0$, what remains is a rank $2$ kernel, and so every term in the Fredholm Pfaffian expansion except the first two are $0$. One can then check that these two terms give the distribution function for a Gaussian. We now justify this limiting procedure.

We first wish to use summation by parts to remove the part of the kernel converging to a distribution, as in Proposition \ref{prop: approx distr pf identity}. The proof proceeds similarly, except that care must be taken to show integrability due to the different bounds. We will assume this for now, and return to this point at the end of the proof. Thus, we will first show convergence of the modified kernel
\begin{equation*}
    \widehat{K}'=\begin{pmatrix}
    \widehat{k}(x,y)&-D_y^{(\gamma)}\widehat{k}(x,y)\\
    -D_x^{(\gamma)}\widehat{k}(x,y)&D_x^{(\gamma)}D_y^{(\gamma)} \widehat{k}(x,y)
    \end{pmatrix}.
\end{equation*}

We have the $n$th term in the expansion is given by
\begin{equation*}
    \Pf(\widehat{K}'(u_i,u_j))_{i,j=1}^n=\Pf(K_1)+\sum_{i,j}\Pf(K_1)_{\widehat{i},\widehat{j}}\Pf(K_2)_{i,j},
\end{equation*}
where $K_1$ is the part of the kernel corresponding to $k_{\nu^{-1}}(u,v)-S(u,v)$ and $K_2$ to $-A(u)B(v)+A(v)B(u)$. We only need to consider even subsets up to size $2$ because $K_2$ is rank at most 2. 

Using the expansion of the Pfaffian into matchings, we have 
\begin{equation*}
    |\Pf(K_1)|\leq \frac{(2n)!}{2^nn!}(Ce^{-c\tau})^ne^{-c(\sum u_i)},
\end{equation*}
which we obtain by taking a geometric average of the two bounds given by Lemmas \ref{lem: asep gaus asymp} and \ref{lem: 6vm gaus asymp}. For the other term, we note that each variable is associated to two rows/columns, and so expanding $\Pf(K_1)_{\widehat{i},\widehat{j}}$ into a sum over matchings $M$ of $\{1,\dotsc,2n\}\setminus \{i,j\}$ and taking the induced matching on variables, we see that there will be factors corresponding to a path $u_{i_1}\to \dotsc \to u_{i_k}$ where $u_{i_1}$ is associated to $i$ and $u_{i_k}$ is associated to $j$ (possibly the same variable). Taking the same geometric average as above for the factors not in this path, and one of the bounds for the factors in the path, we obtain
\begin{equation*}
    \left|\prod_{(i',j')\in M}(K_1)_{i'j'}\right|\leq (Ce^{-c\tau})^ne^{-c\sum_{l\neq i} u_l}\nu^{\sigma \tau^{1/2}(u_i-u_j)}
\end{equation*}
and
\begin{equation*}
    \left|\prod_{(i',j')\in M}(K_1)_{i'j'}\right|\leq (Ce^{-c\tau})^ne^{-c\sum_{l\neq j} u_l}\nu^{\sigma \tau^{1/2}(u_j-u_i)}.
\end{equation*}
Then as
\begin{equation*}
    |\Pf(K_2)_{i,j}|\leq C\nu^{-\sigma\tau^{1/2}u_i+\sigma\tau^{1/2}u_j}e^{-cu_i}+C\nu^{-\sigma\tau^{1/2}u_j+\sigma\tau^{1/2}u_i}e^{-cu_j},
\end{equation*}
we have that
\begin{equation*}
    |\prod_{(i',j')\in M}(K_1)_{i'j'}\Pf(K_2)_{i,j}|\leq (Ce^{-c\tau})^{n-1}e^{-c\sum u_i},
\end{equation*}
and so after summing over the matchings $M$ and summing over the variables $u_i$, we obtain a bound of
\begin{equation*}
    C\frac{(2n)!}{n!}c(\tau)^{n-1},
\end{equation*}
for some $c(\tau)$ arbitrarily small as $\tau\to \infty$. This implies that the $n\geq 2$ terms in the series expansion for the Fredholm determinant are bounded by
\begin{equation*}
    C\sum_{n\geq 2}n^2{2n\choose n}c(\tau)^{n-1}\to 0
\end{equation*}
as $c(\tau)\to 0$, after summing over $i,j$ and $n$.

Turning now to the $n=1$ term, the summands have exponential tails in $u$, and so after replacing the sum with an integral and using a dominated convergence argument, it suffices to take the pointwise limit. The $k_{\nu^{-1}}-S$ term goes to $0$, and $A(u)(D_{x}^{(\gamma_2)}B)(u)-(D_x^{(\gamma_2)}A)(u)B(u)$ converges to $\frac{1}{\sqrt{2\pi}}e^{-u^2/2}$ (taking into account the scaling of the variables which turns the sum into an integral), giving
\begin{equation*}
    \Pf(J-\widehat{K})_{l^2(\Z_{>s})}\to 1-\int_{s}^\infty \frac{1}{\sqrt{2\pi}}e^{-u^2/2}du=\Phi(s).
\end{equation*}

Let us now explain why Proposition \ref{prop: approx distr pf identity} continues to hold in this setting. Again, we apply Lemma \ref{lem: boundary bound}, and note that by an argument similar to above, the first term can be made small enough even after integrating, noting that integrability for the remaining variables is given by a similar argument to above. In particular, although some variables are removed, variables are also set equal, resulting in a Pfaffian where each variable occurs exactly twice.

We thus study the second term, which involves bounds near when some of the variables are set close to $s$. Since the higher order main terms don't contribute, we expect that these error terms should also go to $0$. We use the same ideas as in the bounds derived above to do so. Recall that we expand the Pfaffian as a sum over an even subset $J$ of $\{1,\dotsc, n\}$ indicating the variables for which we select the distributional part of the kernel, $M$ a matching of $J$ given by further expanding the Pfaffian of the distributional part, and $M'\subseteq M$ a non-empty sub-matching indicating which variables we are selecting to be boundary variables. Note that since we are interested in bounding the boundary error from summation by parts, we can assume $M'$ is non-empty. To bound the boundary term
\begin{equation*}
    \sum_{\substack{u_i>s\\i\not\in J\text{ or }(i,j)\in M}}\left|B^{(\gamma_2)}_{(M')^c}D_{M'}^{(\gamma_2)}\Pf(\widehat{K}''(u_i,u_j))_{(2J)^c}\bigg|_{\substack{u_i=u_j=s, (i,j)\in (M')^c\\ u_i=u_j, (i,j)\in M'}}\right|,
\end{equation*}
where $\widehat{K}''$ is $\widehat{K}$ with the distributional part removed, we proceed as before, expanding further as a sum over a product of Pfaffians. Again, although we have removed certain rows/columns, we also set certain variables equal to each other and so each variable remaining still appears exactly twice, allowing the same argument as before to work. For the variables set close to $s$, we note that they are essentially equal, up to a lower order perturbation in the exponent which is killed by any $e^{-c\tau}$ factor. We thus obtain a bound of the form
\begin{equation*}
    \frac{(2(n-|J|))!}{2^{n-|J|}(n-|J|)!}\tau^{-|(M')^c|/12}C^nc(\tau)^{n-|J|/2-1}.
\end{equation*}
Let $k=|J|/2$, and sum over even subsets $J$, matchings $M$ of $J$, and $M'\subseteq M$ non-empty, giving that the boundary terms for $n\geq 3$ altogether contribute an error of at most
\begin{equation*}
\begin{split}
    &\sum_{n=3}^\infty \frac{1}{n!}C\tau^{-1/12}\sum_{k=0}^{n/2}{n\choose 2k}\frac{(2k)!}{2^kk!}2^k\frac{(2n-4k)!}{2^{n-k}(n-2k)!}C^nc(\tau)^{n-k-1}
    \\=&C\tau^{-1/12}\sum_{l=2}^\infty \sum_{k=0}^l\frac{(2l)!}{k!l!^2}C^{l+2k}c(\tau)^{l+k-1}
    \\ \leq & C\tau^{-1/12}\sum_{l=2}^\infty {2l\choose l}(Cc(\tau))^{l-1}\to 0,
\end{split}
\end{equation*}
where we made the substitutions $l=n-2k$. 

Finally, the error from the $n=2$ term can be estimated directly. Note that since $n=2$, we must have $2k=2$ and so there is exactly one boundary term coming from $\widehat{k}(x,y)$. In particular, it suffices to show that $\widehat{k}(u,v)$ is small near $(s,s)$. Again, the $k_{\nu^{-1}}$ and $S$ terms go to $0$ already and so it's easy to see that their boundary terms will also go to $0$, and so we focus on the remaining $A(x)B(y)-A(y)B(x)$ term. The higher order $A_k$ are also small, and so it suffices to look at $A_0$ instead of $A$. We do so by explicitly computing the boundary term, which is $(B_x^{(\gamma_2)}-B_y^{(\gamma_2)})(A_0(x)B(y)-A_0(y)B(x))$. Now $(B_x^{(\gamma_2)}-B_y^{(\gamma_2)})$ acts linearly, and
\begin{equation*}
\begin{split}
    &(B_x^{(\gamma_2)}-B_y^{(\gamma_2)})(w^xz^y)(s,s)
    \\=&\frac{(1-\gamma_2)^2}{2}\frac{(wz)^{s+1}}{1-wz}\left(\frac{w^{-1}}{1-\gamma_2 w^{-1}}-\frac{z}{1-\gamma_2 z}-\frac{z^{-1}}{1-\gamma_2z^{-1}}+\frac{w}{1-\gamma_2 w}\right),
\end{split}
\end{equation*}
so the boundary term is of the exact same form as $A_0(x)B(y)-A_0(y)B(x)$ except that we add this extra factor to the integral defining $B$ with $w=\nu$. A very similar steepest descent analysis then gives the same types of asymptotics as Lemmas \ref{lem: asep gaus asymp} or \ref{lem: 6vm gaus asymp}, but since we are not summing over anything anymore, the $\tau^{-1/2}$ factor means this term also goes to $0$.
\end{proof}

\begin{proposition}
For the ASEP, if $\rho<\frac{1}{2}$ and $t<1$, we have
\begin{equation*}
    \PP\left(-\frac{N\left(\frac{\tau}{1-q}\right)-\mu \tau}{\sigma \tau^{1/2}}\leq s\right)\to \Phi(s),
\end{equation*}
where
\begin{equation*}
    \mu=\frac{\nu}{(1+\nu)^2}, \qquad \sigma^2=\nu^{-2}\frac{1-\nu}{(1+\nu^{-1})^3}.
\end{equation*}
For the six vertex model, if $\rho<\frac{1}{2}$ (we always assume $t<1$), we have
\begin{equation*}
    \PP\left(\frac{h(n,n)-\mu n}{\sigma n^{1/2}}\leq s\right)\to \Phi(s),
\end{equation*}
where
\begin{equation*}
    \mu=\frac{2a^2+a(\nu+\nu^{-1})}{(1+a\nu)(1+a\nu^{-1})}, \qquad \sigma^2=\frac{a(1-a^2)(\nu^{-1}-\nu)}{(1+a\nu)^2(1+a\nu^{-1})^2}.
\end{equation*}
\end{proposition}

\begin{proof}
Given Proposition \ref{prop: gauss pfaff conv}, what remains is to show that convergence of $-N(\tau)+\chi+2S$ or $h(n,n)+\chi+2S$ after rescaling implies convergence of $-N(\tau)$ or $h(n,n)$ after rescaling. This proceeds in a similar manner to the $\rho>\frac{1}{2}$ case in Proposition \ref{prop: GSE asymp}.
\end{proof}

\section*{Acknowledgment}
The author thanks Alexei Borodin for his guidance throughout this project. The author also thanks Guillaume Barraquand for suggesting Corollary \ref{cor: contour int formula}, Matteo Mucciconi for explaining the papers \cite{IMS21,IMS22}, the authors of \cite{BBCW18} for providing the Latex code used to generate the figures in Section \ref{sec: 6vm and HL}, and Amol Aggarwal, Promit Ghosal, and Dominik Schmid for helpful discussions.

\bibliography{bibliography}{}
\bibliographystyle{abbrvurl}

\appendix
\section{Fredholm Pfaffians}
\label{sec: app}
Recall that the Pfaffian $\Pf(K)$ of a skew-symmetric matrix $K$ is defined as
\begin{equation*}
    \Pf(K)=\sum_{M}\sgn(M)\prod _{(i,j)\in M}K_{ij},
\end{equation*}
where the sum is over all perfect matchings of $\{1,\dotsc, 2n\}$, $\sgn(M)$ is the sign of the permutation given in one line notation by recursively selecting the smallest number out of $M$ along with its paired number, and $i<j$. The Pfaffian satisfies $\Pf(K)^2=\det(K)$, and so in particular is $0$ if $K$ is not of full rank.

We also require the following expansion for Pfaffians.
\begin{lemma}[{\hspace{1sp}\cite[Lemma 4.2]{S90}}]
\label{lem: pf expansion}
Let $A$ and $B$ denote $2n\times 2n$ skew symmetric matrices. Then
\begin{equation*}
    \Pf(A+B)=\sum_{\substack{J\subseteq \{1,\dotsc, 2n\}\\|J|\text{ even}}}(-1)^{\sum_{i\in J}i-|J|/2}\Pf(A_I)\Pf(B_{I^c}),
\end{equation*}
where $A_I$ denotes the submatrix given by rows and columns indexed by $I$.
\end{lemma}

Let $(\Omega,\mu)$ be a measure space. For us, $\Omega$ will either be a countable set with the counting measure or $\mathbb{R}$ with Lebesgue measure, so we will drop the dependence on $\mu$. Let $K:\Omega\times \Omega\to\R$ be a measurable $2\times 2$ skew-symmetric matrix-valued kernel
\begin{equation*}
    K(x,y)=\begin{pmatrix}
    K_{11}(x,y)&K_{12}(x,y)
    \\K_{21}(x,y)&K_{22}(x,y)
    \end{pmatrix},
\end{equation*}
and let $J$ denote the kernel defined by
\begin{equation*}
    J(x,y)=\delta_{x=y}\begin{pmatrix}
    0&1\\-1&0
    \end{pmatrix}.
\end{equation*}
We define the \emph{Fredholm Pfaffian} $\Pf(J-K)_{L^2(\Omega)}$ by the series expansion
\begin{equation*}
    \Pf(J-K)_{L^2(\Omega)}=1+\sum_{n=1}^\infty \frac{(-1)^n}{n!}\int_{\Omega^{n}}\Pf(K(x_i,x_j))_{i,j=1}^{n}d\mu^{\otimes n},
\end{equation*}
provided that the series converges absolutely, i.e.
\begin{equation*}
    \Pf(J-K)_{L^2(\Omega)}=1+\sum_{n=1}^\infty \frac{(-1)^n}{n!}\int_{\Omega^{n}}|\Pf(K(x_i,x_j))_{i,j=1}^{n}|d\mu^{\otimes n}<\infty.
\end{equation*}
In particular, we do not require that $K_{ij}(x,y)$ have any decay properties beyond convergence of this series, and indeed we will wish to work with kernels which can exponentially grow in $x$ or $y$. Note here that the convention is that the matrices appearing in the expansion consist of $n^2$ $2\times 2$ blocks rather than $4$ $n\times n$ blocks. This is related to the Fredholm determinant via $\Pf(J-K)_{L^2(\Omega)}^2=\det(1-J^TK)_{L^2(\Omega)}$.

A standard tool to prove absolute convergence of this series is given by Hadamard's bound for determinants, which immediately implies a bound on Pfaffians, since $|\Pf(X)|=\sqrt{|\det(X)|}$ for any skew-symmetric matrix $X$. For convenience, we state the needed result.

\begin{lemma}
\label{lem: Hadamard}
Let $K$ be an $n\times n$ matrix with $|K_{ij}|\leq a_ib_j$. Then
\begin{equation*}
    |\det(K)|\leq n^{\frac{n}{2}}\prod _{i}a_ib_i.
\end{equation*}
Thus, if $K$ is skew-symmetric of size $2n$, and $|K_{ij}|\leq a_ib_j$, then
\begin{equation*}
    |\Pf(K)|\leq (2n)^{\frac{n}{2}}\prod_i \sqrt{a_i b_i}.
\end{equation*}
\end{lemma}
\end{document}